\documentclass[11pt,reqno]{amsart}

\usepackage[T1]{fontenc}
\usepackage{amsmath}							
\usepackage{amssymb}
\usepackage{amsthm}
\usepackage{amscd}
\usepackage{dashrule}
\usepackage{amsfonts}
\usepackage{stmaryrd}
\usepackage{algorithm, algorithmic}
\usepackage{ wasysym }
\usepackage{caption}
\usepackage{subcaption}
\usepackage{euler}

\usepackage{extarrows}
\usepackage[colorlinks, linktocpage, citecolor = red, linkcolor = blue]{hyperref}
\usepackage{color}
\usepackage{tikz}									
\usepackage{fullpage}
\usepackage[shortlabels]{enumitem}
\usepackage{longtable}
\usepackage{array}
\newtheorem{theorem}{Theorem}[section]
\newtheorem{lemma}[theorem]{Lemma}
\newtheorem{conjecture}[theorem]{Conjecture}
\newtheorem{proposition}[theorem]{Proposition}
\newtheorem{corollary}[theorem]{Corollary}  
	
\theoremstyle{definition}
								
\newtheorem{definition}[theorem]{Definition}

\newtheorem{example}[theorem]{Example}

\newtheorem{innercustomthm}{Theorem}
\newenvironment{customthm}[1]
  {\renewcommand\theinnercustomthm{#1}\innercustomthm}
  {\endinnercustomthm}

\newtheorem{remark}[theorem]{Remark}

\newcommand{\mult}[2]{\text{mult}_{\mathbb{#1}}(#2)}

\DeclareMathOperator{\conv}{conv}

\title{Towards tropically counting binodal surfaces}

\author[M. Brandt]{Madeline Brandt}
\address{Department of Mathematics, Brown University,  Providence, RI 02912}
\email{\href{mailto:madeline_brandt@brown.edu}{madeline\_brandt@brown.edu}}
\urladdr{\url{https://sites.google.com/view/madelinebrandt}}
\author{Alheydis Geiger}
\address{Department of Mathematics, University of T\"{u}bingen, Germany}
\email{\href{mailto:alheydis.geiger@math.uni-tuebingen.de}{alheydis.geiger@math.uni-tuebingen.de}}
\urladdr{\url{https://www.math.uni-tuebingen.de/en/research-chairs/geometry/team/hannah-markwig/team/alheydis-geiger}}

\begin{document}

\begin{abstract}
We count tropical multinodal surfaces using \emph{floor plans}, looking at the case when two nodes are tropically close together, i.e., \emph{unseparated}.
    We generalize tropical floor plans to recover the count of multinodal curves. 
   We then prove that for $\delta=2$ or $3$ nodes, tropical surfaces with unseparated nodes contribute asymptotically to the second order term of the polynomial giving the degree of the family of complex projective surfaces in $\mathbb{P}^3$ of degree $d$ with $\delta$ nodes.
   We classify when two nodes in a surface tropicalize to a vertex dual to a polytope with 6 lattice points, and prove that this only happens for projective degree $d$ surfaces satisfying point conditions in Mikhalkin position when $d>4$.
\end{abstract}

\maketitle

\setcounter{tocdepth}{2}
\tableofcontents

\section{Introduction}
The enumeration of a fixed type of variety satisfying certain properties is the subject of many longstanding, well-known, and popular problems in enumerative geometry.
Gromov-Witten invariants counting the number of nodal plane  curves are a famous example. Tropical geometry can be used as a combinatorial and polyhedral tool to solve algebraic enumeration questions.

Mikhalkin pioneered the use of tropical geometry to answer  enumerative questions \cite{Mik}. Tropical methods have successfully computed the Gromov-Witten invariants of curves
\cite{Mik,BM09}. 
Also, tropical geometry provides a rich environment for counting plane curves, and many different combinatorial tools have been developed \cite{GM07,BrMi2007,BM09}.
Moreover, in many instances, tropical methods give an approach for counting problems over $\mathbb{R}$ that are difficult to solve otherwise. For example, it was proven by means of tropical geometry that the Welschinger invariants, a signed count of rational plane curves over the real numbers, are always positive and that they are logarithmic asymptotically equal to the Gromov-Witten invariants, the analogous numbers over $\mathbb{C}$ \cite{IKS03}.

Tropical geometry  has only recently been applied to counting singular surfaces \cite{MaMaSh12,MaMaSh18,MaMaShSh19,Si20}.
Classically, it is known that the family of complex projective surfaces in $\mathbb{P}^3$ of degree $d$ with $\delta$ distinct nodes as their only singularities is of degree 
\begin{equation}
 N_{\delta, \mathbb{C}}^{\mathbb{P}^3}(d) = \frac{4^\delta d^{3 \delta}}{\delta !} - \frac{3 \cdot 4^\delta}{\delta!}d^{3\delta - 1} + \mathcal{O}(d^{3 \delta - 2}).  
 \label{maineq}
\end{equation}

See \cite{MaMaShSh19} for more details. This space has dimension $\binom{d+3}{3}-\delta-1$, so the above polynomial also counts the number of $\delta$-nodal surfaces of degree $d$ satisfying $\binom{d+3}{3}-\delta-1$ generic point conditions.
In \cite{MaMaShSh19} the authors introduce \emph{tropical floor plans} as a tool for counting tropical surfaces passing through points in Mikhalkin position where the node germs are at least one floor apart, and show that these account for $\frac{4^\delta d^{3 \delta}}{\delta !} + \mathcal{O}(d^{3 \delta - 1})$ surfaces, thus recovering the first term of the complex count asymptotically. They also show that in the plane curve case, floor plans easily count curves with one node.
In \cite{BG20} we studied tropical surfaces passing through points in Mikhalkin position with \emph{separated} nodes. These are tropical surfaces where the topological closures of the cells containing the tropicalizations of the nodes have empty intersection. 

Using these, we obtain $214$ of the $280$ complex binodal cubic surfaces passing through $17$ general points coming from tropical binodal surfaces through points in Mikhalkin position. A surface is \emph{binodal} if it has two distinct nodes. The remaining unknown surfaces therefore contain  \emph{unseparated} nodes. Here, the two node germs (see Definition~\ref{def:nodegerms}) are close together, so the cells that would normally contain the nodes interact and their topological closures intersect. 
Understanding surfaces with unseparated nodes is a necessary step towards counting with floor plans.
A deepened understanding of how to use floor plans to count multinodal surfaces could lead to a real count of degree $d$ surfaces  or other toric surfaces in $\mathbb{P}^3$.

In this paper, we move towards generalizing the technique of counting with tropical floor plans to the multinodal case. We do this in three ways: by counting tropical multinodal curves with floor plans, by giving a lower bound to the asymptotic contribution of unseparated nodes, 
and by searching systematically for the missing surfaces. The paper is structured around these three generalizations.

First, in Section \ref{sec:curves},  we define tropical floor plans for tropical plane curves with $\delta$ nodes without restrictions on their positions.  In Theorem \ref{mainthm1} we show that the number of tropical floor plans of degree $d$ plane curves with $\delta$ nodes counted with multiplicity recovers the Gromov-Witten numbers, i.e., the numbers of the genus $g=\tfrac{(d-1)(d-2)}{2}-\delta$ degree $d$ curves with $\delta$ nodes passing through $3d-1+g$ general points.
That floor plans are successful in counting multinodal plane curves motivates their use in higher dimensional counting problems.
We illustrate the use of tropical floor plans in Example \ref{ex:count} by counting the $225$ binodal plane curves of degree $4$ passing through $12$ points.

Second, in Section \ref{sec:asymptotics}, we study the asymptotics of the techniques used in \cite{BG20}. We show that for small $\delta$, relaxing the notion of a floor plan as defined in \cite{MaMaShSh19} to allow nodes in the same or in adjacent floors is not enough to produce the second order coefficient in Equation \eqref{maineq}.

\begin{customthm}{\ref{thm:asymptotic}}
The number of $\delta$-nodal surfaces of degree $d$  passing through points in Mikhalkin position such that their tropicalization has unseparated nodes is at least
$$
\frac{48}{5} d^5 + \mathcal{O}(d^4) >0\,\,\, \text{ for }\,\,\, \delta = 2, \ \ \ \text{  and   } \ \,\,\, \frac{221}{35} d^8 + \mathcal{O}(d^7) >0\,\,\, \text{ for }\,\,\, \delta = 3.
$$
\end{customthm}
Therefore, surfaces with separated nodes are insufficient to count binodal and trinodal surfaces asymptotically up to two degrees.
So, the techniques of \cite{BG20} are insufficient to improve on the asymptotic result of \cite{MaMaShSh19}. Therefore, we conclude that a significant portion of multinodal tropical surfaces contain unseparated nodes. This motivates an investigation of surfaces containing unseparated nodes. In the Newton subdivision of such a surface, one will see (subdivisions of) \emph{multinodal polytopes}. 

Third, in Section \ref{sec:binodalpolytopes}, we embark on a detailed study of \emph{binodal polytopes}. These polytopes are the Newton polytope of a binodal surface, and may appear in the dual subdivision of degree~$d$ multinodal surfaces. 
These polytopes are a source of unseparated nodes.
Such polytopes must contain at least 6 lattice points (Lemma \ref{lem:6latticepoints}). In order to appear in a floor decomposed surface, they must have width 1. Therefore, we naturally study polytopes with 6 lattice points and width 1 as a starting place. There are several infinite families of such polytopes, and they are classified in \cite{BS15}. We first determine which of these are binodal in Section \ref{subsec:6points}. Then, we
determine the 
lattice paths that can pass through these polytopes in \ref{subsec:paths}. Finally, we conjecture how they contribute to counts of degree $d$ surfaces in Theorem \ref{thm:binodal}, Conjecture \ref{conj:mult} and Theorem \ref{conj:main}.

\begin{customthm}{\ref{thm:binodal}}
Let $X$ be a binodal surface of degree $d$ passing through points in Mikhalkin position. The two nodes of the surface $X$ can only tropicalize to a vertex dual to a polytope with 6 lattice points if $d >4.$ In this case the two nodes are located at the vertex dual to a polytope of family 10, 13 or 20, see Figure~\ref{fig:polys}.
\end{customthm}

\begin{figure}[h]
    \centering
    \begin{tabular}{c|c|c}
    Family 10 & Family 13 & Family 20\\
 \includegraphics[height=0.08\textwidth]{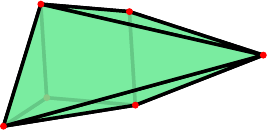}
			&\includegraphics[height=0.10\textwidth]{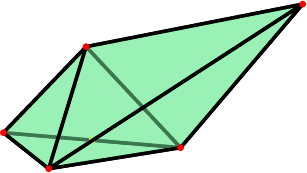}	 &  \includegraphics[height=0.13\textwidth]{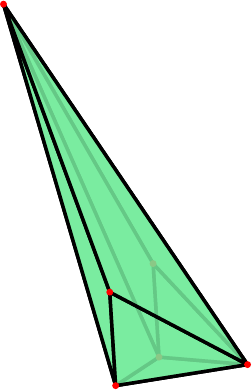} \\
			$\begin{pmatrix}
				0 & 0 & 0 & 0 & 1 & 1\\
				0 & 0 & 1 & 1 & 0 & a\\
				0 & 1 & 0 & 1 & 0 & 1
			\end{pmatrix}$  &
				$\begin{pmatrix}
				0 & 0 & 0 & 0 & 1 & 1\\
				0 & 1 & 1 & 2 & 1 & a+1\\
				0 & 0 & 1 & 0 & 0 & 1
			\end{pmatrix}$
			
			& $\begin{pmatrix}
				0 & 0 & 0 & 1 & 1 & 1\\
				1 & 1 & 2 & 0 & 1 & 1\\
				0 & 1 & 0 & a & 0 & 1
			\end{pmatrix}$ 
			\\
			$a\geq 3$
			& $a\geq 4$
			&$a>3$
    \end{tabular}
    \caption{The three families of polytopes with $6$ lattice points encoding two nodes at the same vertex for degree $d$ binodal surfaces through points in Mikhalkin position.
    }\label{fig:polys}
    \label{fig:my_label}
\end{figure}

In Conjecture \ref{conj:mult}, we make a postulation about the multiplicities each surface would need to be counted with (see Table \ref{tab:Multiplicites}). 
 Assuming Conjecture~\ref{conj:mult}, we determine the asymptotic contribution to counts of binodal degree $d$ surfaces taking these polytopes in to account.

\begin{customthm}{\ref{conj:main}}
 Assuming Conjecture \ref{conj:mult} is true, it follows that tropical degree $d$ surfaces with a binodal polytope with $6$ vertices and width $1$ in the dual subdivision contribute $\frac{1}{4}d^4+\mathcal{O}(d^3)$ surfaces to the count of binodal degree $d$ surfaces $N_{2,\mathbb{C}}^{\mathbb{P}^3}(d)$. So, they contribute to the third-highest term of the polynomial  $N_{2,\mathbb{C}}^{\mathbb{P}^3}(d)$.
\end{customthm}

Throughout the paper, we carry out computations in \texttt{Singular} \cite{singular}. Code in \texttt{OSCAR} \cite{oscar} for these computations is available, see \cite{Code21}. These computations are explained in more detail in \cite{G22}. 
We conclude with remaining open directions for multinodal surface counting.

\medskip

\textbf{Acknowledgements.} The authors thank the anonymous referees for their thoughtful and helpful comments which improved this article considerably. The authors thank Hannah Markwig for her explanations, the insightful discussions and her feedback.
Thanks also to Janko B\"{o}hm for suggesting interesting questions to this project and to Eugenii Shustin. We thank Tommy Hofmann for his assistance with \texttt{OSCAR}. This work is a contribution to the SFB-TRR 195 'Symbolic Tools in Mathematics and their Application' of the German Research Foundation (DFG).
AG was funded by a PhD scholarship from the Cusanuswerk e.V..
MB is supported by the National Science Foundation under Award No. 2001739. Any opinions, findings, and conclusions or recommendations expressed in this material are those of the author(s) and do not necessarily reflect the views of the National Science Foundation.

\section{Preliminaries}\label{sec:preliminaries}

Let $\mathbb{K} = \cup_{m\geq 1} \mathbb{C}((t^{1/m}))$ be the field of complex Puiseux series. We assume familiarity with tropicalization of hypersurfaces and the corresponding dual subdivision of the Newton polytope as in \cite[Chapter 3.1]{tropicalbook}. Throughout, we use the $\max$-convention.
We say a tropical surface in $\mathbb{R}^3$ is of \emph{degree $d$} if its Newton polytope is $d\Delta_3 = \mathrm{conv}(\{(d,0,0),(0,d,0),(0,0,d),(0,0,0)\})$.

Tropical techniques for enumerating surfaces passing through points reduce to the case where the points are assumed to be in \emph{Mikhalkin position}. These are points in $\mathbb{K}^3$ that are in general position such that their tropicalization in $\mathbb{R}^3$ is in tropical general position. Additionally, the combinatorics of the resulting tropical surfaces is well-understood. For $\binom{d+3}{3}-\delta-1$ tropical points in $\mathbb{R}^3$ \emph{tropical general position} means, that they allow only finitely many tropical surfaces of degree $d$ with $\delta$ nodes of maximal geometric type \cite{MaMaSh18} to pass through them.
 
 \begin{definition}[{\cite[Section 3.1]{MaMaSh18}}]\label{def:pointconfig}
For a point configuration $\omega=(p_1,...,p_{n})$ of $n$ points in $\mathbb{K}^3$, we denote by $q_i\in \mathbb{R}^3$ the tropicalization of $p_i$ for $i=1,...,n$. We say $\omega$ is in \emph{Mikhalkin position} if the $q_i$ are distributed with growing distances along a line $\{\lambda\cdot (1,\eta,\eta^2)|\lambda \in \mathbb{R} \}\subset\mathbb{R}^3$, where $0<\eta\ll1,$ and the $p_i$ are in generic position. 
\end{definition}
Note, that this implies, that $q_1,\,\ldots,\,q_n$ are in tropically general position \cite{MaMaSh18}.
We say a tropical surface passes through points in Mikhalkin position if it passes through the $q_i$. From now on, we assume that the points that our surfaces pass through are in Mikhalkin position. 
For a  tropical surface passing through points in Mikhalkin position, we know that the points are contained in the interior of $2$-dimensional cells of the surface \cite[Section 2.1]{MaMaSh18}. As $2$-dimensional cells correspond to edges in the dual subdivision of the Newton polytope $\Delta$, these point conditions lead to a path of the lattice points of $\Delta$ as follows. Let $v= (1,\eta,\eta^2)$ with $0<\eta\ll1$.
\begin{definition}[{\cite[Section 3.2]{MaMaSh18}}]\label{def:partialorder}
Consider the partial order $>$ in $\mathbb{R}^3$ where $u > u'\, \Leftrightarrow \langle u - u', v\rangle>0.$ Order the points of $\Delta \cap\mathbb{Z}^3$ according to $>$:
\begin{equation*}
    \Delta \cap\mathbb{Z}^3 = \{w_0,..., w_{N+1}\},\,\,  w_i < w_{i+1} \text{ for all } i = 0,..., N.
\end{equation*}
Note, that since $1\gg\eta$, we can assume that $\tfrac{1}{\eta}>\max_{(x,y,z)\in\Delta\cap\mathbb{Z}^3}\{|x+y+z| \}$, so that the partial order $>$ restricted to $\Delta\cap\mathbb{Z}^3$ is a total order.

Given a subset $A\subset \Delta \cap\mathbb{Z}^3$ consisting of $m \geq 2$ points $a_1 < a_2 < \cdots < a_m$, we call an ordered subset of the set of
segments $P(A) := \{[a_i , a_{i+1}] : i = 1,..., m - 1\}$ a \emph{lattice path} supported on $A$ if it covers the whole set $A$. Here, $[a_i , a_{i+1}]$ denotes the line segment from $a_i$ to $a_{i+1}$. The set $P(A)$ is called the \emph{complete lattice path} supported on $A$. We call a lattice path \emph{connected (disconnected)} if
the union of its segments is connected (disconnected).
\end{definition}

For examples of lattice paths, consider Example \ref{ex:binodalcount1}.

Given a smooth or nodal tropical surface with Newton polytope $\Delta$ passing through points in Mikhalkin position, the set of edges in the dual subdivision of $\Delta$ corresponding to 2-dimensional cells containing one of the points forms a lattice path as above.
In the smooth case, we can use the lattice path to construct the dual subdivision of the Newton polytope by the smooth extension algorithm \cite[Section 3.3]{MaMaSh18}.

If we wish to count surfaces with singularities, we have fewer point conditions, which in turn corresponds to leaving out lattice points from the lattice path or to leaving out steps, i.e., segments, from the lattice path in the case of disconnected lattice paths.
While a smooth tropical surface has only unimodular tetrahedra in the dual subdivision, tropicalizations of singular surfaces contain other polyhedra in their dual subdivisions.
 
\begin{figure}
    \centering
    \begin{subfigure}[b]{0.18\textwidth}
    \centering
        \includegraphics[height=0.1\textheight]{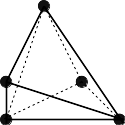}
        \caption*{circuit A}\label{fig:circuitA}
    \end{subfigure}
    \begin{subfigure}[b]{0.18\textwidth}
    \centering
        \includegraphics[height=0.09\textheight]{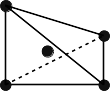}
        \caption*{circuit B}\label{fig:circuitB}
    \end{subfigure}
    \begin{subfigure}[b]{0.18\textwidth}
    \centering
        \includegraphics[height=0.08\textheight]{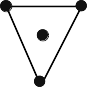}
        \caption*{circuit C}\label{fig:circuitC}
    \end{subfigure}
    \begin{subfigure}[b]{0.18\textwidth}
    \centering
        \includegraphics[height=0.08\textheight]{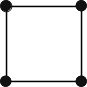}
        \caption*{circuit D}\label{fig:circuitD}
    \end{subfigure}
    \begin{subfigure}[b]{0.18\textwidth}
    \centering
        \includegraphics[height=0.1\textheight]{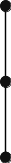}
        \caption*{circuit E}\label{fig:circuitE}
    \end{subfigure}
    \caption{Circuits that encode a single node inside a dual subdivision.}
    \label{fig:circuits}
\end{figure}

For one-nodal tropical surfaces, the possibilities have been fully classified in \cite{MaMaSh12} as follows. If a tropical surface contains exactly one node, then its dual subdivision contains exactly one of the circuits depicted in Figure \ref{fig:circuits}.
For circuit A and B, the tropical singularity is the vertex dual to the polytope. Circuit C encodes a singularity in the surface if it is the base triangle of either one or two tetrahedra in the subdivision. In the first case, the circuit is dual to an unbounded edge that contains the singularity at a distance from the vertex depending on the six coefficients associated to the lattice points of the tetrahedron containing the circuit C. In the second case, the circuit is dual to a bounded edge that contains the singularity either in its midpoint or at the ratio $3:1$. Circuit D, however, only encodes a singularity if it is the base parallelogram for two pyramids. The singularity is then the midpoint of the edge dual to the circuit.
Circuit E encodes a singularity if it has at least three neighboring points in one of the positions specified in \cite{MaMaSh12} forming at least two tetrahedra with the edge. The weighted barycenter of the 2-cell dual to the circuit is the node, where the weight is given by the choice of the three neighbors.

When the surface has multiple nodes, it is not yet well understood what the dual subdivision of the tropical surface can look like. One possibility is that the nodes are \emph{separated}.

\begin{definition}[\cite{BG20}]
We say two nodes are \emph{separated} if their tropicalizations are contained in cells dual to the one-nodal polytope complexes described above, and they intersect at most in a unimodular face.
\end{definition}

Two singularities can tropicalize to points on the tropical surface that are closer together. For example, this occurs for binodal tropical cubic surfaces satisfying point conditions in Mikhalkin position, see \cite{BG20}.
Unseparated nodes are encoded in the dual subdivision via larger polytope complexes, which have not been classified. In Section \ref{sec:binodalpolytopes} we see how polytopes with $6$ lattice points can encode two nodes that are tropically at the same point and coincide with a vertex of the surface.

When choosing points in Mikhalkin position and a lattice path through the Newton Polytope, we expect to obtain a sliced subdivision corresponding to a floor decomposed surface \cite{BeBrLo17}. This decomposition of the surface is used in the description of the surface via tropical floor plans. Tropical floor plans were first introduced in \cite{MaMaShSh19} and generalized in \cite{BG20}. As stated in \cite[Proposition 5.9]{MaMaShSh19}, a floor plan defines a unique tropical surface passing through points in Mikhalkin position.

As far as we know, there is no proof that multinodal surfaces passing through points in Mikhalkin position are always floor decomposed, neither have we come across a counterexample. We make the additional assumption that our surfaces will be floor decomposed, in order to be able to use tropical floor plans.

Floor plans are a tool for viewing the surface as a sequence of tropical plane curves. These curves can encode nodes of the surface. The structures appearing in a curve that could produce a node are called \emph{node germs}. 
\begin{definition}[\cite{MaMaShSh19}, Definition 5.1]\label{def:nodegerms}
Let $C$ be a plane tropical curve of degree $d$ passing through $\binom{d+2}{2}-2$ points in general position. A \textit{node germ} of $C$ is one of the following:
\begin{enumerate}
    \item  a vertex  
    dual to a parallelogram in its Newton subdivision,
    \item[(1')]  a midpoint of an edge of weight two
    dual to a pair of adjacent triangles of area 2 each sharing an edge of weight two in its Newton subdivision,
\item a horizontal or diagonal end (ray) of weight two,
\item a right or left \emph{string} (see below).
\end{enumerate}
If none of the points is contained in an end adjacent to the lower right or left vertex of the Newton polytope, then we can prolong the adjacent bounded edge in direction $(1,0)$ or $(-1,-1)$, respectively. This produces a family of curves, all of which pass through the points. The union of the two ends is called a \emph{right} or \emph{left} \emph{string} depending of the direction of the degree of freedom.
 See Figures \ref{fig:lef_string} and \ref{fig:right_string}. 
\begin{figure}[h]
    \centering
     \begin{subfigure}{.32\textwidth}
      \includegraphics[height=0.4\textwidth]{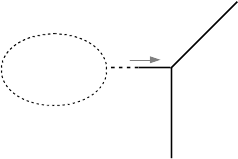}
      \caption{Right string}
      \label{fig:right_string}
    \end{subfigure}
        \begin{subfigure}{.32\textwidth}
      \includegraphics[height=0.4\textwidth]{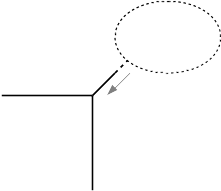}
      \caption{Left string}
      \label{fig:lef_string}
    \end{subfigure}{}
    \caption{Right and left strings.}
    \label{fig:strings}
\end{figure}{}
\end{definition}

We use the following definition of floor plans from \cite{BG20}, which is generalized from \cite{MaMaShSh19}. In \cite{MaMaShSh19}, floor plans required node germs to be in floors that are separated by a smooth floor. Here, we allow node germs in adjacent or the same floors.

\begin{definition}[\cite{BG20}, generalized from \cite{MaMaShSh19}]\label{def:floorplan}
		Let $Q_i$ be the projection of $q_i$ along the $x$-axis, into the $yz$-plane. A \emph{$\delta$-nodal floor plan $F$ of degree $d$} is a tuple $(C_d,\ldots,C_1)$ of plane tropical curves $C_i$ of degree $i$ called \emph{floors} together with a choice of indices $d\geq i_{\delta'}>\ldots>  i_1\geq1$ each assigned a natural number $k_j$ 
		such that $\sum_{j=1}^{\delta'}k_j=\delta$, where $0<\delta'\leq\delta$,  satisfying:
	\begin{enumerate}
		\item The curve $C_i$ passes through the following points: (we set $i_0=0, i_{\delta+1}=d+1.$)\begin{align*}
			\text{if } i_{\nu}>i>i_{\nu-1}:\,\,\,\,\,\,\,\, &Q_{\sum_{k=i+1}^d \binom{k+2}{2}-\delta+(\sum_{j: i > i_j} k_j)+1},...,Q_{\sum_{k=i}^d \binom{k+2}{2}-\delta+(\sum_{j: i > i_j} k_j)-1}, \\
			\text{if } i=i_{\nu}:\,\,\,\,\,\,\,\,\,\,\,\, &Q_{\sum_{k=i+1}^d \binom{k+2}{2}-\delta+(\sum_{j: i \geq i_j} k_j)+1},...,Q_{\sum_{k=i}^d \binom{k+2}{2}-\delta+(\sum_{j: i > i_j} k_j)-1}.
		\end{align*}
		\item The plane curve $C_{i_j}$ has $k_j$ node germs for every $j=1,\ldots,\delta'.$
		\item If  $C_{i_j}$  contains a left string as a node germ, then its horizontal end aligns either with a horizontal bounded edge of $C_{i_j+1}$ or with a $3$-valent vertex (where edges are counted with multiplicity) of $C_{i+1}$  not adjacent to a horizontal edge.
		\item  If $C_{i_j}$ contains a right string as a node germ, then its diagonal end aligns either with a diagonal bounded edge of $C_{i_j-1}$ or with a $3$-valent vertex (where edges are counted with multiplicity) of $C_{i_j-1}$ not adjacent to a diagonal edge.
		\item If $i_{d}=\delta'$, then the node germs of $C_d$ can only be diagonal ends of weight two or a right string.
		\item If $i_1=1$, then the node germ of $C_1$ is a left string.
	\end{enumerate}
\end{definition}

We use the term \emph{align} in connection with left and right strings  to mean that the horizontal ray of the left string or diagonal ray of the right string contains a given vertex or bounded edge. For examples of binodal floor plans, see the proof of the main theorem of \cite{BG20}.

We now describe how to compute the multiplicity of a floor plan.
This enables us to translate the problem of counting tropical surfaces through points to counting floor plans. This leverages the structure imposed by choosing the points to be in Mikhalkin position, since we now count curves (which are well-understood) and the ways they can interact to produce nodes (which is outlined by Definition \ref{def:floorplan}).
\begin{definition}[Generalised from Definition 5.4, \cite{MaMaShSh19}]\label{def:multiplicities}
Let $F$ be a $\delta$-nodal floor plan of degree $d$. Let $C^{*}_{i_j}$ be a node germ of $C_{i_j}$, we
define the following local complex multiplicity $\text{mult}_{\mathbb{C}}(C^{*}_{i_j})$:

\begin{enumerate}
\item If $C^{*}_{i_j}$ is dual to a parallelogram, then $\text{mult}_{\mathbb{C}}(C^{*}_{i_j}) =2$.
\item If $C^{*}_{i_j}$ is the midpoint of an edge of weight two, , then $\text{mult}_{\mathbb{C}}(C^{*}_{i_j}) =8$.
\item If  $C^{*}_{i_j}$ is a horizontal end of weight two, then $\text{mult}_{\mathbb{C}}(C^{*}_{i_j}) = 2(i_j + 1)$.
\item  If $C^{*}_{i_j}$ is a diagonal end of weight two, then $\text{mult}_{\mathbb{C}}(C^{*}_{i_j}) = 2(i_j - 1)$.
\item If $C^{*}_{i_j}$ is a left string whose horizontal end aligns with a horizontal bounded edge, then $\text{mult}_{\mathbb{C}}(C^{*}_{i_j}) = 2$.
\item If $C^{*}_{i_j}$ is a left string whose horizontal end aligns with a vertex not adjacent to a horizontal edge, then $\text{mult}_{\mathbb{C}}(C^{*}_{i_j}) = 1$.
\item  If $C^{*}_{i_j}$ is a right string whose diagonal end aligns with a diagonal bounded edge, then $\text{mult}_{\mathbb{C}}(C^{*}_{i_j}) = 2$.
\item If $C^{*}_{i_j}$ is a right string whose diagonal end aligns with a vertex not adjacent to a diagonal edge, then $\text{mult}_{\mathbb{C}}(C^{*}_{i_j})= 1$.

\end{enumerate}
The multiplicity of a $\delta$-nodal floor plan  $F$ is $\text{mult}_{\mathbb{C}}(F) = 
\prod_{j=1}^{\delta'}\prod_{\text{node germs of } C_{i_j}}\text{mult}_{\mathbb{C}}(C^{*}_{i_j}).$

\end{definition}

For examples of how to compute multiplicities of floor plans, see \cite{BG20}.

\begin{remark}
All pentatopes appearing in the dual subdivisions of degree $d$ surfaces passing through points in Mikhalkin position have multiplicity $1$, see \cite[Lemma 5.5]{MaMaSh18}. This is why the node germs that give rise to a pentatope in the subdivision (right and left string aligning with a vertex) are counted with multiplicity one for tropical floor plans of degree $d$. For more general surfaces, it is possible that pentatopes with higher multiplicity occur. In these cases, the multiplicity has to be determined from the point conditions and the discriminant. See \cite{MaMaSh18} for a more detailed explanation.
\end{remark}

\begin{example}\label{ex:binodalcount1} Floor plans can easily be generalized for surfaces that are not of degree $d$, i.e., their Newton polytope is not $d\Delta_3$. We demonstrate this by counting the binodal surfaces with Newton polytope $P$ (pictured below) passing through five general points.
   \begin{figure}[h]
       \centering
       \begin{subfigure}{0.45\textwidth}
       \centering
           \begin{align*}
     P=\conv\begin{pmatrix}
      0 & 0 & 0 & 0 & 1 & 1 & 1 & 1\\
      0 & 1 & 1 & 2 & 0 & 1 & 1 & 2\\
      1 & 0 & 1 & 1 & 0 & 0 & 1 & 0
     \end{pmatrix}
     \end{align*}  
       \end{subfigure}
       \begin{subfigure}{0.25\textwidth}
       \centering
           \includegraphics[width = 0.75\textwidth]{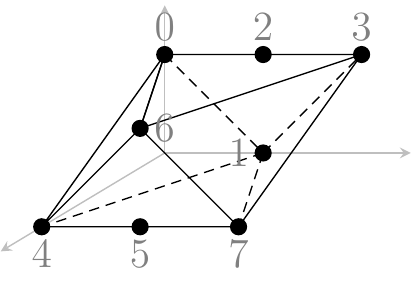}
       \end{subfigure}
   \end{figure}
   \addtocounter{figure}{-1}  
   
Since we are counting binodal surfaces, we look for lattice paths that go through six of the eight lattice points of $P$. Additionally such a lattice path must induce a subdivision dual to a binodal tropical surface. By exhaustive search, we find four possible lattice paths passing through six lattice points of $P$. These are pictured in red on the left in Figures \ref{fig:pentatopes}, \ref{fig:bipyramids} and \ref{fig:weights2}.  
 They give rise to five different floor plans due to different alignment options. The advantage of tropical floor plans is that they break down the dimension by $1$, so our tropical floor plans are curves in the plane. The floor plans corresponding to the four possible lattice paths are depicted on the right in Figures \ref{fig:pentatopes}, \ref{fig:bipyramids} and \ref{fig:weights2}. The green curve corresponds to the subdivision in the $x=0$ facet of $P$, the blue curve to the subdivision in the $x=1$ facet of $P$. We display the point conditions the curves satisfy by the black points.
 
 These floor plans give rise to separated nodes, so we can count them with their multiplicities from \cite{MaMaSh18}.
For the two floor plans in Figure \ref{fig:pentatopes}, we investigate the discriminant of the two pentatopes occurring. This computation yields that in both cases one pentatope has to be counted with multiplicity $2$ and the other with multiplicity $1$, so in total we count $2\cdot 1 + 2\cdot 1 = 4$ surfaces coming from these floor plans.

In Figure \ref{fig:bipyramids}, we see two lattice paths producing the same type of alignment. The two floor plans look nearly the same, only that the points of the point conditions are distributed differently on the two curves. The two alignments give each rise to a bipyramid in the subdivision, which has multiplicity $2$.  
So, we count $2\cdot(2\cdot 2)=8$ surfaces.

The floor plan in Figure \ref{fig:weights2} has multiplicity $2\cdot 2 =4$ as follows from \cite{MaMaSh18}.

In total, we obtain $16$ surfaces from the count of the tropical floor plans. This coincides with the degree of the binodal variety of surfaces with Newton polytope $P$, which can be computed using \texttt{OSCAR}  \cite{oscar,Code21}.

\begin{figure}[h]
    \centering
        \begin{subfigure}{0.2\textwidth}
        \includegraphics[width=1\textwidth]{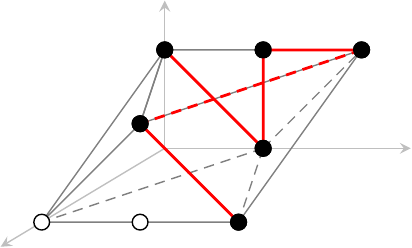}
    \end{subfigure}
    \begin{subfigure}{0.35\textwidth}
        \includegraphics[width=0.8\textwidth]{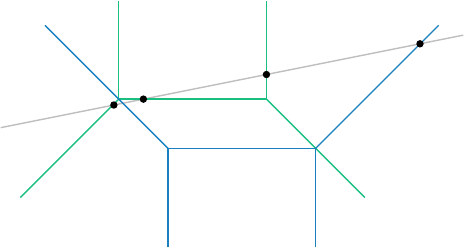}
    \end{subfigure} \begin{subfigure}{0.35\textwidth}
        \includegraphics[width=0.8\textwidth]{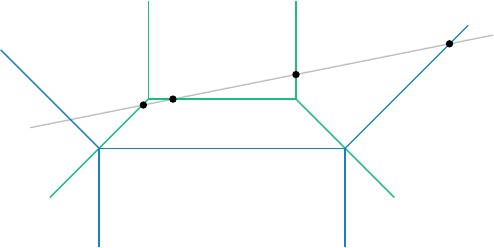}
    \end{subfigure}
        \caption{This lattice path allows for two alignments. Both correspond to subdivisions of the polytope that contain two pentatopes.}
    \label{fig:pentatopes}
 \end{figure}
 
 \begin{figure}[h]
    \centering
    \begin{subfigure}{0.2\textwidth}
        \includegraphics[width=1\textwidth]{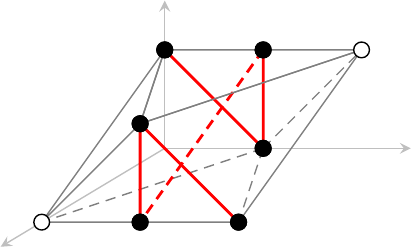}
    \end{subfigure}\quad
 \begin{subfigure}{0.4\textwidth}
        \includegraphics[width=0.9\textwidth]{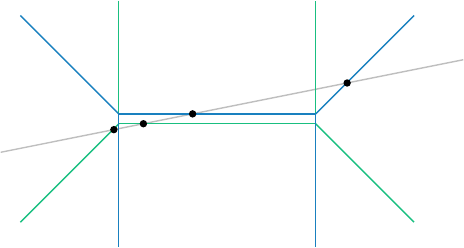}
    \end{subfigure}

\vspace{0.25cm}
    \begin{subfigure}{0.2\textwidth}
        \includegraphics[width=1\textwidth]{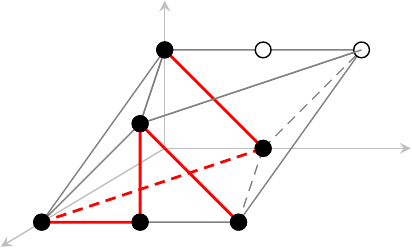}
    \end{subfigure}\quad \begin{subfigure}{0.4\textwidth}
        \includegraphics[width=0.9\textwidth]{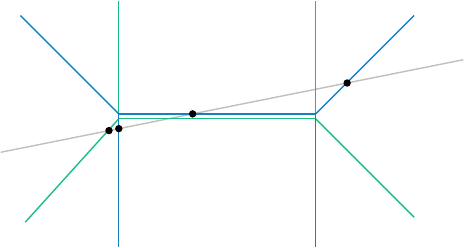}
    \end{subfigure}
        \caption{There are two lattice paths that allow an alignment that corresponds to two bipyramids in the dual subdivision.}
    \label{fig:bipyramids}
    \end{figure}
 \begin{figure}[h]
    \centering   
    \begin{subfigure}{0.3\textwidth}
    \centering
        \includegraphics[width=0.66\textwidth]{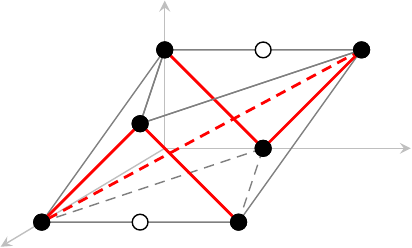}
    \end{subfigure}
    \begin{subfigure}{0.4\textwidth}
        \includegraphics[width=0.7\textwidth]{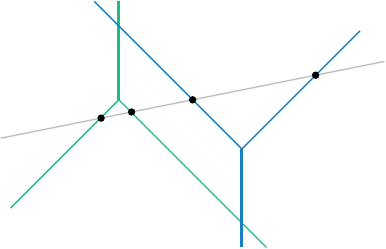}
    \end{subfigure}
    \caption{There is one lattice path for which the two nodes are coming from intersections of rays with weight two ends.}
    \label{fig:weights2}
\end{figure}

This example shows that for some surfaces the tropicalization only contains separated nodes, so that a complete count can be achieved using only the tools of separated nodes.

\end{example}

\section{Floor plans of plane curves}\label{sec:curves}

In this section we count multinodal plane curves by generalizing the definition of tropical floor plans for one-nodal plane curves from \cite{MaMaShSh19} to the multinodal case. We show that tropical floor plans of multinodal plane curves can be used as a counting tool to recover the Gromov-Witten numbers. The Gromov-Witten numbers give the number $N_{d,g}$ of plane curves of degree $d$ and genus $g$ passing through $3d-1+g$ points in general position. By the genericity of the points, all the curves counted are then nodal. By the degree-genus formula, each curve has $\delta = \tfrac{(d-1)(d-2)}{2}-g$ nodes. 

The following definition is inspired by the definition of vertically stretched points in \cite[Definition 3.4]{FM09}. 
\begin{definition}
A collection of finitely many points $(q_0,\ldots,q_N)$ in $\mathbb{R}^2$, $q_i =(x_i,y_i)$, is said to be in \emph{horizontally stretched position}, if they satisfy
\begin{align*}
    &x_0\ll x_1 \ll.\ldots \ll x_N, \\
    &y_0 <y_1 <\ldots <y_N, \\
&\max_{i\neq j} |y_i - y_j | \in (0,\epsilon), \text{ for } \epsilon \text{ small.}
\end{align*}
\end{definition} 

For counts of curves of degree $d$ with $\delta$ nodes in $\mathbb{P}^2$, we pick a horizontally stretched point configuration $q_1,...,q_n$ of $n= \binom{d+2}{2}-(\delta+1)$ points in $\mathbb{R}^2$. We denote by $Q_i$ the projection of $q_i$ to its second coordinate.

\begin{definition}
A \emph{tropical floor plan} for a tropical plane curve of degree $d$ with $\delta$ nodes consists of a tuple of tropical divisors $(D_d,D_{d-1},\ldots,D_1)$  on $\mathbb{R}$ where each $D_i$ is of degree $i$, together with a choice of indices $d\geq i_{\delta'} > ... > i_1\geq 1$ each assigned a natural number $k_j\leq i_j$ such that $\sum_{j=1}^{\delta'}k_j = \delta$, where $0<\delta'\leq\delta.$ 
The divisor $D_i$ is supported on the following points in $\mathbb{R}$, where we set $i_0=0, i_{\delta+1}=d+1$:\medskip
	
	\begin{tabular}{rl}
	 $i_j > i > i_{j-1}$ :&  $Q_{\binom{d+2}{2}-\binom{i+2}{2}-\delta+ \sum_{l=1}^{j-1}k_l+1},\ldots,Q_{\binom{d+2}{2}-\binom{i+1}{2}-\delta+\sum_{l=1}^{j-1}k_l-1}, $\\
	 $i=i_{j}$:& $Q_{\binom{d+2}{2}-\binom{i+2}{2}-\delta+ \sum_{l=1}^{j}k_l+1},\ldots,Q_{\binom{d+2}{2}-\binom{i+1}{2}-\delta+\sum_{l=1}^{j-1}k_l-1} $.\\
	\end{tabular}

\medskip

Furthermore, a divisor $D_{i_j}$ may have points of weight $w$, where $2\leq w \leq \delta+1$, if $i_j\notin\{1,d\}.$ The total number of points with higher weight is $\leq \delta$.
If $D_{i_j}$ contains $r$ points of weight $\geq 2$, then $k_j-r$ points of weight $1$ in $D_{i_j}$ have to align with a point of $D_{i_j-1}$ or of $D_{i_j+1}$.
\end{definition}

We give examples of tropical floor plans for tropical plane curves of degree 4 with two nodes in Example \ref{ex:count}.

\begin{definition}
We define the complex multiplicity $\mult{C}{F}$ of a floor plan $F$ with $r$ points of weights $w_1,...,w_r$ $\geq 2$ to be $\mult{C}{F}=\prod_{j=1}^r w_j^2 $.
\end{definition}
This extends the definition of tropical floor plans for one-nodal curves from \cite{MaMaShSh19}.

\begin{remark} For the divisors in a tropical floor plan of a tropical plane curve the following holds:
If $i$ is not in the index tuple, then $D_i$ is supported on $i$ distinct points, which are the $Q_t$ as above.
If $i=i_j$, then $D_i$ is supported on $i-k_j$ distinct points of the $Q_t$. If the last point of $D_i$ is $Q_s$, then the first point of $D_{i-1}$ is $Q_{s+2}$, so that there is always exactly one point, $Q_{s+1}$, between $D_i$ and $D_{i-1}$.
\end{remark}

\begin{theorem}
\label{mainthm1}
    The number of tropical floor plans of degree $d$ with $\delta$ nodes (i.e., genus $g=\tfrac{(d-1)(d-2)}{2}-\delta$) counted with complex multiplicity equals the Gromov-Witten number $N_{d,g}$, which counts genus $g$ degree~$d$ curves passing through $3d-1+g$ general points: 
     $$\sum_{F \text{ trop. floor plan}\atop \text{of trop curve of degree } d \text{ with } \delta \text{ nodes}} \mult{C}{F}=N_{d,g}. $$
\end{theorem}
\begin{proof}
First we show that there is a bijection between tropical floor plans for curves of degree $d$ and certain column-wise subdivisions of the Newton polytope of degree $d$ curves $d\Delta_2=\conv(\{(0,0),(0,d),(d,0)\})$.
By \cite[Section 3]{GaMa07} the weighted count of these column-wise subdivisions recovers the Gromov-Witten invariant $N_{d,g}$. 

The column-wise subdivisions of $d\Delta_2$ considered in \cite{GaMa07} satisfy special properties stated in \cite[Remark 3.9]{GaMa07}. There properties ensure that they are exactly the subdivisions dual to tropical plane curves through horizontally stretched points defined by a tropical floor plan.

Hence we have the following bijection:
A tropical floor plan uniquely defines a floor decomposed tropical plane curve passing through the horizontally stretched points. Via duality these curves are in bijection to the column-wise subdivisions of the Newton polytope $d\Delta_2$ considered in \cite{GaMa07}. It remains to show that the multiplicity of a tropical floor plan coincides with the multiplicity assigned to such a column-wise subdivision.

The column-wise subdivisions have multiplicity $\prod_{\text{all triangles}}(2\cdot \text{area of triangle})$. So only triangles of area $2w\geq1$ contribute to the multiplicity of the subdivision. These arise from edges of length $w\geq 2$ inside the subdivision, and each such edge is adjacent to two triangles of area $\frac{w}{2}$. 
So the multiplicity of the subdivision can also be given as $\prod_{\text{edge of length } w_j\geq 2}w_j^2$. 

A point of weight $w\geq 2$ in a divisor $D_j$ in the tropical floor plan corresponds to a vertical edge of length $w$ in the dual subdivision of the plane curve induced by the floor plan. Hence, the multiplicity of columnwise subdivisions coincides with the complex multiplicity assigned to the tropical floor plans.

Note, that a non-fixed point of a divisor aligning with another point from a neighboring divisor produces a parallelogram in the dual subdivision. Both for tropical floor plans and for column-wise subdivisions, these do not contribute to the multiplicity.

\end{proof}

\begin{example} 
\label{ex:count}
We now use floor plans to count the 225 binodal plane curves of degree 4 passing through 12 general points \cite[Figure 1]{FM09}.
We distinguish cases by the index tuple encoding the position of the divisors that are not entirely fixed by the point conditions or that contain points of higher weight. We do the case $i_1=4, i_2=3$ with $k_1=k_2=1$ in detail as an example. The other index choices and their contributions are listed in Table \ref{tab:floorplan}. For each case one considers the distribution of points of higher weights in the divisors and, if applicable, the possible alignment options of the non-fixed points in the divisors. Then we compute the multiplicities for each case. 

	Let $i_1=4,i_2=3$. In this case the divisors $D_3$ and $D_4$ are not completely fixed by the point conditions. The divisor $D_4$ passes through $Q_1,\,Q_2$ and $Q_3$. As its degree is $4$, there is one more point in $D_4$ and it must align with a point of $D_3$ and thus depends on the shape of $D_3$. Assuming $D_3$ has a point of weight $2$, we have only one alignment choice as the free point in $D_4$ has to align with a weight $1$ point in $D_3$. There are two choices for a weight $2$ point in $D_3$: $Q_5$ or $Q_6$. This gives \textbf{$2\cdot4=8$} curves.
	If $D_3$ does not have a point of weight $2$, it contains a third point. This could align with a point in $D_2$, giving $3\cdot 2=6$ curves. It can also  align with a point of $D_4$. In that case, there are $2$ alignment possibilities for the point of $D_4$ with $D_3$ (since we cannot align the two aligning points of the two floors). For the point in $D_3$, there are $3$ alignment options in $D_4$, giving $2\cdot3=6 $ curves.
	In total, this contributes $8+6+6=20$ curves. 
	All floor plans with indices $i_1=4,i_2=3$ can be seen in Figure \ref{fig:(4,3)}.

	 \begin{minipage}{\textwidth}
  \begin{minipage}[b]{0.69\textwidth}
    \centering
    \includegraphics[width=\textwidth]{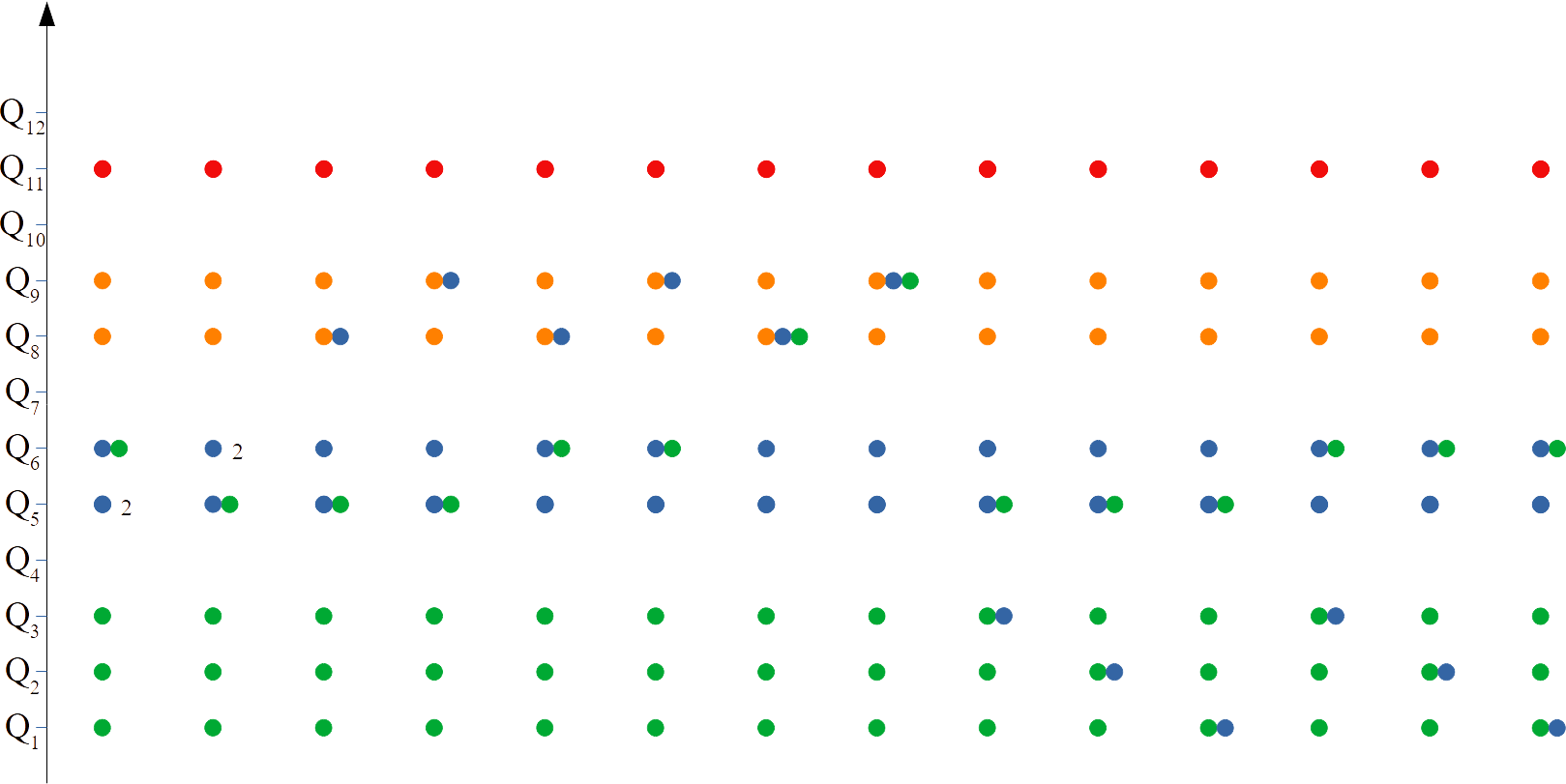}
    \captionof{figure}{Floor plans for curves of degree $4$ with $2$ nodes and index tuple $(4,3)$. Here, $D_1$ is pictured in red, $D_2$ is pictured in orange, $D_3$ is pictured in blue, and $D_4$ is pictured in green.}
    \label{fig:(4,3)}
  \end{minipage}
  \hfill
  \begin{minipage}[b]{0.29\textwidth}
    \centering
    \begin{tabular}{cc}\hline
    indices & count \\ \hline
    $i_1=4$,$k_1=2$ & $3$ \\
    $i_1=3$,$k_1=2$ & $48$ \\
    $i_1=2$,$k_1=2$ & $6$ \\
    $i_1=4,i_2=3$ & $20$ \\
    $i_1=4,i_2=2$ & $24$ \\
    $i_1=4,i_2=1$ & $6$ \\
    $i_1=3,i_2=2$ & $84$ \\
    $i_1=3,i_2=1$ & $28$ \\
    $i_1=2,i_2=1$ & $6$ \\
 \hline
 sum & $225$
      \end{tabular}
      \captionof{table}{}
      \label{tab:floorplan}
    \end{minipage}
  \end{minipage}

\end{example}

\section{Asymptotics}\label{sec:asymptotics}

In \cite{MaMaShSh19}, the authors state the known fact that the family of complex projective surfaces in $\mathbb{P}^3$ of degree $d$ with $\delta$ nodes as their only singularities is of degree
$$
N_{\delta, \mathbb{C}}^{\mathbb{P}^3}(d) = \frac{4^\delta d^{3 \delta}}{\delta !} - \frac{3 \cdot 4^\delta}{\delta!}d^{3\delta - 1} + \mathcal{O}(d^{3 \delta - 2}).
$$
Tropical floor plans are introduced in \cite{MaMaShSh19} to count tropical surfaces where the node germs are at least one floor apart. They show that these account for $\frac{4^\delta d^{3 \delta}}{\delta !} + \mathcal{O}(d^{3 \delta - 1})$ algebraic surfaces, thus recovering the first term of the complex count asymptotically. We will denote by $O_{\delta}(d)$ the number of these \textbf{O}riginal floor plans. In \cite{BG20} we considered floor plans of binodal cubic surfaces where there is no constraint on which floors the node germs may appear. However, some floor plans will produce tropical surfaces in which the nodes are unseparated. In these cases, the multiplicities and possibilities are unknown, and so we do not know how to count them. Let $S_{\delta}(d)$ denote the number of tropical degree $d$ surfaces with $\delta$ \textbf{S}eparated nodes, as counted in \cite{BG20} for the $d=3$, $\delta=2$ case. 

We now introduce artificial $\delta$-nodal floor plans. These do not count any true multinodal tropical surfaces, but we use them as a counting tool; they provide an upper bound for $S_{\delta}(d)$.

\begin{definition}
An \emph{artificial} $\delta$-nodal floor plan of degree $d$ is a tuple of 1-nodal floor plans $(P_1, \ldots, P_\delta)$ where each floor plan $P_i$ has degree $d$ and where the node germ of $P_i$ appears in a higher or the same floor as the node germ of $P_{i+1}$. The \emph{multiplicity} of an artificial floor plan is the product of the multiplicities of the $P_i$. Let $A_{\delta}(d)$ denote the number of \textbf{A}rtificial floor plans counted with multiplicities.
\end{definition}

We think of artificial $\delta$-nodal floor plans as simulating how many surfaces there would be if the nodes did not interact with one another at all, and were not constrained to be separated by a smooth floor.
It is clear that $A_{1}(d) = S_1(d) = O_1(d) =  N_{1, \mathbb{C}}^{\mathbb{P}^3}(d)$ from the definitions. Additionally, in general we have an inequality
$$
A_{\delta}(d) \geq S_\delta(d) \geq O_\delta(d), \text{ where}
$$
\begin{itemize}
 \item   $O_{\delta}$   counts multiplicities of floor plans with $\delta$ nodes separated by a smooth floor, 
 \item $S_{\delta} $ counts multiplicities of floor plans with $\delta$ \textbf{S}eparated nodes, 
 \item $A_{\delta} $ counts multiplicities of \textbf{A}rtificial $\delta$ nodal floor plans. 
\end{itemize}
The floor plans counted by $O_\delta(d)$ are always counted by $S_\delta(d)$.
Contributions of the node germs as in $S_\delta(d)$ are always overcounted by the corresponding artificial floor plan in $A_{\delta}(d)$. For example, in \cite[Case (7c,9f)]{BG20} illustrated in Figure \ref{fig:(7c,9f)}, we have that $S_\delta(d)$ and $A_{\delta}(d)$ count the same number of alignments, but the alignment of the left string with a vertex not adjacent to a horizontal edge (counted by $S_\delta(d)$) has multiplicity 1, whereas $A_{\delta}(d)$ counts an alignment with a horizontal bounded edge with multiplicity 2. This horizontal bounded edge is replaced by the vertex dual to the triangle without vertical edges in \cite[Case (7c,9f)]{BG20}.
\begin{figure}
\centering
\begin{subfigure}{.24\textwidth}
  \centering
  \includegraphics[height = 0.9 in]{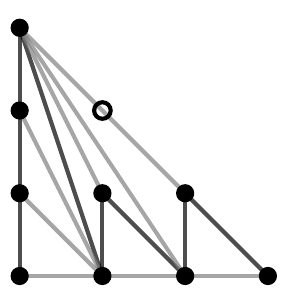}
\end{subfigure}%
\begin{subfigure}{.12\textwidth}
  \centering
  \includegraphics[width=.9\linewidth]{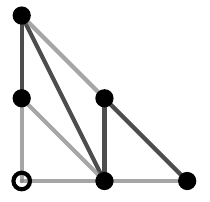}
\end{subfigure}%
\caption{The dual subdivisions to the nodal floors in case \cite[Case (7c,9f)]{BG20}. In the cubic part, the node germ induces a triangle without vertical edges. Without the node germ, there would be two bounded horizontal edges in the cubic floor for the left string in the conic floor to align with.}\label{fig:(7c,9f)}
\end{figure}
As another example, right strings have the maximum possible number of alignment options when the adjacent floor is smooth, so $A_{\delta}(d)$ counts more cases than $S_\delta(d)$.

We will now show that for small $\delta$, the overcount $A_\delta(d)$ is still too small to produce the second coefficient of $N_{\delta, \mathbb{C}}^{\mathbb{P}^3}(d)$. Therefore, relaxing the notion of a floor plan to allow nodes to be in the same or adjacent floors is not enough to produce the second order coefficient in the above count. Thus, the 
unseparated nodes contribute in degree $d^{3 \delta - 1}$.

\begin{theorem}
\label{thm:asymptotic}
For $\delta = 2,3$, there are at most
$$
A_2(d)= 8 d^6 - \frac{168}{5} d^5 + \mathcal{O}(d^4), \ \ \ \ \ 
A_3(d) =\frac{32}{3}d^9 - \frac{1341}{35} d^8 + \mathcal{O}(d^7),
$$
surfaces with separated nodes.
Therefore, surfaces with separated nodes are insufficient to asymptotically count binodal and trinodal surfaces up to two degrees. 
\end{theorem}
This is equivalent to the lower bound for the contribution of unseparated nodes as stated in the introduction.

\begin{proof}

 We begin by counting the multiplicity of a given artificial floor plan $I$. Let $(i_1, \ldots, i_\delta)$ be the weakly decreasing sequence which records the floors of the node germs of $I$.

We now calculate the contribution from the presence of a node germ in floor $i_j$. This will be a polynomial in $i_j$. When counting, we will only consider the top two orders in $i_j$, because these are the only terms that will impact the top two degrees in our asymptotic count.
Given a node germ in floor $i_j$, there are several possibilities for what kind of node germ it is:
\begin{enumerate}
    \item Parallelogram ($i_j \not = 1,d$). By Proposition 4.6 of \cite{MaMaShSh19}, the number of one-nodal curves of degree $i_j$ with a parallelogram is $3 i_j^2 - 6 i_j + 3$, and the complex multiplicity of such a node germ is 
    2. This case contributes $6 i_j^2 - 12 i_j + \mathcal{O}(1)$ to our count.
    \item Horizontal / diagonal end (if $i_j = d$, only a diagonal end of weight 2 is permitted). The number of curves of degree $d$ with a weight 2 horizontal end is $i_j-1$, because there are that many possible locations for a weight 2 horizontal end. This is the same for diagonal ends. For horizontal ends, the multiplicity is $2(i_j + 1)$ by Definition 5.4 of \cite{MaMaShSh19} and for diagonal ends it is $2(i_j-1)$. Together, this contributes $4 i_j^2 - 4i_j + \mathcal{O}(1)$ to our count. 
    \item Left string / right string (If $i_j = d$ then there can only be a right string, and if $i_j = 1$ there can only be a left string). 
    If it is a left string, it can align with any horizontal bounded edge of $C_{i_j+1}$, and since the latter is a smooth tropical plane curve of degree $i_j+1$ satisfying horizontally stretched point conditions, there are $i_j+\ldots+1=\frac{1}{2}(i_j^2+i_j)$ such edges. For each, the multiplicity is 2.
     Right strings can either meet a diagonal bounded edge of $C_{i_j-1}$ of which there are $i_j-2$ many (and each is counted with multiplicity 2), or a vertex not adjacent to a diagonal edge, of which there are $2(i_j-3+i_j-4+\ldots+1)= (i_j-3)^2+(i_j-4) = i_j^2 - 5i_j + \mathcal{O}(1)$ many. 
      In total, this gives $2i_j^2 - 2i_j + \mathcal{O}(1)$ to our count.
\end{enumerate}
Thus, a node germ in floor $i_j$ for $i_j \not = 1,d$ contributes $12i_j^2 - 18 i_j + \mathcal{O}(1)$. If $i_j = d$, then it contributes $3i_j^2 - 7 i_j + \mathcal{O}(1)=3d^2 - 7 d + \mathcal{O}(1)$, and if $i_j = 1$, then it contributes $i_j^2 + i_j + \mathcal{O}(1) = \mathcal{O}(1)$.

Suppose $\delta = 2$. We sum over all choices of $i_1, i_2$ to obtain the count. We ignore the possibility that there is ever a node in the first floor because this will not contribute to the top two terms asymptotically.
We have

\begin{align*}
A_2(d) &= \sum_{i_1 = 2}^{d-1} \sum_{i_2 = 2}^{i_1} \prod_{j=1}^2 (12 i_j^2 - 18 i_j)
+ (3 d^2 - 7 d) \sum_{i_2 = 2}^{d-1} (12 i_2^2 - 18 i_2)
+ (3 d^2 - 7 d)^2\\
 &= 8 d^6 - \frac{168}{5} d^5 + \mathcal{O}(d^4).
\end{align*}

Suppose $\delta = 3$. We sum over all choices of $i_1, i_2, i_3 > 1$ to obtain the count. We ignore the possibility of  a node in the first floor because this will not contribute to the asymptotic count. We have

\begin{align*}
  A_3(3) =& \sum_{i_1 = 2}^{d-1} \sum_{i_2 = 2}^{i_1} \sum_{i_3 = 2}^{i_2} \prod_{j=1}^3 (12 i_j^2 - 18 i_j)
    + (3 d^2 - 7 d) \sum_{i_2 = 2}^{d-1}\sum_{i_3 = 2}^{i_2} \prod_{j=2}^3 (12 i_2^2 - 18 i_2)\\
   & + (3 d^2 - 7 d)^2 \sum_{i_3 = 2}^{d-1} (12 i_2^2 - 18 i_2)
    + (3 d^2 - 7 d)^3 \\
    =&\frac{32}{3}d^9 - \frac{1341}{35} d^8 + \mathcal{O}(d^7).
\end{align*}

\end{proof}

We obtain the following as an immediate corollary.

\begin{corollary}
There are at least $ \frac{48}{5} d^5 + \mathcal{O}(d)$ binodal surfaces of degree $d$ with unseparated nodes.
There are at least $\frac{221}{35} d^8 + \mathcal{O}(d)$ trinodal surfaces of degree $d$ with unseparated nodes.
\end{corollary}

\begin{remark}
For $\delta = 4$, our method of proof is inconclusive. One can see this as follows. We sum over all choices of $i_1, i_2, i_3, i_4 > 1$ to obtain the count. We ignore the possibility of  a node in the first floor because this will not contribute to the count. Additionally, we allow at most $i_1$ to possibly equal $d$, because (as can be seen for $\delta= 2, 3$) having more than one $i_j = d$ does not contribute to the top two degrees asymptotically. We have

\begin{align*}
   A_4(d)&= \sum_{i_1 = 2}^{d-1} \sum_{i_2 = 2}^{i_1} \sum_{i_3 = 2}^{i_2} \sum_{i_4 = 2}^{i_3} \prod_{j=1}^4 (12 i_j^2 - 18 i_j)
    + (3 d^2 - 7 d) \sum_{i_2 = 2}^{d-1}\sum_{i_3 = 2}^{i_2} \sum_{i_4 = 2}^{i_3} \prod_{j=2}^4 (12 i_2^2 - 18 i_2)\\
    &= \frac{32}{3} d^{12} - \frac{64}{5} d^{11} + \mathcal{O}(d^{10}).
\end{align*}

Now, $\tfrac{-64}{5} > -32$. So, for 4 nodes this method is inconclusive to determine whether surfaces with unseparated nodes are sufficient to produce the top two degrees of the count.
\end{remark}

\section{Binodal polytopes}\label{sec:binodalpolytopes}

In order to tropically count surfaces of a given degree with at least two nodes we have to count unseparated nodes. To do this, we must first understand the smallest possible Newton polytopes of binodal surfaces. We say a three-dimensional polytope $P$ is \emph{binodal} if there exists a surface with two distinct nodes in the torus that has Newton polytope $P$. It has support $P\cap\mathbb{Z}^3$ if all coefficients of the polynomial defining the surface are non-zero. 
In \cite{BG20}, we give the following necessary condition for a polytope to appear in a subdivision of a binodal tropical surface, which can often be verified computationally.

	\begin{lemma}\cite[Lemma 3.3.]{BG20}\label{lem:computations}
	Let $\Gamma\subset \mathbb{Z}^3$ be finite, and let $B_{\Gamma}$ be the variety of binodal hypersurfaces
whose defining polynomial has support $\Gamma$. If the projective dimension of  $B_{\Gamma}$ is less than $|\Gamma|-3$, then any tropical surface whose dual subdivision consists of unimodular tetrahedra away from $Conv(\Gamma)$ is not the tropicalization of a complex binodal cubic surface.
If $B_\Gamma$ is empty, then $Conv(\Gamma)$ is not binodal.
	\end{lemma}

 Satisfying Lemma \ref{lem:computations} is not sufficient to ensure that the polytope is binodal. For example, it is possible that $B_\Gamma$ actually consists only of surfaces with non-isolated singularities, which can be difficult to check in general.
 
 In this section we will analyze polytopes $P$ with small numbers of lattice points $|P \cap \mathbb{Z}^3|$ and investigate whether or not they are binodal. We use the classifications of three-dimensional polytopes with small numbers of lattice points given by Blanco and Santos \cite{bs14, BS15, BS18}. These classifications give the finitely many 3-polytopes with $n$ lattice points for $5 \leq n \leq 11$ and width greater than 1. For polytopes with width equal to 1, they give a classification of the finitely many infinite families for polytopes with 5 and 6 lattice points. We will only consider polytopes with 6 or more lattice points for the following reason.
 
 \begin{lemma}
 \label{lem:6latticepoints}
     A binodal 3-polytope contains at least six lattice points.
 \end{lemma}
\begin{proof}
Suppose for contradiction that there were a binodal 3-polytope $\Gamma$ with 5 lattice points. Then the projective dimension of $B_\Gamma$ is 2.
On the other hand, there exists a binodal surface with Newton polytope $\Gamma$. By translating this surface, we obtain a 3-dimensional family of binodal surfaces inside $B_\Gamma$, so the dimension of $B_\Gamma$ is at least 3, giving a contradiction.
\end{proof}

\begin{remark}
When a polytope has width greater than 1, it will not appear in the subdivision of a floor decomposed surface without being further subdivided. However, we do not know that tropical surfaces through points in Mikhalkin position with more than one node are necessarily floor decomposed.

As we will see below, binodal polytopes with 6 lattice points can only appear in the dual subdivision of a tropical surface if the binodal polytope has the trivial subdivision. Therefore, only the ones with width 1 could appear in a floor decomposed surface.
\end{remark}

\begin{proposition}
\label{prop:fan}
Let $\Gamma$ be a binodal polytope with 6 lattice points, and let $S$ be a binodal surface over $\mathbb{C}\{\{t\}\}$ with Newton polytope $\Gamma$. Then $trop(S)$ is a fan and both nodes tropicalize to the vertex of the fan.
\end{proposition}

\begin{proof}
Let $B_\Gamma$ be the variety of binodal surfaces with Newton polytope $\Gamma$. It has projective dimension $3$ inside $\mathbb{P}^5$. 

Without loss of generality, we may translate the surface $S$ so that one node of $S$ is located at $(1,1,1)$. Let $L$ be the variety of surfaces with Newton polytope $\Gamma$ and one node at $(1,1,1)$. 
Let the lattice points of $\Gamma$ be described by a $3 \times 6$ matrix $A = (a_{ij})$, and let $c_1, \ldots, c_6$ be the coordinates on $\mathbb{P}^5$. Let $f(x,y,z) = c_1x^{a_{1,1}}y^{a_{2,1}}z^{a_{3,1}}+ \cdots + c_6x^{a_{1,6}}y^{a_{2,6}}z^{a_{3,6}}$ be a generic polynomial defining a surface with Newton polytope $\Gamma$.
Then $L$ is the variety cut out by the linear equations $f(1,1,1)=0$, $df/dx(1,1,1) = 0,$ $df/dy(1,1,1) = 0,$ and $df/dz(1,1,1) = 0$. 
Explicitly, $L$ is the linear space defined by the kernel of the $4 \times 6$ matrix 
$$
\begin{bmatrix}
1 & 1 & 1& 1& 1& 1 \\
a_{1,1} & a_{1,2} & a_{1,3} & a_{1,4} & a_{1,5} & a_{1,6}\\
a_{2,1} & a_{2,2} & a_{2,3} & a_{2,4} & a_{2,5} & a_{2,6}\\
a_{3,1} & a_{3,2} & a_{3,3} & a_{3,4} & a_{3,5} & a_{3,6}\\
\end{bmatrix}.
$$
 Since $\Gamma$ is 3-dimensional, this matrix has rank 4, and so the kernel has dimension 2, and thus gives a linear space of projective dimension 1. 
 
 Let $D_\Gamma$ be the discriminant of $\Gamma$, i.e., the variety of those surfaces with at least one node. Then $D_\Gamma$ has codimension 1 in $\mathbb{P}^5$, and $B_\Gamma \subset D_\Gamma$. Further, the linear space $L \subset D_\Gamma$. Since $B_\Gamma$ has codimension 1 in $D_\Gamma$ and $L$ has dimension 1 in $D_\Gamma$, the intersection $B_\Gamma \cap L \subset D_{\Gamma}$ has dimension 0 and consists of finitely many points ($L$ is not contained in $B_\Gamma$ because there may be surfaces with exactly one node, located at $(1,1,1)$, which are described by points in $L$ and not in $B_\Gamma$).

Since $B_\Gamma$ and $L$ are defined by equations over $\mathbb{C}$, and 
$\mathbb{C} \subset \mathbb{C}\{\{t\}\}$ is algebraically closed, points in $B_\Gamma \cap L$  are defined over $\mathbb{C}$. 
Each point in $B_\Gamma \cap L$ is a complex surface whose nodes are also defined over $\mathbb{C}$. Thus, the tropicalization of surfaces in $B_\Gamma \cap L$ are fans with vertex at $(0,0,0)$ and both nodes tropicalizing to $(0,0,0)$. 
\end{proof}

\begin{remark}
This proposition and its proof hold more generally for $\delta$-nodal polytopes of dimension $d$ with $\delta + d + 1$ lattice points. 
\end{remark}

\begin{remark}
In this paper we do not discuss polytopes with 7 or more lattice points. Binodal polytopes with 7 lattice points can be properly subdivided inside the subdivision. There are 496 three-dimensional polytopes with 7 lattice points and width greater than 1 \cite{BS18}. The infinite families of polytopes with 7 vertices and width 1 are not classified.
\end{remark}

\subsection{Polytopes with 6 lattice points of width 1}
\label{subsec:6points}
We consider polytopes up to integral unimodular affine transformations, that means affine translations and multiplication, of the lattice points with an element of $GL_3(\mathbb{Z})$. We call two polytopes that can be transferred into each other by these operations \emph{IUA-equivalent}.

We use the classification of polytopes with $6$ lattice points of width $1$ from \cite{BS15} to search for binodal polytopes. 
We have translated each polytope to ensure that all coordinate entries are positive.
For consistency with \cite{BG20} we permute the coordinates to $(x\,z\,y)$ and order the points by how they appear in the lattice path.
Additionally, we delete those polytopes for which a permuted version appears twice in the Tables 5, 6, and 7 from  \cite{BS15}; see \cite[Remark 4.3]{BS15}. We enumerate the remaining polytope families from 1 to 21. Then, we remove any polytopes that can be eliminated using Lemma~\ref{lem:computations}. We list the remaining polytopes in  Figure \ref{tab:allpolytopes}. In \cite{BS15} polytope family 21 is  given with the restriction $ad-bc=\pm 1$. By applying the $GL_3(\mathbb{Z})$ operation that switches the entries in the $y$ and $z$ coordinates, we see that the polytope family satisfies a symmetry condition, and without loss of generality, we can assume that $ad-bc=1$.

\begin{figure}
    \centering
\begin{minipage}{0.3 \linewidth}
  \begin{center}
  No. 1 
  
  $
   \begin{pmatrix}
				0&0&0&0&0&1 \\
				0&1&1&1&2&1\\
				0&0&1&2&0&0
			\end{pmatrix}$\\
			
			\hfill
			
			\hfill
			
			\hfill 
			
  \end{center}
\end{minipage}
\begin{minipage}{0.3 \linewidth}
  \begin{center}
  No. 8
  
  $
   \begin{pmatrix}
	0&0&0&0&1&1\\
	0&1&1&2&0&a\\
	0&1&2&1&0&b 
		\end{pmatrix}$\\
			
			$gcd(a,b) = 1,$  $2b \not=a,\ 0<b<a$
  \end{center}
\end{minipage}
\begin{minipage}{0.3 \linewidth}
  \begin{center}
  No. 9
  
  $
  \begin{pmatrix}
			0&0&0&0&1&1\\
			0&0&1&1&0&0\\
			0&1&0&1&0&1
		\end{pmatrix}$\\
		
		\hfill
		
		\hfill
		
		\hfill
		
  \end{center}
\end{minipage}
\vspace{0.2 in}

\begin{minipage}{0.25 \linewidth}
  \begin{center}
  No. 10
  
  $
  \begin{pmatrix}
			0&0&0&0&1&1\\
			0&0&1&1&0&a\\
			0&1&0&1&0&b\\
		\end{pmatrix}$\\
		
		$gcd(a,b)=1$, $0< b\leq a$ 
		
  \end{center}
\end{minipage}
\begin{minipage}{0.3 \linewidth}
  \begin{center}
  No. 13
  
  $
\begin{pmatrix}
	0&0&0&0&1&1\\
	0&1&1&2&0&a\\
	0&0&1&0&0&b
\end{pmatrix}$\\
		
		$gcd(a,b)=1$, $0< b < a$

  \end{center}
\end{minipage}
\begin{minipage}{0.3 \linewidth}
  \begin{center}
  No. 14
  
  $
\begin{pmatrix}
	0&0&0&0&1&1\\
	0&1&1&2&0&b\\
	0&0&1&0&0&a
\end{pmatrix}$\\
		
		$gcd(a,b)=1$, $0< b < a$ 
		
  \end{center}
\end{minipage}
\vspace{0.3 in}

\begin{minipage}{0.25 \linewidth}
  \begin{center}
  No. 15
  
  $
\begin{pmatrix}
	0&0&0&0&1&1\\
	0&0&0&0&0&a\\
	0&1&2&3&0&b
\end{pmatrix}$\\
		
		$gcd(a,b)=1$, $0\leq b < a$ 
		
  \end{center}
\end{minipage}
\begin{minipage}{0.45 \linewidth}
  \begin{center}
  No. 16
  
  $
\begin{pmatrix}
0&	0&0&1&1&1\\
0&	1& 2 &0& b & 2b &\\
0&	0&0&0&a&2a& 
	\end{pmatrix}$\\
		
		$gcd(a,b)=1$, $0\leq b < a$
		
		\hfill
		
  \end{center}
\end{minipage}
\vspace{0.2 in}

\begin{minipage}{0.33 \linewidth}
  \begin{center}
  No. 17
  
  $
\begin{pmatrix}
		0&0&0&1&1&1\\
		0&0&1&0&a&2a\\
		0& 1&0 &0& b & 2b\\
	\end{pmatrix}$\\
		
		$gcd(a,b)=1$, $0< b \leq  a$
		
		\hfill
		
  \end{center}
\end{minipage}
\begin{minipage}{0.25 \linewidth}
  \begin{center}
  No. 20
  
  $
\begin{pmatrix}
	0&0&0&1&1&1\\
	1&1&2&0&1&1\\
	0&1&0&a&0&1
\end{pmatrix}$\\
		
		$a \geq 3$
		
		\hfill
		
  \end{center}
\end{minipage}
\begin{minipage}{0.3 \linewidth}
  \begin{center}
  No. 21
  
  $
\begin{pmatrix}
	0&0&0&1&1&1\\
	0&0&1&0&a&c\\
	0&1&0&0&b&d
	\end{pmatrix}$\\
		
		 $ad-bc =  1$, $a,b,c,d >0$, $c+d>a+b$, $c>a$
		
  \end{center}
\end{minipage}
\caption{The polytopes with 6 lattice points of width 1 from \cite{BS15} remaining after applying Lemma \ref{lem:computations}.}
\label{tab:allpolytopes}
\end{figure}

\begin{proposition}
Of the polytope families listed in Table \ref{tab:allpolytopes}, polytopes from families 1,\ 9,\ 15,\ 16, and \ 17 are not binodal.
\label{prop:eliminatepolytopes}
\end{proposition}

\begin{proof}
We give an argument eliminating each polytope. Details can be found in Appendix~\ref{app:prop:eliminatepolytopes}.

\begin{enumerate}
    \item[1.] The binodal variety  is empty.
    \item[9.] Every surface in the binodal locus contains non-isolated singularities because the polynomial defining a surface with this Newton polytope always factors, see Appendix~\ref{app:prop:eliminatepolytopes} for details.
  \item[15.] The binodal variety is empty, independent of $a,b$.
     \item[16.] For any choice of $a,b$, the surfaces in the binodal locus will have non-isolated singularities. See Appendix~\ref{app:prop:eliminatepolytopes} for details.
    \item[17.] For any $a,b$, the binodal locus is empty. This can be deduced from the equations asserting that a hypersurface with this Newton polytope has a singularity; see Appendix~\ref{app:prop:eliminatepolytopes} for details.
\end{enumerate}
\end{proof}

\begin{remark}
In the list of remaining polytopes, it is possible that some still have the property that their binodal locus contains only surfaces with non-isolated singularities.
\end{remark}

\begin{proposition}
The following values for the parameters in the given families do not lead to a binodal polytope:
\begin{itemize}
    \item[]family 10: $a=1, b=1$ and $a=2,b=1$.
    \item[]family 13: $a=2,b=1$ and $a=3, b=1$.
    \item[]family 14: $a=2,b=1$ and $a=3, b=1$.
    \item[]family 20: $a=3$.
\end{itemize}
\end{proposition}
\begin{proof}
Using \texttt{OSCAR} or \texttt{Singular} we can compute the dimension of the binodal varieties. Only for $a=3, b=1$ in families 13 and 14 we obtain the expected dimension \cite{singular,oscar,Code21}. For the other families and parameter values the computed dimension is too small. However, the binodal variety for families 13 and 14 for $a=3,b=1$ are empty.
\end{proof}

\subsection{Lattice paths}
\label{subsec:paths}
We compute a list of possible lattice paths for each polytope family.
We note, that for each polytope family in  Table \ref{tab:allpolytopes} the edge-vertex combinatorics remains constant across different choices of parameters, except for polytope family 8, where there are two cases, depending on whether $2b>a$ or $2b<a$. This is significant, because the possible lattice paths only depend on the ordering of the vertices and the existence of edges between the vertices.

By Proposition \ref{prop:fan}, the dual subdivision to a binodal tropical surface whose Newton polytope is a binodal polytope with 6 lattice points
is always the trivial subdivision. It follows that if the surface passes through points in Mikhalkin position, then the corresponding lattice path goes through the edges of the polytope.

 \begin{lemma}
     The direction vector for the points in Mikhalkin position $v =(1,\eta,\eta^2)$ with $0<\eta \ll 1$ gives an  order on the lattice points of a polytope $P:$ $p_0, \ldots, p_n$. The line $L$ on which the point conditions are distributed always passes through the areas dual to $p_0$ and $p_n$.
 \end{lemma}
\begin{proof}
By \cite[Lemma 3.2]{MaMaSh18},
  the line $L$ on which the points in Mikhalkin position are distributed has to pass through the 3-dimensional areas defined by the tropical surface via $\mathbb{R}^3\setminus S$ in the order  given by the partial order induced on their dual vertices by the direction vector of $L$.

The ordering of the lattice points induced by $v$ is given by $p_i<p_j$ if and only if $\langle p_j-p_i,v\rangle >0,$ see Definition \ref{def:partialorder}. 
The ordering of the vertices stands for the order in which the line through the point configuration passes through the cells dual to the vertices. For a smooth surface, the line (with direction vector $(1,\eta,\eta^2)$) has to pass through all $n+1$ areas into which the surface divides $\mathbb{R}^3.$ Due to the direction vector of the line on which the points are distributed, the line always has to start in the area of $\mathbb{R}^3$ dual to the first vertex $p_0$ and it always has to end in the area dual to $p_n$, even when the surface is no longer smooth.
\end{proof}

Tables \ref{tab:bigtab1} and \ref{tab:bigtab2} list the possible lattice paths for each polytope from Table \ref{tab:allpolytopes} after reducing by Proposition \ref{prop:eliminatepolytopes}. The possible lattice paths are determined in Lemmas \ref{lem:disconnected42} and \ref{lem:nodiscon20}. By Lemma \ref{lem:disconnected42}, no polytope in Table \ref{tab:bigtab1} can have disconnected lattice paths. 
Lemma \ref{lem:nodiscon20} determines the lattice paths for the polytopes in Table \ref{tab:bigtab2}.

 \begin{lemma}
 \label{lem:disconnected42}
 For the polytopes with $6$ lattice points of width $1$ with distribution of the lattice points $4+2$ between the two floors there are no disconnected lattice paths possible for the positions as in Table \ref{tab:allpolytopes}.

 \end{lemma}
 \begin{proof} 
 The only cases of binodal polytopes with 6 vertices and a lattice point distribution of $4+2$ are the polytope families 8, 10, 13 and 14. We translate the tropical surface dual to the polytope such that the vertex is at $(0,0,0)$.  The direction vector of the line $L$ on which the points of the point condition lie is $(1,\eta, \eta^2)$ with $0<\eta\ll 1$.
 
 In the following, we use the notation $\cdot^\vee$ to denote the dual object under the duality connection between the Newton polytope and the tropical surface.
 
 Consider a polytope from family 8, 13, or 14. 
 There can only exist disconnected lattice paths for one of these polytopes if the line $L$, on which the points from the point conditions are distributed, intersects with the tropical surface more than 3 times. Label the lattice points of the polytope with $A, B, C, D, E, F$, corresponding to the columns of the matrices in Figure \ref{tab:allpolytopes}. In these cases, point $B$ is not a vertex of the polytope. So, the line must pass through all five areas into which the tropical surface subdivides $\mathbb{R}^3$. After the line passes through $A^\vee, C^\vee, D^\vee$, we can see from examining the generators of these cones that the remaining ray of $L$ is contained in the $\{z>0, y>0\}$ quadrant of $\mathbb{R}^3$. However, for each of these polytopes, $E^\vee \cap \{z>0,y>0\} = \emptyset$. Thus, $L$ cannot pass through all five 3-dimensional areas defined by the tropical surface.
 
 Now consider a polytope from family 10.
 The line $L$ must pass through five or more of the $3$-dimensional areas into which the tropical surface subdivides $\mathbb{R}^3$.
 We can compute that $B^\vee \subset \{y<0,z>0\}$ and $C^\vee\subset \{y>0,z<0\}$, so we know that $L$ can only pass through one of them. Moreover, $D^\vee\subset \{y>0,z>0\}$ and $E^\vee\cap \{y>0\}\subset \{z<0\}$ while $E^\vee\cap\{z>0\}\subset \{y<0\}$, so again $L$ can only pass through one of these areas.   It follows that $L$ can pass at most through four of the six 3-dimensional areas defined by the tropical surface.
 \end{proof}
\begin{remark}
Under transformations, polytopes of family 10 can be brought into a position where the lattice points are distributed 3+3 between two floors. We will see this position again in Section \ref{subsec:degdsurface}. As the argument in the proof of Lemma \ref{lem:disconnected42} depends on the exact position of the vertices under transformation, this case has to be checked independently. However, an analysis using the same tools as above shows that there are still no disconnected lattice paths possible.
\end{remark}

\begin{lemma}\label{lem:nodiscon20}
 For the polytopes of families 20 and 21
 we can only have the lattice paths 
 shown in Table \ref{tab:bigtab2}. 
 \end{lemma}
\begin{proof}
We examine each possible lattice path through the vertices of the two polytopes and check whether it is possible for a line to pass through the corresponding regions and through points in Mikhalkin position. This is done by choosing general points $(0,0,0)$, $(1, \eta, \eta^2)$, $(\lambda, \lambda \eta, \lambda \eta^2)$, $ 0<\eta\ll 1\ll \lambda$ and letting $(x,y,z)$ be the vertex of the tropical surface. This gives rise to 9 equations, one for each coordinate of each point,  asserting that the point is contained in the claimed cell of the tropical surface, in 6 variables corresponding to the coefficients on the rays and the 3 variables $x,y,z$. Solving these equations either results in a solution or a contradiction. See Appendix \ref{app:lem:nodiscon} for the detailed computations. Techniques similar to those used in Appendix \ref{app:lem:nodiscon} show that the paths in Table \ref{tab:bigtab1} are all possible with respect to the point conditions.
\end{proof}

\begin{lemma}
 A disconnected lattice path of a polytope from family 20 or 21 can have at most one gap.
\end{lemma}
\begin{proof}
If there were a disconnected lattice path with two gaps, then the path would contain all lattice points of the binodal polytope. This means that the line $L$ on which the points from the point condition are distributed passes through all the areas of $\mathbb{R}^3\setminus S$. It is a conclusion from the proofs of Lemma \ref{lem:disconnected42} and \ref{lem:nodiscon20} that this is not possible for any of our binodal polytope families.
\end{proof}

\begin{conjecture}
The degree of the binodal locus of binodal polytopes in families $10$, $13$, $14$, $20$ and $21$, as well as the multiplicities of their different lattices paths, are as stated in Tables \ref{tab:bigtab1} and \ref{tab:bigtab2}. 
\end{conjecture}
The conjecture relies on an extrapolation of the values that were computable. We do not make a conjecture for polytope family 8, but we give the results of our computations in Table \ref{tab:degreepolyA8}.

We compute the multiplicity of a lattice path for a binodal polytope as follows. 
For a connected lattice path that leaves out the two lattice points $a_i$ and $a_j$ we substitute the remaining parameters in the binodal variety (i.e., all but $a_i$ and $ a_j$) by generic values. Computing the degree of this variety gives the multiplicity of the lattice path. The values for the parameters have not been chosen generically enough if the dimension of the variety in $a_i$ and $a_j$ is not zero.

For a disconnected lattice path the procedure is similar.  If the segment $[a_i,a_i+1]$ is the one filling the gap in the disconnected lattice path and the path up to $a_i$ is connected, then we proceed with the parameters up to $a_i$ as described before. The parameter $a_{i+1}$ does not get assigned a value, but stays variable. However, all the following parameters, i.e., $a_j$ with $j>i+1$, will be substituted by generic linear equations in $a_j,$ such that the quotient $a_{j}/a_{i+1}$ is constant. The degree of the new variety gives the multiplicity of the disconnected lattice path. The values and linear equations are generic enough if the variety is of dimension 0.

We performed all these computations using \texttt{Singular} \cite{singular}, and code in \texttt{OSCAR} \cite{oscar} will be available, see \cite{Code21}. 
We note that this construction works even for invalid paths, e.g., paths that do not satisfy the previous lemmas. Hence, a non-zero multiplicity resulting from these computations does not imply that the lattice path contributes to the total count.

\begin{table}[]
    \centering
    \begin{tabular}{p{1.4cm}|c | c | c|p{1.3cm}}
    Family & Conjectured Degree & Lattice paths & Conjectured Multiplicity & Verified\\
    \hline
8 & -- &  \includegraphics[width=1 in]{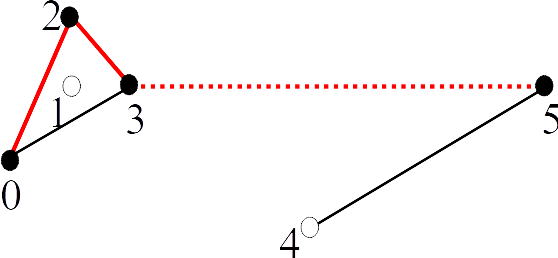} & -- & --\\
 &  &  
   \includegraphics[width=1 in]{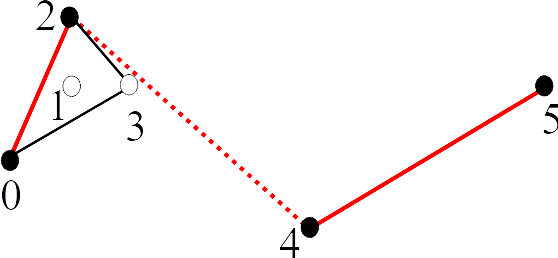} 
 & -- & \\
  &  &
  \begin{minipage}{1.2 in}
  \hspace{-0.13 in}
  - - - - - - - - - - - - - - -
   \hspace{-0.2 in}
  
     \includegraphics[width=1 in]{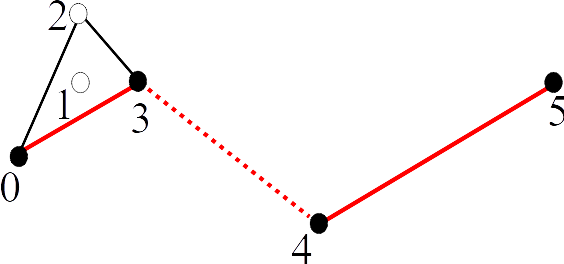} 
     
     \hspace{0.7 in} $a<2b$
      
      \vspace{0.1 in}
  \end{minipage}
 & -- & \\
 \hline
10 & $(a-2)(a+b+2)$ & \includegraphics[width=1 in]{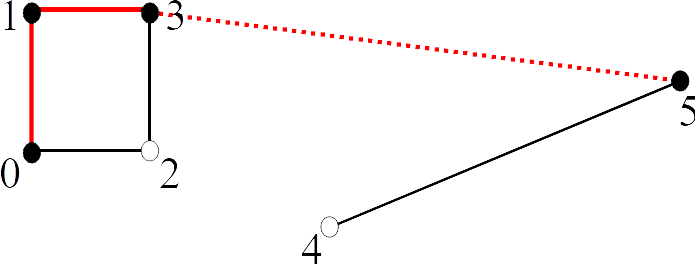}  & $a-2$  & $a\leq 7,$\newline $b\leq 4$\\
&& \includegraphics[width=1 in]{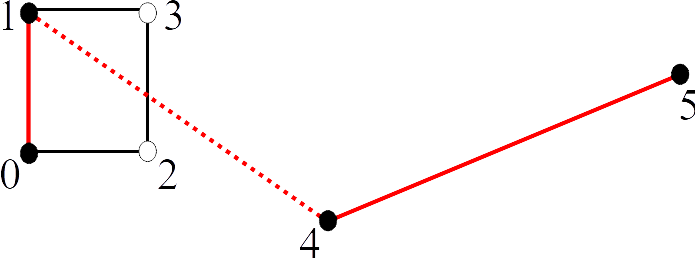}& $a(a-2)$ &\\
&&\includegraphics[width=1 in]{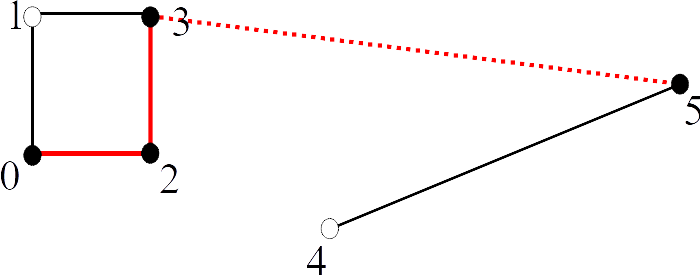}& $a-2$ &\\ 
&&\includegraphics[width=1 in]{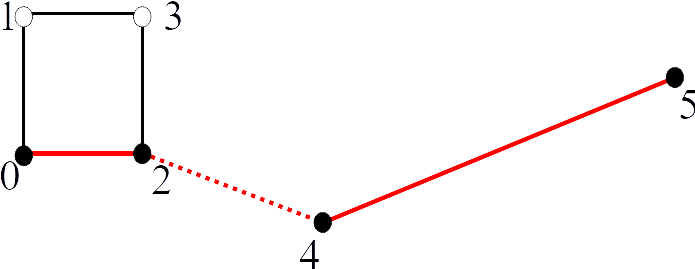}&  $b(a-2)$ &\\
\hline
13 & \shortstack{$\frac{1}{2}(a^2 + b^2 +(\kappa-2)(a+b))$ \\ $+ ab-4+\kappa$}  
& \includegraphics[width=1 in]{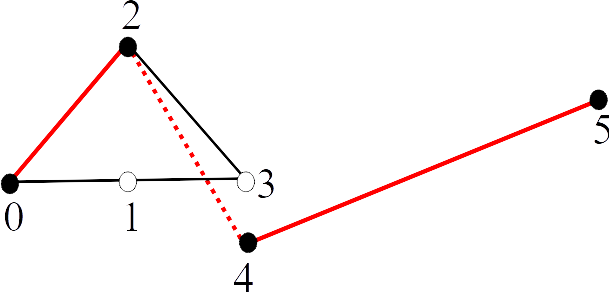}  & \shortstack{$\frac{1}{2}(a^2 -b^2 +(\kappa-4)(a-b))$} & $a\leq 9,$ \newline$b\leq 8$ \\
&& \includegraphics[width=1 in]{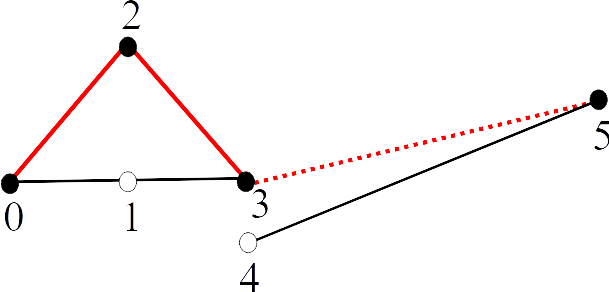}&  $a+b-4+\kappa$ &\\
&& \includegraphics[width=1 in]{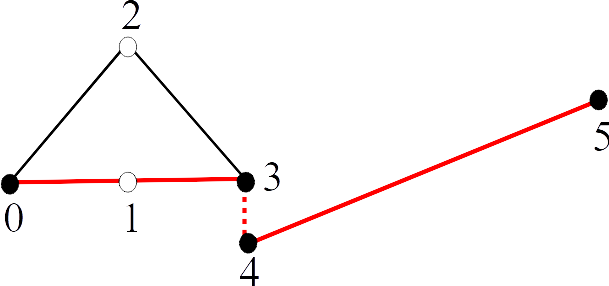}& $b(a+b-4+\kappa)$ &\\
\hline
14 & $(a+1)(a+b-4+\kappa)$ & \includegraphics[width=1 in]{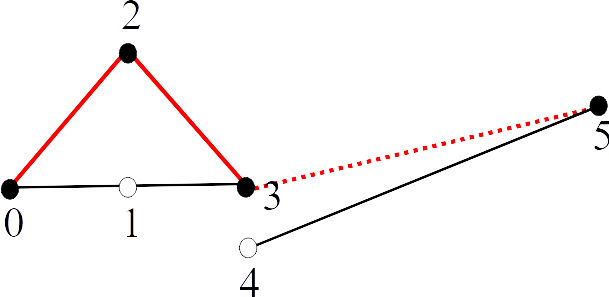}  & $a+b-4+\kappa$ & $a \leq 9,$\newline $b \leq 7$ \\
&& \includegraphics[width=1 in]{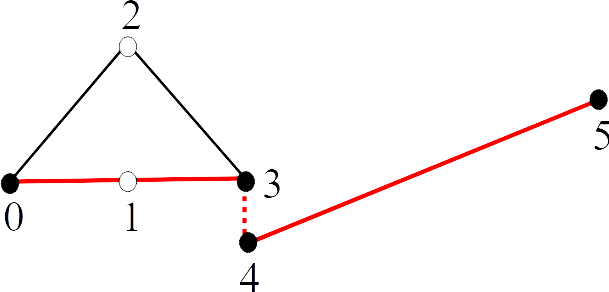}& $a(a+b-4+\kappa)$ &\\
    	\hline
    \end{tabular}
    \caption{The conjectured degree of the binodal locus depending on $a$ and $b$, possible lattice paths, and their conjectured multiplicities for polytopes in families 8, 10, 13 and 14. . Polytopes from family 10 are only binodal when $a \geq 3$. For polytope families 13 and 14, the polytope is only binodal if, for $b=1$, we have $a > 3$. In the formulas for polytopes 13 and 14 we set $\kappa = a+b \mod 2$.
    }
    \label{tab:bigtab1}
\end{table}

	\begin{table}[]
    \centering
    \begin{tabular}{c|c | c | c| c}
    Family & Conjectured Degree & Lattice paths & Conjectured Multiplicity & Verified \\
\hline
20 & $(a-3)(a+2)$ & \includegraphics[height=0.7 in]{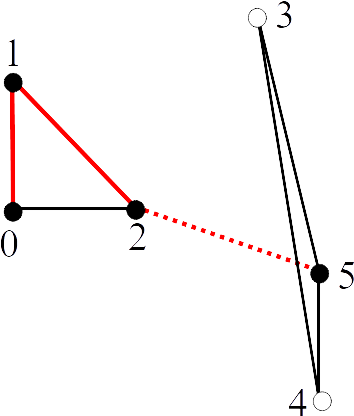}  &  $a-3$  & $a \leq 7$\\
&& \includegraphics[height=0.7 in]{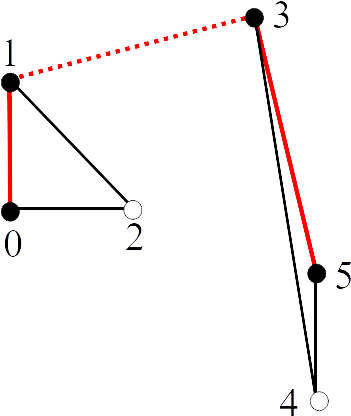} & $a-3$ & \\
&& \includegraphics[height=0.7 in]{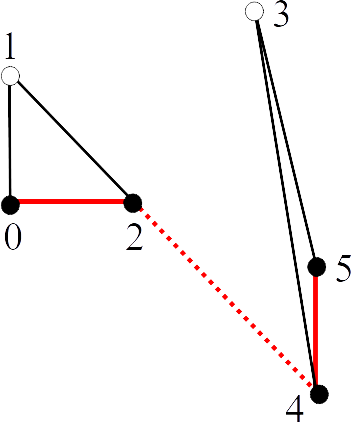} & $a-3$ & \\
&& \includegraphics[height=0.7 in]{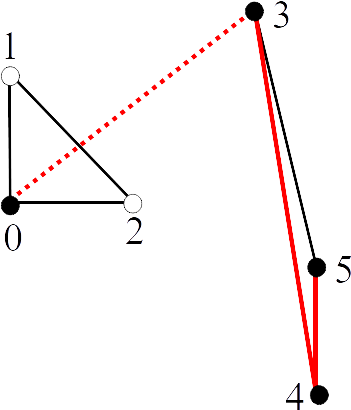} & $a-3$ &\\
&& \includegraphics[height=0.7 in]{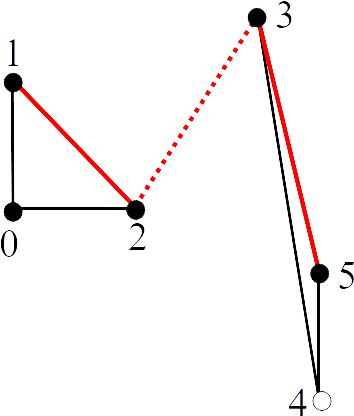} &$(a-2)(a-3)$&\\
\hline 
21 & $(d+c+2)(d+c-4)$ & \includegraphics[height=0.7 in]{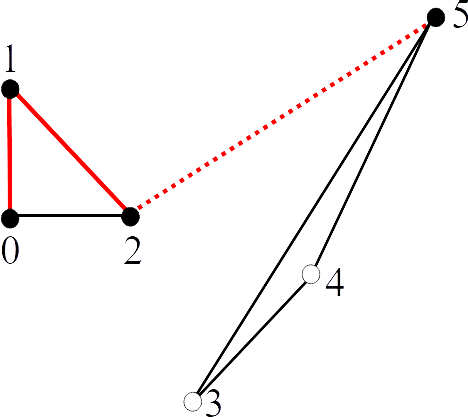} & $c+d-4$ & 
\begin{minipage}{1.5 in}
$a,b,c,d \leq 5$ 
\end{minipage}
 \\ 
&&\includegraphics[height=0.7 in]{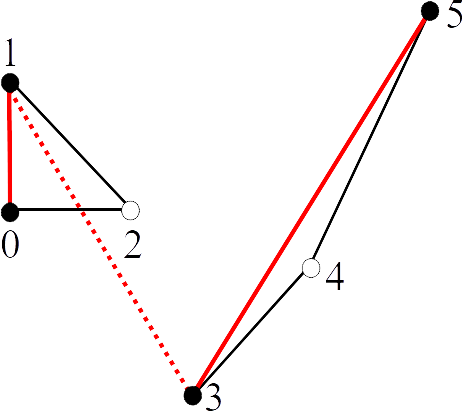} &  $c(c+d-4)$ & \\
&& \includegraphics[height=0.7 in]{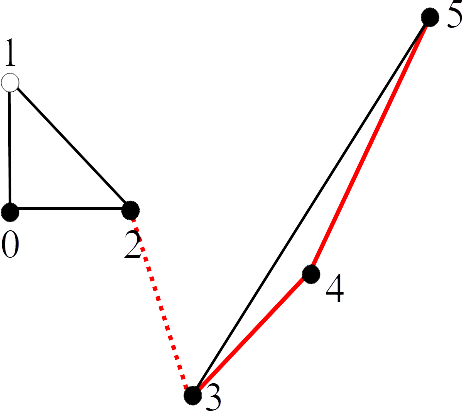} &  $c+d-4$ &  \\
&& \includegraphics[height=0.7 in]{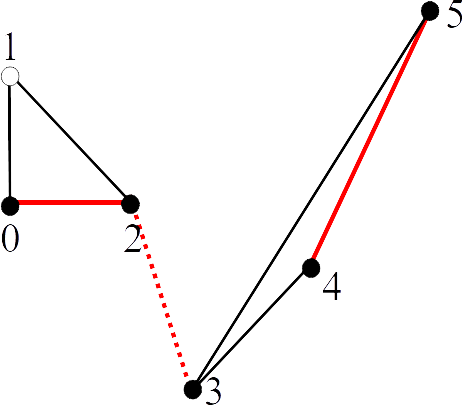} &   $(d-b)(c+d-4)$ & \\
&& \includegraphics[height=0.7 in]{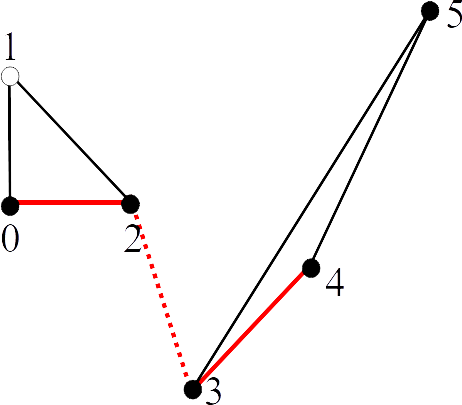}	& $b(c+d-4)$ &  \\
	\hline
    \end{tabular}
    \caption{The conjectured degree of the binodal locus depending on $a$ and $b$, possible lattice paths, and their conjectured multiplicities for polytopes in families 20 and 21. Polytopes from family 20 are only binodal when $a > 3$.}
    \label{tab:bigtab2}
\end{table}

\begin{table}[h]
    \centering
    \begin{tabular}{c|c|c|c|c|c}
		$b \setminus a $ & 3& 4& 5& 6& 7  \\ \hline
		1  &  8 & 20& 24& 56& 80\\
		&   3+5&6+14&6+18&12+44&15+65 \\
		\hline
		2  &  8 & x& 33& x& 60\\
		&   3+(4+1)&  &9+24& &12+48\\
		\hline
		3  &  x & 20& 33& x& 70  \\
		&   & 6+(10+4)& 9+(21+3)&  & 15+55   \\
		\hline
		4  &  x & x&  24&  x& 70  \\
		&    &  &  6+(12+6)&   & 15+(50+5)   \\
		\hline
		5  &  x & x& x& 56& 60 \\ 
		&    &  &  & 12+(28+16)& 12+(36+12) \\ 
		\hline
		6  &  x & x& x& x& 80   \\ 
		&   & & & & 15+(40+25)  \\ 
	\end{tabular}
    \caption{The degree of the binodal variety and the multiplicities of the lattice paths for 
    polytopes in family 8. These were computed using \texttt{Oscar}.}
    \label{tab:degreepolyA8}
\end{table}
\clearpage

\subsection{A complete binodal counting example}
Consider the polytope given by the convex hull of 
\begin{equation*}
    P =\begin{pmatrix}
    0&0&0&0&0& 1&1\\
    0&0&1&1&2& 1&4\\
    0&1&0&1&0& 0&1
    \end{pmatrix}.
\end{equation*}
This is a polytope from family 13 with $a=3,\,b=1$ plus an additional point. 
To enumerate all binodal surfaces with this support that satisfy our point conditions, we consider the possible floor plans and their multiplicities.
The first floor is a curve dual to a subdivision of $P_1= \conv\{ (0,0),(0,1),(1,0),(1,1),(2,0)\}$,
while the second floor is a line with direction vector orthogonal to the line spanned by  the points $(1,0)$ and $(4,1).$ 
Alignment of the second floor with a $4$-valent vertex of the first floor gives a node germ corresponding to a binodal polytope. There are three possible subdivisions of $P_1$ that give rise to such a vertex.
Below we investigate each case, draw the floor plan and the lattice paths, and compute the multiplicities.

A 4-valent vertex where each edge has weight 1 is dual to a parallelogram in the subdivision. The first option to include a parallelogram in $P_1$ is the square given by $(0,0,0)$, $(0,0,1)$, $(0,1,0)$, $(0,1,1)$.
 Figure \ref{fig:ltpt1} shows in one picture the three floor plans that are induced by this alignment. The different lattice paths encode the ways for the curves to satisfy the point conditions. The first two choices are very similar and are therefore displayed in the same floor plan. The gray points correspond to the first lattice path, the black points to the second lattice path. For clarity the third choice of point conditions is displayed in a separate floor plan.

  \begin{figure}[h]
     \centering
     \begin{subfigure}[b]{.48\textwidth}
         \centering
         \begin{subfigure}[b]{.6\textwidth}
         \includegraphics[width=\textwidth]{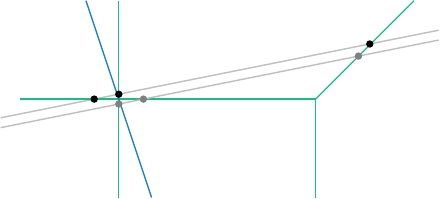}
         
               \includegraphics[width=0.48\textwidth]{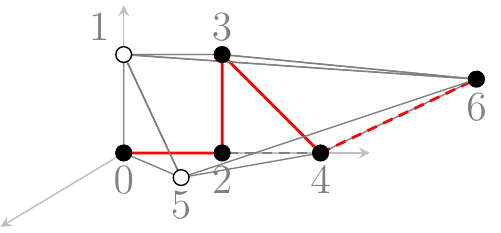}
             \includegraphics[width=0.48\textwidth]{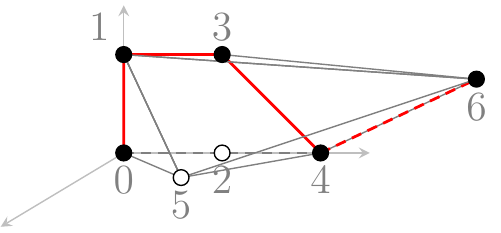}
             \end{subfigure}
             \begin{subfigure}[b]{.3\textwidth}
         \includegraphics[width=\textwidth]{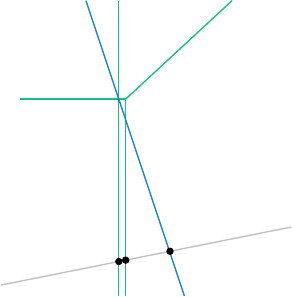}
             \includegraphics[width=0.95\textwidth]{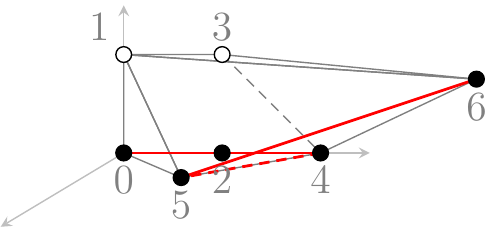}
             \end{subfigure}
         \caption{}\label{fig:ltpt1}
     \end{subfigure}
     \begin{subfigure}[b]{.29\textwidth}
         \centering
         \includegraphics[width=0.8\textwidth]{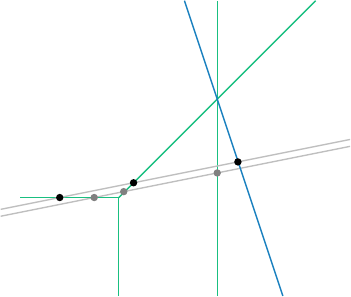}
         
         \includegraphics[width=0.48\textwidth]{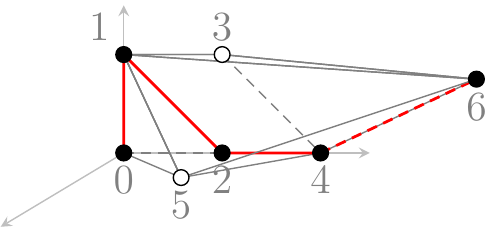}
             \includegraphics[width=0.48\textwidth]{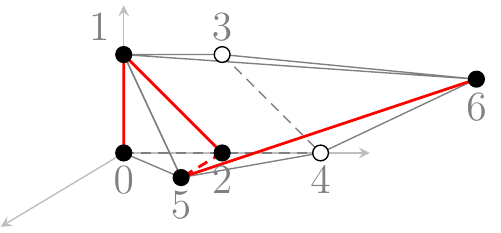}

         \caption{}\label{fig:ltpt2}
     \end{subfigure}
      \begin{subfigure}[b]{.19\textwidth}
         \centering
         \includegraphics[width=\textwidth]{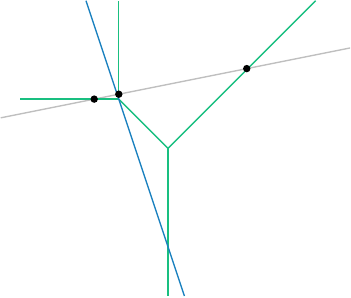}
         
                  \includegraphics[width=0.8\textwidth]{poly13+1-path2.pdf}
         \caption{}\label{fig:ltpt4}
     \end{subfigure}
     \caption{Floor plan and lattice paths for each alignment option.}
     \label{fig:13+1sub2}
 \end{figure}
 
 All three lattice paths give rise to the same subdivision consisting of two tetrahedra and a polytope equivalent to a polytope from family 10 with $a=3, b=1$:
\begin{equation*}
    \begin{pmatrix}
    0&0&0& 1\\
    1&1&2& 4\\
    0&1&0& 1
    \end{pmatrix}, \begin{pmatrix}
    0&0&1& 1\\
    1&2&1& 4\\
    0&0&0& 1
    \end{pmatrix} \text{ and }  \begin{pmatrix}
    0&0&0&0& 1&1\\
    0&0&1&1& 1&4\\
    0&1&0&1& 0&1
    \end{pmatrix}.
\end{equation*}
From the computations for family 10, as displayed in Table \ref{tab:bigtab1}, we see that these three paths each have multiplicity $(a-2) = 1$ (since $b=1$). These results can also be computed directly in \texttt{OSCAR} by using the functions in \cite{Code21}.

The second option for a 4-valent vertex with weights 1 in the first floor is given by the parallelogram $(0,0,1),(0,1,0),(0,1,1),(0,2,0)$. The corresponding alignment is depicted in Figure \ref{fig:ltpt2}. There are two options how the floor plan can satisfy the point conditions, giving two possible lattice paths. In the figure, gray points correspond to the first lattice path, the black points to the second lattice path. The difference is choosing the last or second to last intersection of the gray line with floor plan curves.

 Both lattice paths give rise to the same subdivision consisting of two tetrahedra and a polytope equivalent to a polytope from family 10 with $a=4, b=1$:
\begin{equation*}
    \begin{pmatrix}
    0&0&0& 1\\
    0&0&1& 1\\
    0&1&0& 0
    \end{pmatrix} \text{ and }  \begin{pmatrix}
    0&0&0&0& 1&1\\
    0&1&1&2& 1&4\\
    1&0&1&0& 0&1
    \end{pmatrix}.
\end{equation*}
Applying IUA-equivalence, we can bring the polytope with 6 lattice points in the form as in family $10$ and read the multiplicity off from Table \ref{tab:bigtab1}. We see that the left lattice path has multiplicity 2, while the right lattice path has multiplicity 8. So this alignment gives rise to two floor plans that together contribute 10 surfaces. Again this can be confirmed by a direct computation in \texttt{OSCAR}.

It is also possible to have two separated nodes. 
Considering the curves from the previous floor plan, we observe another alignment possibility. We can align the second floor with the 3-valent vertex, thus giving rise to a pentatope in the subdivision, which has multiplicity $3$ due to the shape of the pentatope, see \cite[Lemma 4.4]{MaMaSh18}. 
As the second floor also intersects the weight 2 vertical end, we obtain our second node which has multiplicity $2$, see \cite[Lemma 4.4]{MaMaSh18}.
This floor plan and the corresponding lattice path are shown in Figure \ref{fig:ltpt4}. The induced subdivision consists of a pentatope and two tetrahedra with an edge of length 2:
\begin{equation*}
    \begin{pmatrix}
    0&0&0&1& 1\\
    0&1&2&1& 4\\
    0&0&0&0& 1
    \end{pmatrix}, \begin{pmatrix}
    0&0&0&0& 1\\
    0&1&1&2& 4\\
    0&0&1&0& 1
    \end{pmatrix} \text{ and }  \begin{pmatrix}
    0&0&0&1&1\\
    0&0&1&1&4\\
    0&1&1&0&1
    \end{pmatrix}.
\end{equation*}
This case gives a total multiplicity of $6$.

 These are the only possibilities, which follows from an extensive investigation of lattice paths and subdivisions of the polytope $P$. So, in total we obtain $(1+1+1)+(2+8)+6=19$ surfaces. We can verify computationally that this is the correct number: The binodal variety of $P$ can be computed via \texttt{Singular} \cite{singular} or \texttt{OSCAR} \cite{oscar,Code21}, and it is of degree $19$.

\subsection{Binodal polytopes in degree $d$ surfaces}\label{subsec:degdsurface}

We now turn to the question of counting binodal degree $d$ surfaces. In \cite{BG20}, we found 214 of 280 binodal cubic surfaces passing through 17 general points. With the binodal polytopes on 6 vertices discussed above, can we add any more surfaces to this count? 

We first recover an observation from \cite[Section 5.1]{BG20} that we can now phrase as a definition.

\begin{definition}\label{def:doublerightstring}
We fix points in Mikhalkin position in $\mathbb{R}^3$ and consider their projections $q_i$ to the $yz$-plane. Let $C$ be a tropical plane curve passing through the $q_i$. If none of the points $q_i$ are contained in the edges dual to the edges of the lower right triangles in the subdivision of the Newton polytope of $C$, 
then we can both prolong the edge in direction $(2,1)$ and the edge in direction $(1,0)$. This produces a family of curves that all pass through the chosen points in Mikhalkin position. The union of the two ends together with the unfixed horizontal edge is called a \emph{double right string}. See Figure \ref{fig:doublestring}.

\begin{figure}
    \centering
    \includegraphics[width = 0.2\textwidth]{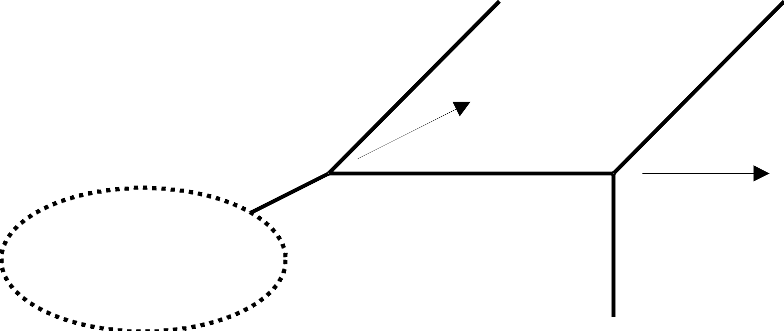}
    \caption{A double right string.}
    \label{fig:doublestring}
\end{figure}
\end{definition}
\begin{remark}
In accordance with the naming of the right and left strings from \cite{MaMaShSh19} we call it a double \emph{right} string, even though we will not look at the left analogue. This is because a ``left double string'' would have two degrees of freedom in the directions $(-1,-1)$ and $(-1,-2)$, for which the slopes are too steep to allow any curve fixing alignments.
\end{remark}

\begin{proposition}
\label{prop:541}
Let $\Gamma$ be a binodal surface of degree $d$ passing through points in Mikhalkin position.
Then the dual subdivision to $\Gamma$ cannot contain polytopes of families 8, 14 or 21. 
\end{proposition}
\begin{proof} We consider the polytope families separately.
\begin{itemize}
    \item Family 8: 
    A lattice path inducing a subdivision that contains a triangle with an interior point in one slice must skip at least 3 lattice points. This contradicts our point conditions.
    \item Family 14: We make a case distinction by the direction of the edge of length 2. Due to the points in Mikhalkin position, the edge can only be either vertical, horizontal or diagonal. Other directions are not possible, since they would require leaving out more points to make the alignment possible. Since we need an alignment between two floors (this is due to the fact that the polytope has a triangle in one slice and an edge in the other), and since we have to leave out at least one lattice point to obtain the edge of length 2, we  have only two options: Either the triangle with the edge of length 2 is missing another lattice point from the path, or we can use only  tools from the one-nodal case, i.e., right and left string, for the alignment.
    
 If the edge of length 2 is horizontal, it has to be contained in the boundary, so it is dual to a vertical weight 2 end. 
  In the slice $x=1$, the edge of the polytope has direction $(b,a)$ with $0<b<a$ with $gcd(a,b)=1$. This only fits between two columns if $b=1$.
 We can apply IUA-equivalences to obtain %
    $$\begin{pmatrix}
    0 & 0 & 0 & 0 & 1 & 1\\
    y-2 & y-1 & y-1 & y & 0& 1 \\
    0 & 0 & 1 & 0 &  a & 0
    \end{pmatrix}.$$
    This polytope can be contained in $d\Delta_3$ with $a+1\leq d-1$ and it does not immediately contradict our point conditions, since we can draw a lattice path leaving out the point $(0,y-1,0)$. If $y=d$, it is possible to draw a lattice path that leaves out $(0,d-1,0)$ and $(0,d,0).$
    In both cases, the vertex of the curve from which the vertical weight 2 end starts, will be below the line on which the points from the point conditions are distributed. That means that if the triangle has three points in the lattice path, an alignment can only be possible via a left string. However, considering the IUA-representative of the polytope above, we see that this is not possible for $1=b<a$. 
    If the triangle only has two points in the lattice path, it must be in the position $y=d$.
    In this case, the left out points in the path lead to a $Y-$formed string whose vertex needs to be aligned with the edge dual to the $x=1$ slice of the polytope above. 
    Due to the fact that the vertex that would be movable is below the line on which the points from the conditions are distributed, while the edge it has to align with would be relatively close to that line but with $y$-coordinate much larger, this is impossible.
    
    If the edge of length 2 is diagonal, it has to be dual to a diagonal end of weight 2. 
    If the triangle that contains the edge of length 2 has two edges in the lattice path, i.e., the midpoint of the edge of length two is the only lattice point from the triangle left out of the lattice path, the vertex  in the floor plan from which the diagonal end of weight 2 starts is fixed and lies above the line on which the points from the point conditions are distributed.
    This implies that an alignment is only possible via a right string. However, the end of the right string that could align with the vertex is the diagonal end, which thus has the same direction as the end of weight 2. This leads to a polytope of family 12 in the subdivision.
   It follows that the triangle needs to be the right tip of $d\Delta_3$. 
    We  apply IUA-equivalences to obtain 
     $$\begin{pmatrix}
    0 & 0 & 0 & 0 & 1 & 1\\
    d-2 & d-1 & d-1 & d & 0 & b\\
    2 & 0 & 1 & 0 & a+b & 0
    \end{pmatrix}.$$
    This implies $b=1.$ Since we have already determined that the lattice path needs to leave out two points in the triangle, we can move the vertex dual to the triangle in $(2,1)$-direction. As the vertex dual to teh triangle in the floor plan is above the line, on which the points from the point conditions are distributed, it can never be aligned with any of the bounded dual edges in the $x=1$ slice that we obtain for $b=1$, $a\geq4$ in the polytope above. 
    The argument is analogous to the case before. 
    
    If the edge of length 2 is vertical, it is dual to a horizontal edge of weight 2 in the floor plan. This edge can be bounded or unbounded depending on whether the edge is in the boundary of the Newton polytope. Applying IUA-equivalence we obtain two representatives
        $$\begin{pmatrix}
    0 & 0 & 0 & 0 & 1 & 1\\
    0 & 0 & 0 & 1 & 0 & a\\
    z-2&z-1&z & 0 & az-(a+b) & 0
    \end{pmatrix} \text{ or } \begin{pmatrix}
    0 &0 & 0 & 0 & 1 & 1\\
    0 & 1 & 1 & 1 & 0 & a\\
    z & 0 & 1 & 2 & b+a(d-1) & 0
    \end{pmatrix},$$
    where $z \leq d$.
    In both cases the edge in the  $x=1$-slice cannot be made to fit between two columns, so this case is not possible.
    \item Family 21: Polytopes of family 21 have no parallel edges. With just two lattice points that can be left out, it is impossible to fit two triangles into adjacent floors such that they have no parallel edges and such that the point conditions allow an alignment of their dual vertices. As long as the induced subdivision on the level of the floors is column-wise this follows directly. Leaving out only two lattice points we cannot fit a triangle in a floor such that it does not have either a vertical or a diagonal edge. In the adjacent floor the alignment needs to be done with a string or double string, depending on how many lattice points we can still leave out. However, the triangle dual to the vertex that aligns in a (double) string always has a diagonal and vertical edge. Thus, even if the point conditions allow an alignment, we always have parallel edges, so the emerging polytope will not belong to family 21.
\end{itemize}
\end{proof}

\begin{proposition}
\label{prop:poly10}
Let $\Gamma$ be a binodal surface of degree $d$ passing through points in Mikhalkin position.
Then, the dual subdivision to $\Gamma$ can only contain a polytope of family 10 if $d\geq 5.$ 
\end{proposition}

\begin{proof}
For $d<3$ a surface cannot be binodal. Let $P$ be a polytope of polytope family 10.

One slice of $P$ is a sloped edge from the point $(1,0,0)$ to the point $(1,a,b)$, where $a\geq b>0$. For any IUA-equivalence of this polytope to appear in a floor decomposed surface we need $b = 1$. 
There are two cases how the polytope can be contained in $d\Delta_3$:
the parallelogram from $P$ is either contained in a floor or between two floors.

If the parallelogram is contained between two floors, it comes from the alignment of two vertices of the curves of the two adjacent floors, which each have an edge of the same direction that align when the vertices align.
Moreover, when investigating the combinatorics of $P$, we observe that the two triangles in the two floors that form the polytope have only one pair of parallel edges and this is the pair that forms the parallelogram on one side.
The two edges have to be vertical, since the triangles dual to the aligning vertices are of minimal lattice volume, and we only obtain such triangles without a vertical edge if the lattice path leaves out a point on the diagonal boundary. But then, we do not have enough degrees of freedom left to obtain a binodal polytope of family 10.  

We can align two vertices of adjacent floors either by fixing one vertex and leaving the other vertex 2 degrees of freedom in a way that it can be moved to align with the first vertex, or by leaving one degree of freedom to both vertices.  If we leave out one lattice point in each of the two aligning triangles, at least one of the vertices will not be a string and thus, can only have limited area of movement as the surrounding vertices of the curve are fixed by the point conditions. The alignment of a left and right string does yield a polytope of family 9, which is not binodal.
Leaving two degrees of freedom to one vertex such that it can be moved to align with the other vertex is only possible with a double right string.
So, the IUA-equivalence representative of the polytope is
$$\begin{pmatrix}
0 & 0 & 0 & 1 & 1 & 1\\
d-1 & d-1 & d & y & y & y+1 \\
0 & 1 & 0 & a & a+1& 0
\end{pmatrix}, $$
where $d$ is the degree of the lower floor 
and $y$ is a parameter for the column in the higher floor. The generic version of the IUA-representative can be found in Table \ref{tab:Multiplicites}. Figure \ref{fig:poly10quintic1} shows this for $a=3$, $y=0$ and $d=5$. Since a polytope of family 10 with $b=1$ is only binodal if $a\geq3$, we can conclude $d\geq 5.$  
\begin{figure}[h]
    \centering
    \begin{subfigure}{0.35\textwidth}
        \centering
        \includegraphics[width=1\textwidth]{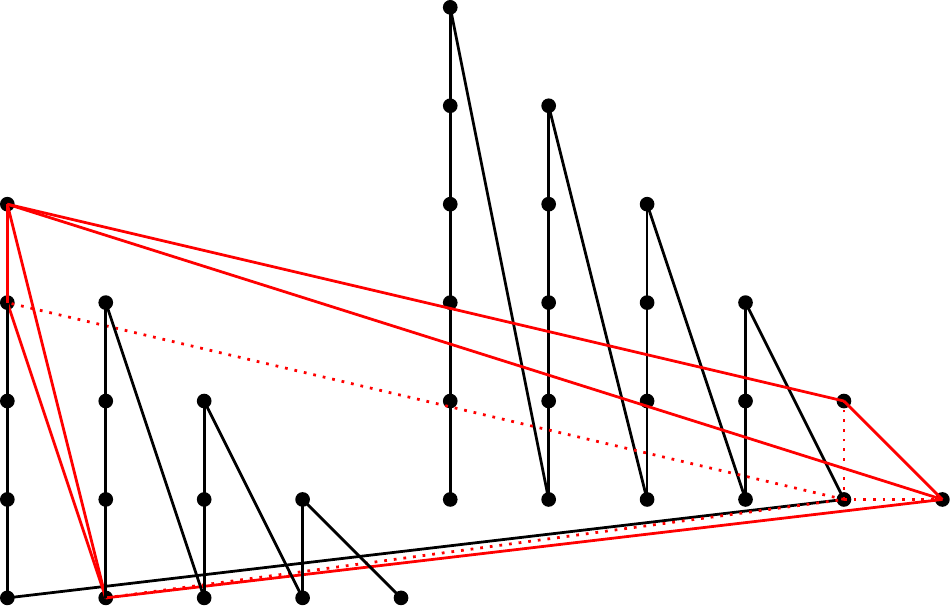}
    \end{subfigure}
    \begin{subfigure}{0.62\textwidth}
        \centering
        \includegraphics[width=1\textwidth]{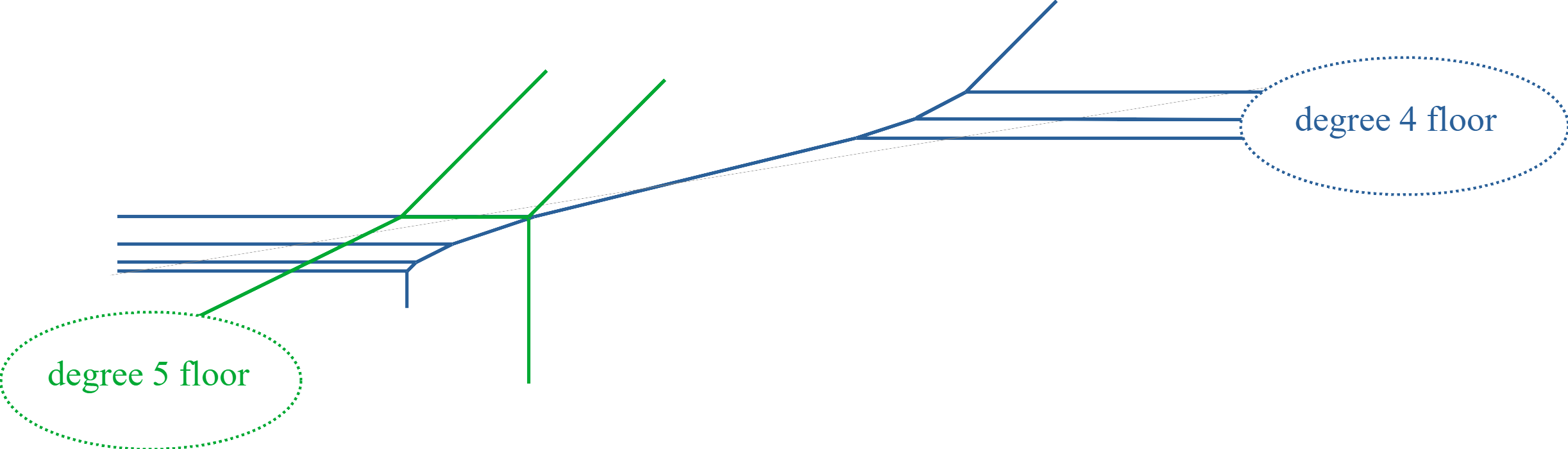}
    \end{subfigure}
    \caption{Polytope of family 10 in $5\Delta_3$}
    \label{fig:poly10quintic1}
\end{figure}

If the parallelogram is contained in one of the floors, it comes from the alignment of an edge of the other floor with $4$-valent vertex dual to the parallelogram. The parallelogram arises from a left out lattice point and the folding out of a corner. This means that we have one more lattice point that we can leave out to make the alignment possible. It follows  that the alignment has to happen via a left or right string, since an additional left out point in the parallelogram does not give enough freedom to the vertex to move the relatively far way towards the next floor. 
Hence, the edge in $P$ at $x=1$ level needs to be transformed to either the diagonal edge of the right string or the vertical edge of the left string. For the left string, we know that the edge opposite the parallelogram is vertical. Thus, due to the point conditions, the parallelogram must have either a pair of horizontal or a pair of diagonal edges. We obtain the following IUA-representative
$$\begin{pmatrix}
 0 & 0& 0 & 0 &  1 & 1 \\
  0 &1 &a  & a+1  & 0& 0\\
 1 & 1& 0 &  0 &  0 & 1
\end{pmatrix} \text{ or } \begin{pmatrix}
  0 & 0 &0 &0 &  1 & 1 \\
  0 & 1 &a  &a+1 &  0& 0\\
  2+a &a+1& 1 &0 &  0 & 1
\end{pmatrix}. $$
In both cases the parallelogram is in a position that cannot be contained in the subdivision with our choice of point conditions. 

It follows that the alignment must be due to a right string. Because of the point conditions and the fact that the edges of the parallelogram cannot be parallel to the edge dual to the diagonal end of the right string, it follows that the parallelogram has a pair of vertical parallel edges. Such a parallelogram appears from the folding out of an edge when a lattice point at the top or bottom of a column is left out.
There are two ways to leave out a lattice point in the lattice path to make a parallelogram appear. 
So, there are two IUA-equivalence representatives of the polytope:
$$
\begin{pmatrix}
0 & 0 & 1 & 1 & 1 & 1\\
d-1 & d & y & y & y+1 & y+1\\
 1 & 0 & d-2 & d-1 & d-a-3 & d-a-2
 \end{pmatrix} \text{ or }
\begin{pmatrix} 
0 & 0 & 1 & 1 & 1 & 1\\
d-1 & d & y & y & y+1 & y+1\\
1 & 0 & a+1 & a+2 & 0 & 1
\end{pmatrix}, $$
 where $y\in\mathbb{N}_0$ is a parameter for the column that contains the parallelogram. The generic representatives can be found in Table \ref{tab:Multiplicites}.
 
Since the polytopes from family 10 with $b=1$ are only binodal for $a\geq 3$, it follows that $d\geq 6$.
 Figure \ref{fig:poly10quintic2} shows this for $a=3$, $y=0$ and $d=6$.

\begin{figure}[h]
    \centering
    \begin{subfigure}{0.38\textwidth}
        \centering
        \includegraphics[width=0.8\textwidth]{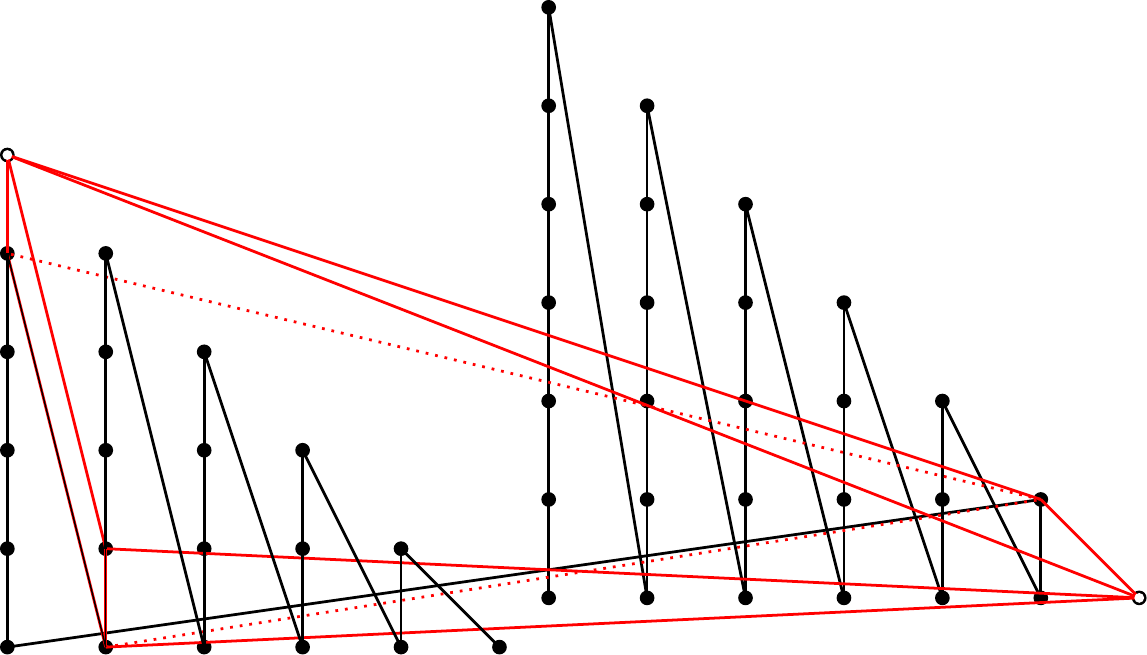}
    \end{subfigure}
    \begin{subfigure}{0.6\textwidth}
        \centering
        \includegraphics[width=1\textwidth]{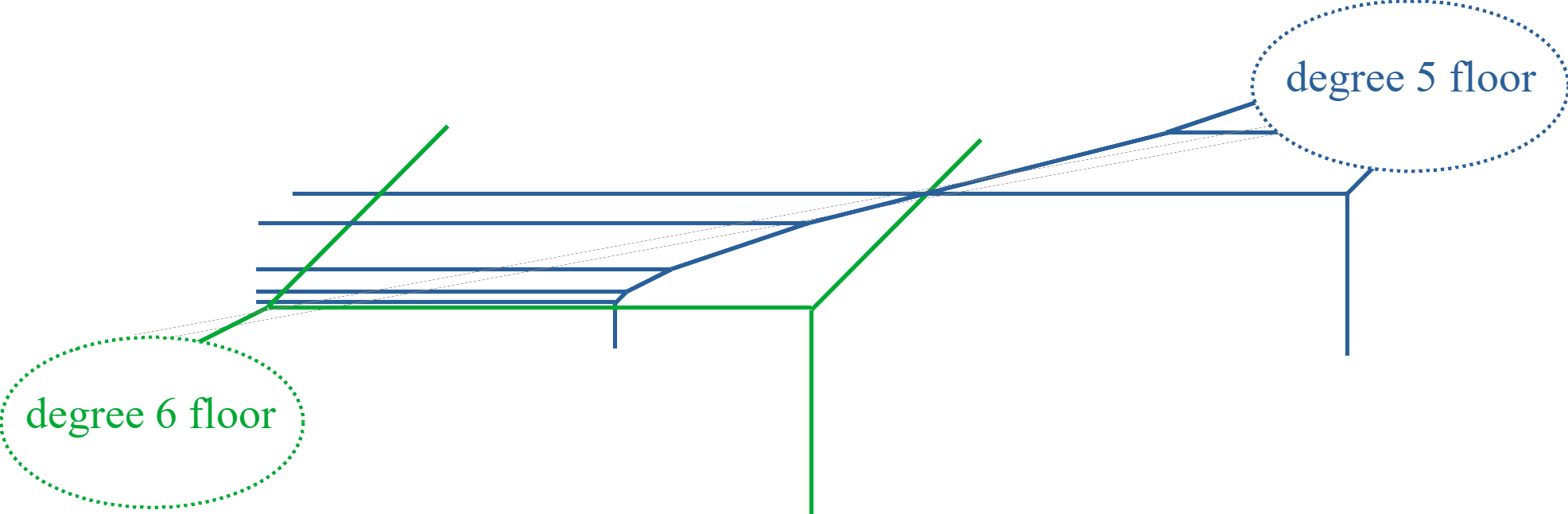}
    \end{subfigure}
    \caption{Polytope of family 10 in $6\Delta_3$. The lattice path can also follow the upper part of the parallelogram instead of the lower part, passing through $(1,0,5)$ and leaving out $(1,1,0)$.}
    \label{fig:poly10quintic2}
\end{figure}
\end{proof}

\begin{proposition}\label{prop:poly13}
Let $\Gamma$ be a binodal surface of degree $d$ passing through points in Mikhalkin position.
Then the dual subdivision to $\Gamma$ can only contain a polytope of family 13 if $d\geq 7.$
In this case the edge of length 2 has to be a vertical edge of the subdivision.
\end{proposition}
\begin{proof} 
 We make a case distinction by the direction of the edge of length 2. Due to the points in Mikhalkin position, the edge can only be either vertical, horizontal or diagonal; see the argument for the polytope family 14 in the proof of Proposition \ref{prop:541}. 
 If the edge of length 2 is horizontal, it has to be contained in the boundary, so it is dual to a vertical weight 2 end. In the slice $x=1$, the edge of the polytope has direction $(a,b)$ with $0<b<a$ with $gcd(a,b)=1$. This never fits between two columns.
 
    If the edge of length 2 is diagonal, it has to be dual to a diagonal end of weight 2. If the triangle that contains the edge of length 2 has two edges in the lattice path, i.e., the midpoint of the edge of length two is the only lattice point from the triangle left out of the lattice path, the vertex in the floor plan from which the diagonal end of weight 2 starts is fixed and lies above the line on which the points from the point conditions are distributed. The same argument as for polytope family 14 in the proof of Proposition \ref{prop:541} applies, and it follows that this case is not possible.
    Thus, the triangle needs to be the right tip of $d\Delta_3$.  
    We  apply IUA-equivalences to obtain $$\begin{pmatrix}
    0 & 0 & 0 & 0 & 1 & 1\\
    d-2 & d-1 & d-1 & d & 0 & a\\
    2 & 0 & 1 & 0 & a+b & 0
    \end{pmatrix}.$$ 
    Due to the point conditions, the edge in the $x=1$ slice can never be part of the subdivision, so this case is not possible. 
    
    If the edge of length 2 is vertical, it is dual to a horizontal edge of weight 2 in the floor plan. This edge can be bounded or unbounded depending on whether the edge is in the boundary of the Newton polytope. Considering the polytope we observe that the $x=1$ floor can only be contained between two columns if $b=1$. Not being contained between two columns would contradict the point conditions.
    
    We first exclude the case that the lattice path contains only two lattice points from the triangle that contains the edge of length 2. In this case the edge of length 2 is at the boundary of $d\Delta_3$. Since it is a vertical edge, this means it has $y$-coordinate zero. It follows that the edge either has $z$ coordinates $0,1,2$ (we say it is at the bottom) or the maximal values possible in the corresponding floor (we say it is at the top). 
     If it is at the top, the vertex dual to the triangle is above the line containing the points from the conditions. Moreover, we can only move the vertex along the direction of the edge dual to the bottom edge of the triangle and the movement is restricted due to the surrounding point conditions. 
     This does not allow enough freedom to align the vertex with an edge of an adjacent floor.
    If the edge of length two is at the bottom, the vertex dual to the triangle is below the line containing the points from the conditions. Moreover, we can only move the vertex along the direction $(-2,-1)$. Since the slope of the line containing the points from the point conditions is very small and the points are distributed with growing distances it is not possible to align the vertex with an edge of the adjacent floors.
    
    It follows that the lattice path contains all three vertices of the triangle containing the edge of length 2. This means that an alignment is only possible via a right or left string. For a left string the end aligning with the triangle has the same direction as the weight 2 edge and this leads to a polytope from family 11, which is not binodal. It follows that the alignment happens via a right string. 
    
    With this alignment, the relative positions of 4 vertices are fixed, and after applying IUA-equivalence we obtain two representatives for the containment in $d\Delta_3$: 
    $$\begin{pmatrix}
    1 & 1 & 1 & 1 & 0 & 0\\
    1 & 0 & 0 & 0 & d-1 & d\\
0&a&a+1 & a+2 & 1 & 0
    \end{pmatrix} \text{ or } 
    \begin{pmatrix}
    1 & 1 & 1 & 1 & 0 & 0\\
    0 & 1 & 1 & 1 & d-1 & d\\
d-1 & d-a-3 & d-a-2 & d-a-1 & 1 & 0
    \end{pmatrix},
  $$
    where $a$ is the parameter of the polytope family. When $b=1$, this is only binodal for $4\leq a$. 
    The positions above are exemplary: the same containments can be achieved for different slices ($x$-value) and different columns ($y$-value), as long as the minimum of the absolute values of the slopes of the non-vertical edges of the volume 2 triangle is at least $4$.
    For both polytopes we conclude that $d\geq 7.$ 
 \begin{figure}
    \centering
    \begin{subfigure}{1\textwidth}
        \centering
   
    \begin{subfigure}{0.42\textwidth}
        \centering
        \includegraphics[width=0.8\textwidth]{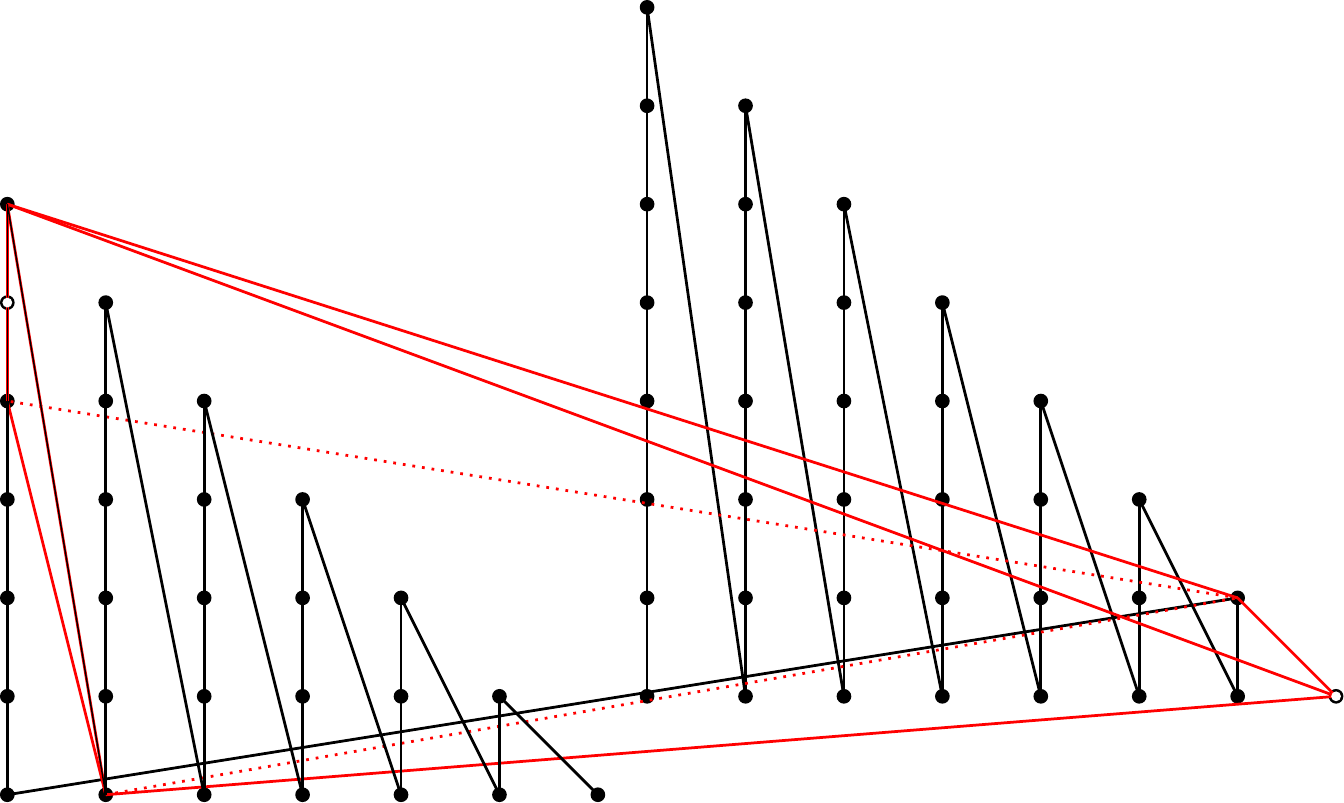}
    \end{subfigure}
    \begin{subfigure}{0.55\textwidth}
        \centering
        \includegraphics[width=1\textwidth]{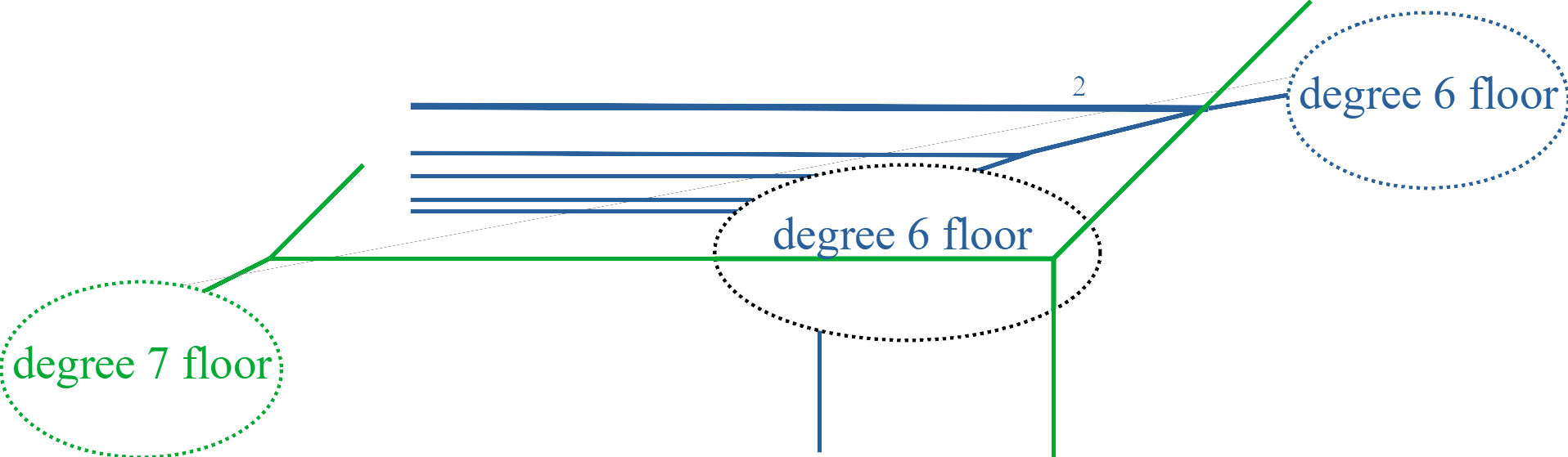}
    \end{subfigure}
    \caption{Polytope of family 13 in $7\Delta_3$ first option}
    \label{fig:poly13septic1}
 \end{subfigure}
    \centering
    \begin{subfigure}{1\textwidth}
    \centering
    \begin{subfigure}{0.42\textwidth}
        \centering
        \includegraphics[width=0.8\textwidth]{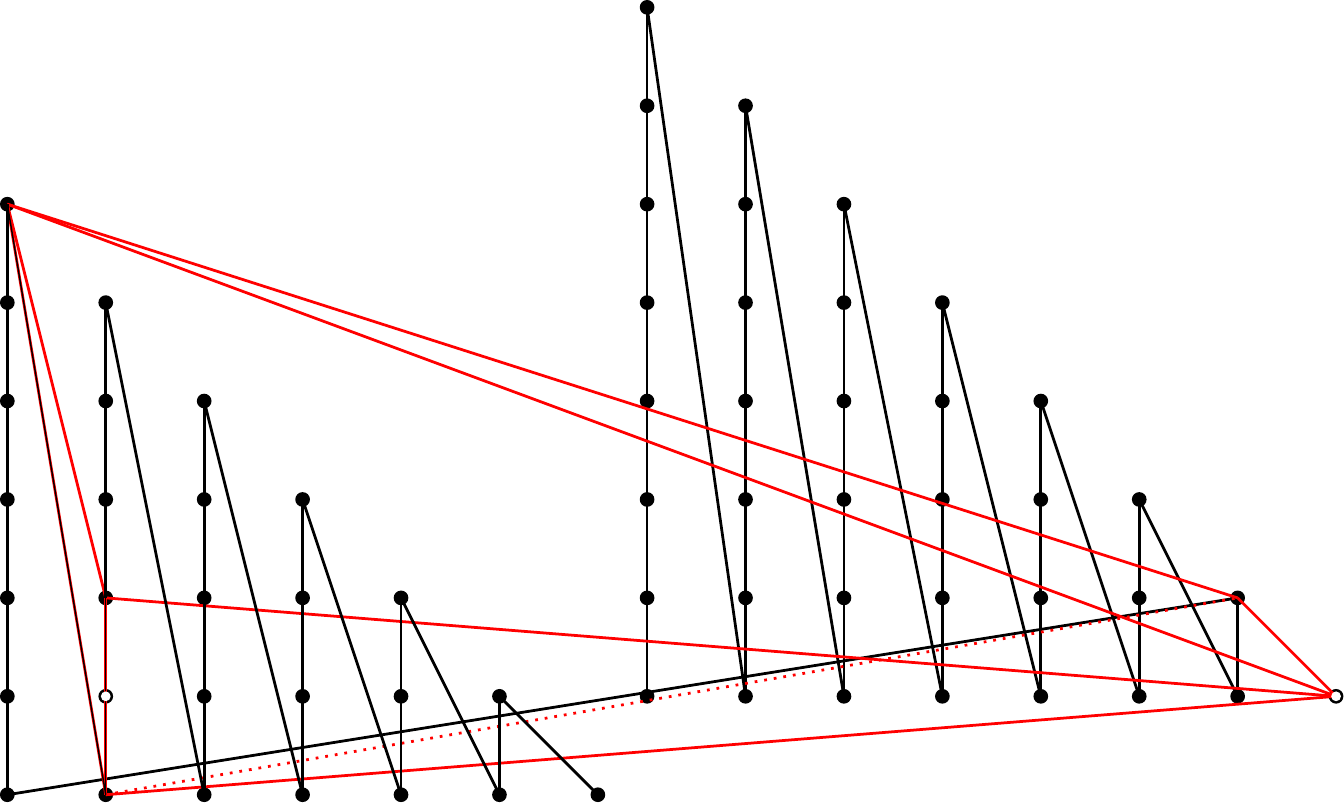}
    \end{subfigure}
    \begin{subfigure}{0.55\textwidth}
        \centering
        \includegraphics[width=1\textwidth]{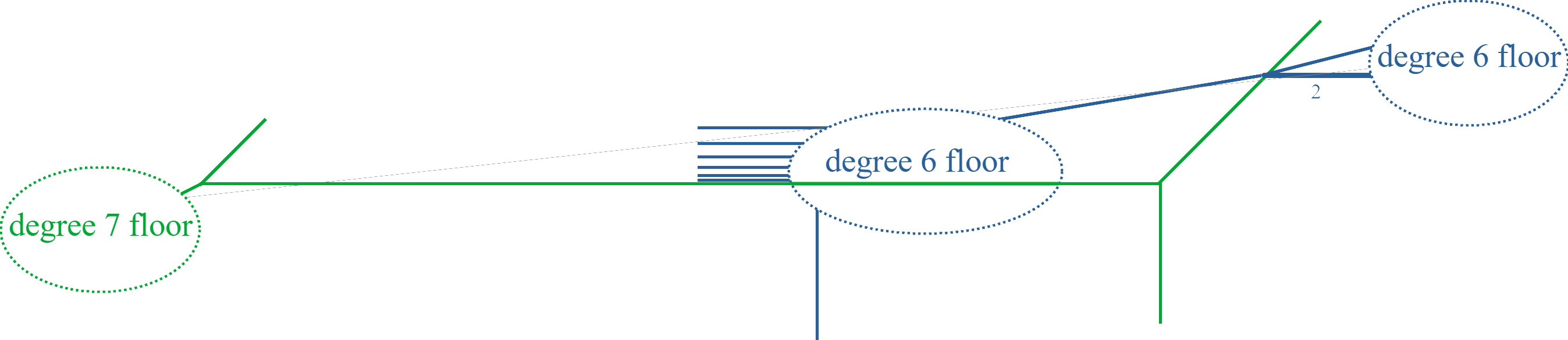}
    \end{subfigure}
    \caption{Polytope of family 13 in $7\Delta_3$ second option}
    \label{fig:poly13septic2}
    \end{subfigure}
    \caption{There are two ways how Polytope $13$ can be contained in $7\Delta_3$.}
\end{figure}
\end{proof}

\begin{proposition}
\label{prop:poly20}
Let $\Gamma$ be a binodal surface of degree $d$ passing through points in Mikhalkin position.
Then, the dual subdivision to $\Gamma$ can only contain a polytope of family 20 if $d\geq 5.$ 
\end{proposition}

 \begin{proof}
For $d<3$ a surface cannot be binodal. Let $P$ be a polytope of family 20.

 Only one pair of the edges of $P$ is parallel. Since we assume that the subdivision is sliced, it follows that the two parallel edges have to be in different, though adjacent floors. If they were in the same floor, the other two lattice points of the polytope would have to be distributed on the two adjacent floors. The polytope would then have width 2 in the $x$-direction. 
 
 Since $P$ has no parallelogram as a face, the polytope comes from aligning two vertices of adjacent floors. The two parallel edges have to be vertical, since  
 it is not possible to make two triangles in adjacent floors that have exactly one pair of parallel edges that are not vertical and such that their dual vertices can be aligned in the floor plan.
 We can align two vertices of adjacent floors either by  fixing one vertex entirely while the other has 2 degrees of freedom such that it can be aligned with the fixed vertex, or by leaving one degree of freedom to each vertex such that both can be moved to enable the alignment.
The first approach gives a double right string.
This leads us to the following IUA-equivalence representative of $P$ for the containment in $d\Delta_3$:
$$\begin{pmatrix}
0 & 0 & 0 & 1 & 1 & 1 \\
d-1 & d-1 & d & y & y+1 & y+1\\
0 & 1 & 0 & d-y-1 & d-a-1 & d-a
\end{pmatrix}, $$
where
 $a$ is the parameter of the polytope family and $y$ is the parameter the position of the second triangle in the $y$-coordinate. The generic representatives can be found in Table \ref{tab:Multiplicites}.
 Since polytopes of family  20 are only binodal for $a\geq 4$, it follows that $d\geq 5.$
Figure \ref{fig:poly20quintic} depicts this containment and the dual floor plan alignment for $d=5,y=z=0$. 

If we leave out one lattice point in each of the two triangles, at least one of the vertices will only have a limited area for movement, since the surrounding vertices are fixed by the point conditions. This is only not true for left and right strings, but an alignment of a left with a right string does not induce a polytope of family 20.  
\end{proof}

\begin{figure}
    \centering
    \begin{subfigure}{0.35\textwidth}
        \centering
        \includegraphics[width=1\textwidth]{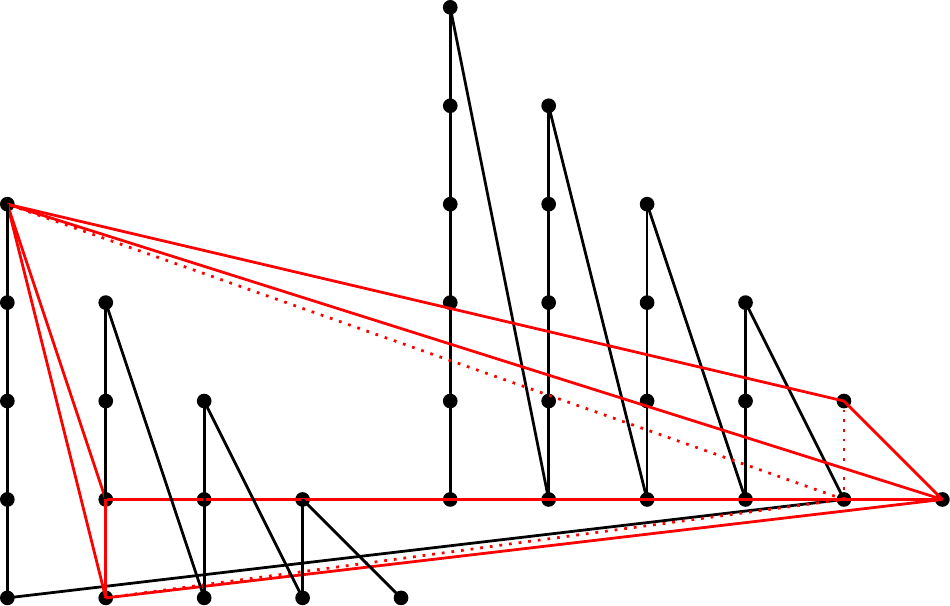}
    \end{subfigure}
    \begin{subfigure}{0.62\textwidth}
        \centering
        \includegraphics[width=1\textwidth]{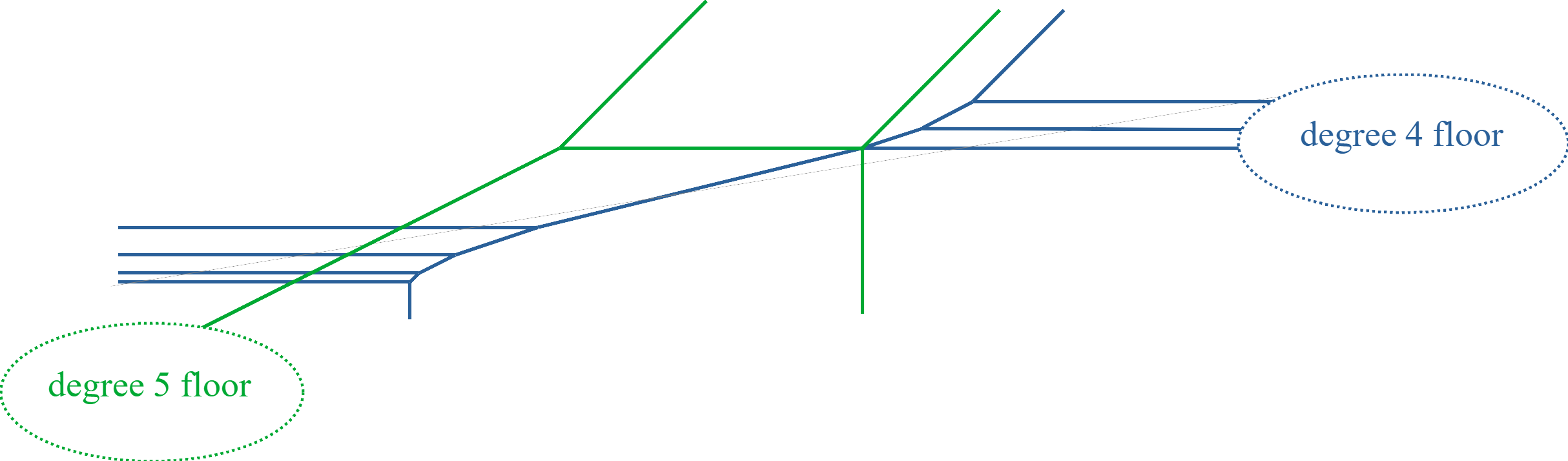}
    \end{subfigure}
    \caption{Polytope of family 20 in $5\Delta_3$}
    \label{fig:poly20quintic}
\end{figure}

\begin{theorem}\label{thm:binodal}
Let $X$ be a binodal surface of degree $d$ passing through points in Mikhalkin position. The two nodes of the surface $X$ can only tropicalize to a vertex dual to a polytope with 6 lattice points if $d >4.$ In this case the two nodes are located at the vertex dual to a polytope of family 10, 13 or 20, see Figure~\ref{fig:polys}.
\end{theorem}

\begin{proof}
This follows from Propositions \ref{prop:poly10}, \ref{prop:poly13} and \ref{prop:poly20}.
\end{proof}
\begin{conjecture}\label{conj:mult}
Consider a surface of degree $d$ passing through points in Mikhalkin position such that the dual subdivision contains one of the binodal polytopes from one of the families 10, 13, or 20, and only unimodular simplices everywhere else. Then the multiplicity is given in Table \ref{tab:Multiplicites}.
\begin{table}[h]
    \centering
    \begin{tabular}{p{1.4cm}|p{13 cm}|p{1.5cm}}
    Polytope &  \centering{Position and lattice path} &  Mult.\\
        No. 10\newline $a\geq 3$, \newline $b=1$, \newline$d\geq 5$ & 
        \begin{minipage}{\linewidth}
        \begin{center}
         $\begin{pmatrix}
        A & B & C & D & E & F\\
        d-(f+1) & d-(f+1) & d-(f+1) & d-f & d-f & d-f\\
f & f & f+1 & y & y & y+1 \\
0 & 1 & 0 & a & a+1& 0
        \end{pmatrix}$ \\ 
        Path: $A-D-E-F$ 
        \vspace{-0.4 in}
        \end{center}
        \end{minipage}
        &   $a-2$\\
        \hline
          No. 10\newline $a\geq 3$, \newline $b=1$,\newline $d\geq 6$ & 
                  \begin{minipage}{\linewidth}
        \begin{center}
          $\begin{pmatrix} A & B & C & D & E & F\\
d-(f+1) & d-(f+1) & d-f & d-f & d-f & d-f\\
f & f+1 & y & y & y+1 & y+1\\
1 & 0 & f-y-1 & f-y & f-y-a-2 & f-y-a-1
\end{pmatrix} $ \\
 Path: $A-C-E-F$
    \vspace{-0.4 in}
        \end{center}
        \end{minipage}
& $a-2$\\
         \hline
         No. 10 \newline$a\geq 3$,\newline $b=1$,\newline $d\geq 6$ & 
             \begin{minipage}{\linewidth}
        \begin{center}  
         $\begin{pmatrix} A & B & C & D & E & F\\
d-(f+1) & d-(f+1) & d-f & d-f & d-f & d-f\\
f & f+1 & y & y & y+1 & y+1\\
1 & 0 & a+1 & a+2 & 0 & 1
\end{pmatrix} $\\
Path: $A-C-D-F$
   \vspace{-0.4 in}
        \end{center}
        \end{minipage}
& $a-2$\\
         \hline
        No. 13 \newline$a\geq 4$,\newline $b=1$,\newline $d\geq 7$ & 
                     \begin{minipage}{\linewidth}
        \begin{center}  
        $\begin{pmatrix} A & B & C & D & E & F\\
    d-(f+1) & d-(f+1)   & d-f & d-f & d-f & d-f\\
    f & f+1 & y & y & y &y+1 \\
1 & 0 &a&a+1 & a+2 & 0 
    \end{pmatrix}  $\\
    Path: $A-C-E-F$
       \vspace{-0.4 in}
        \end{center}
        \end{minipage}
    &  $a-2$ \newline if $a$ even. \newline $a-3$ \newline if $a$ odd.\\
    \hline 
      No. 13 \newline$a\geq 4$,\newline $b=1$,\newline  $d\geq 7$ &
  \begin{minipage}{\linewidth}
        \begin{center}
    $\begin{pmatrix} A & B & C & D & E & F\\
   d-(f+1) & d-(f+1) & d-f  & d-f  & d-f  & d-f  \\
 f& f+1 & y & y+1 & y+1 & y+1 \\
  1 & 0 &f-y & f-y-a-2 & f-y-a-1 & f-y-a 
    \end{pmatrix}
  $ \\ 
  Path: $A-C-D-F$
         \vspace{-0.4 in}
        \end{center}
        \end{minipage}
  &  $a-2$ \newline if $a$ even. \newline $a-3$ \newline if $a$ odd.\\
         \hline
        No. 20\newline $a\geq 4$, \newline $d\geq 5$ & 
  \begin{minipage}{\linewidth}
        \begin{center}
$\begin{pmatrix} A & B & C & D & E & F\\
d-(f+1) & d-(f+1) & d-(f+1) & d-f & d-f & d-f \\
f & f & f+1 & y & y+1 & y+1\\
0 & 1 & 0 & f-y & f-y-a & f-y-a+1
\end{pmatrix} $\\ 
Path: $A-D-E-F$
        \end{center}
        \end{minipage}
&  $a-3$\\
    \end{tabular}
    \caption{The multiplicities of the binodal polytopes appearing in degree $d$ surfaces through points in Mikhalkin position. The parameters are to be read like this: $f$ is the degree of the floor containing the main part of the polytope, $y$ is the value that can be obtained in the $y$-coordinate and $a$ is the parameter of the polytope family.}
    \label{tab:Multiplicites}
\end{table}
\end{conjecture}
Conjecture \ref{conj:mult} is verified for polytope family 13 when $4 \leq a \leq 10$, for family 20 when $a \leq 7$, and for family 10 when $3 \leq a \leq 7$.
\medskip

We use these results to extend the definition of tropical floor plans. For that purpose we adapt some of the definitions from \cite{MaMaShSh19} to define floor plans that enumerate tropical surfaces with a binodal polytope in their dual subdivision. First, we note some differences from the original definition.

We will count a double right string as two node germs to  assure that two node germs in the tropical floor plan correspond to two nodes in the associated surface. 

Unlike the original definition of node germs by \cite{MaMaShSh19}, we let a bounded edge of length two be a node germ, instead of just its midpoint. 
We do this because for polytopes from family 13, the edge of weight 2 in the tropical floor plan contributes to the two nodes via one of its vertices, not via its midpoint.

\begin{definition}[Generalised from \cite{MaMaShSh19,BG20}]\label{def:floorplan-new}
		Let $Q_i$ be the projection of $q_i$ along the $x$-axis. A \emph{$\delta$-nodal floor plan $F$ of degree $d$} is a tuple $(C_d,\ldots,C_1)$ of plane tropical curves $C_i$ of degree $i$ together with a choice of indices $d\geq i_{\delta'}>\ldots>  i_1\geq1$ each assigned a natural number $k_j$  
		such that $\sum_{j=1}^{\delta'}k_j=\delta$, where $0<\delta'\leq\delta$,  satisfying:
	\begin{enumerate}
		\item The curve $C_i$ passes through the following points: (we set $i_0=0, i_{\delta+1}=d+1.$)\begin{align*}
			\text{if } i_{\nu}>i>i_{\nu-1}:\,\,\,\,\,\,\,\, &Q_{\sum_{k=i+1}^d \binom{k+2}{2}-\delta+(\sum_{j: i > i_j} k_j)+1},...,Q_{\sum_{k=i}^d \binom{k+2}{2}-\delta+(\sum_{j: i > i_j} k_j)-1}, \\
			\text{if } i=i_{\nu}:\,\,\,\,\,\,\,\,\,\,\,\, &Q_{\sum_{k=i+1}^d \binom{k+2}{2}-\delta+(\sum_{j: i \geq i_j} k_j)+1},...,Q_{\sum_{k=i}^d \binom{k+2}{2}-\delta+(\sum_{j: i > i_j} k_j)-1}.
		\end{align*}
		\item The plane curves $C_{i_j}$ has $k_j$ node germs for each $j=1,\ldots,\delta',$ where  the double right string counts for two.
		\item If  $C_{i_j}$  contains a left string as a node germ, then its horizontal end aligns either with a horizontal bounded edge of $C_{i_j+1}$ or with a $3$-valent (where edges are counted with multiplicity) vertex of $C_{i+1}$ not adjacent to a horizontal edge.
		\item  If $C_{i_j}$ contains a right string as a node germ, then its diagonal end aligns either
		\begin{itemize} 
		\item with a diagonal bounded edge of $C_{i_j-1}$,
		\item with a $3$-valent (where edges are counted with multiplicity) vertex of $C_{i_j-1}$ which is not adjacent to a diagonal edge,
		\item with a vertex dual to a parallelogram which has two vertical edges and two edges of slope of absolute value at least $4$, or
		\item with a vertex dual to a triangle which consists of one vertical edge of length two and two edges with slope of absolute value at least $4$.
		\end{itemize}
		\item If $C_{i_j}$ contains a double right string, then the vertex adjacent to the diagonal and vertical end of the double right string aligns with a vertex dual to a triangle that is formed of a vertical edge and two edges with slope of absolute value at least $3$.
		\item If $i_{d}=\delta'$, then the node germs of $C_d$ can only be diagonal ends of weight two, a right string, or a double right string.
		\item If $i_1=1$, then the node germ of $C_1$ is a left string.
	\end{enumerate}
\end{definition}

To generalize the definition of the multiplicity of a floor plan, we distinguish between node germs appearing in separated mode and unseparated mode.
\begin{definition}
A node germ appears in \emph{separated} mode if the corresponding singularity is separated. 
Otherwise, we say the node germ appears in \emph{unseparated} mode.
By definition a double right string is always in unseparated mode. All other node germs in unseparated mode are accompanied by a second node germ in unseparated mode with which it interacts, e.g. by alignment. These node germs in unseparated mode with be always given as a pair.
\end{definition}

The definition of separated or unseparated modes of node germs allows us to distinguish between node germs giving rise to binodal polytopes in the dual subdivision (unseparated mode) and node germs giving rise to the polytope complexes described in connection with the circuits in Section \ref{sec:preliminaries} (separated mode).

The following definition is based on Conjecture \ref{conj:mult}.

\begin{definition}[Generalised from Definition 5.4, \cite{MaMaShSh19}]\label{def:multiplicities-new}
Let $F$ be a $\delta$-nodal floor plan of degree $d$. Let $C^{*}_{i_j}$ be a node germ of $C_{i_j}$ in separated mode. Then we
define the following local complex multiplicity $\text{mult}_{\mathbb{C}}(C^{*}_{i_j})$:

\begin{enumerate}
\item If $C^{*}_{i_j}$ is dual to a parallelogram, then $\text{mult}_{\mathbb{C}}(C^{*}_{i_j}) =2$.
\item If $C^{*}_{i_j}$ is the midpoint of an edge of weight two, then $\text{mult}_{\mathbb{C}}(C^{*}_{i_j}) =8$.
\item If  $C^{*}_{i_j}$ is a horizontal end of weight two, then $\text{mult}_{\mathbb{C}}(C^{*}_{i_j}) = 2(i_j + 1)$.
\item  If $C^{*}_{i_j}$ is a diagonal end of weight two, then $\text{mult}_{\mathbb{C}}(C^{*}_{i_j}) = 2(i_j - 1)$.
\item If $C^{*}_{i_j}$ is a left string whose horizontal end aligns with a horizontal bounded edge, then $\text{mult}_{\mathbb{C}}(C^{*}_{i_j}) = 2$.
\item If $C^{*}_{i_j}$ is a left string whose horizontal end aligns with a vertex not adjacent to a horizontal edge, then $\text{mult}_{\mathbb{C}}(C^{*}_{i_j}) = 1$.
\item  If $C^{*}_{i_j}$ is a right string whose diagonal end aligns with a diagonal bounded edge, then $\text{mult}_{\mathbb{C}}(C^{*}_{i_j}) = 2$.
\item If $C^{*}_{i_j}$ is a right string whose diagonal end aligns with a vertex not adjacent to a diagonal edge, then $\text{mult}_{\mathbb{C}}(C^{*}_{i_j})= 1$.

\end{enumerate}
If $(C^{*}_{i_j},C^{*}_{i_j-1})$ is a pair of node germs in unseparated mode, where $C^{*}_{i_j-1}$ is the node germ in $C_{i_j-1}$ associated to $C^{*}_{i_j}$, or $C^{*}_{i_j}$ is a double right string, we assign the following local multiplicities 
\begin{enumerate}
    \item[(9)] If $C^{*}_{i_j}$ is a double right string that aligns with a vertex dual to a triangle that is formed of a vertical edge and two edges with slope $|a|\geq 3$ 
with the vertical edge on the left side of the triangle, it has multiplicity $|a|-2.$ Otherwise it has multiplicity $|a|-3.$

    \item[(10)] If $(C^{*}_{i_j},C^{*}_{i_j-1})$ consists of a right string and an edge of weight 2 dual to the vertical edge of length two of a triangle for which the minimum of the absolute value of the slope of the other two edges is $|a|\geq 4$, the pair has multiplicity  $|a|-2$ if $|a|$ is even, and $|a|-3$ if $|a|$ is odd.
    
    \item[(11)] If $(C^{*}_{i_j},C^{*}_{i_j-1})$ consists of a right string and a parallelogram which has two vertical edges and two edges of slope $|a| \geq 4$, the pair has multiplicity $|a|-2.$

\end{enumerate}

The multiplicity of a $\delta$-nodal floor plan  $F$ is \begin{align*}
    \text{mult}_{\mathbb{C}}(F) = 
\prod_{j=1}^{\delta'}&\prod_{\text{node germs in}\atop \text{separated mode of } C_{i_j}}\text{mult}_{\mathbb{C}}(C^{*}_{i_j}) \\ & \cdot \prod_{\text{double right strings of }C^{*}_{i_j}} \text{mult}_{\mathbb{C}}(C^{*}_{i_j})
\cdot\prod_{\text{node germs in unseparated}\atop \text{ mode of } (C_{i_j},C_{i_j-1}) } \text{mult}_{\mathbb{C}}( (C^{*}_{i_j},C^{*}_{i_j-1})).
\end{align*}

\end{definition}

As final result, we now use these multiplicities to compute the contribution of tropical surfaces with binodal polytopes in their dual subdivision to the count of binodal surfaces of degree $d$.

\begin{theorem}
\label{conj:main}
Assuming Conjecture \ref{conj:mult} is true, it follows that tropical degree $d$ surfaces with a binodal polytope with $6$ vertices and width $1$ in the dual subdivision contribute $\frac{1}{4}d^4+\mathcal{O}(d^3)$ surfaces to the count of binodal degree $d$ surfaces $N_{2,\mathbb{C}}^{\mathbb{P}^3}(d)$. So, they contribute to the third highest term of the polynomial  $N_{2,\mathbb{C}}^{\mathbb{P}^3}(d)$. 
\end{theorem}
\begin{proof}
From Propositions \ref{prop:541}, \ref{prop:poly10},  \ref{prop:poly13}  and \ref{prop:poly20} we know that the only polytopes that can contribute are from family 10, family 13 and family 20.

For polytopes of family 10 there were two different ways of including the polytope in $d \Delta_3$. We consider them separately.
First, we consider the case that the parallelogram is contained between two floors.
We have to count with multiplicity all the ways how the triangle $(x+1,y,a)$, $(x+1,y,a+1)$, $(x+1,y+1, 0)$ can be contained in $d\Delta_3.$ We know it can only occur in floors $F$ of degree $4\leq deg(F)<d$. For such a floor $F$, there are $deg(F)-3$ many columns that can accommodate the triangle. 
If we fix the left most column of floor $F$, we see that there are $deg(F)-3$ options to fit in the triangle via
$(d-deg(F),0,a)$, $(d-deg(F),0,a+1)$, $(d-deg(F),1, 0)$ with $a\geq 3$. For an arbitrary column that is high enough, we have to fit in
$(d-deg(F),y,a)$, $(d-deg(F),y,a+1)$, $(d-deg(F),y+1, 0)$ ensuring $a\geq 3$. This yields $deg(F)-y-3$ many options. We count each option with multiplicity $a-2$.
This yields a total contribution of 
$$\sum_{f = 4}^{d-1}\sum_{y=0}^{f-4}\sum_{a=3}^{f-y-1}  (a-2) = \frac{d^4}{24}-\frac{5 d^3}{12} + \frac{35 d^2}{24} - \frac{25 d}{12} +1.
$$

The next inclusion option is the parallelogram contained in a floor. We know there are two options to get the parallelogram in the subdivision, depending on which lattice point we leave out. However, these two options are symmetric, and the corresponding paths give the same type of multiplicity. So, we have to count only the case where the left out point is the top point of the parallelogram.

We have to count with multiplicity all the ways how the parallelogram $(x+1,y,d-2)$, $(x+1,y,d-1)$, $(x+1,y+1,d-a-3)$, $(x+1,y+1,d-a-2)$ can be contained in $d\Delta_3.$ We know it can only occur in floors $F$ of degree $5\leq deg(F)<d$. For such a floor $F$, there are $deg(F)-4$ many columns that can accommodate the parallelogram. 
If we fix the left most column of floor $F$, we see that there are $deg(F)-4$ options to fit in the triangle via  $(d-deg(F),0,deg(F)-1)$, $(d-deg(F),0,deg(F))$, $(d-deg(F),1,deg(F)-a-3)$, $(d-deg(F),1,deg(F)-a-2)$ ensuring $a\geq 3$. For an arbitrary column that is high enough we have to fit in
$(d-deg(F),y,deg(F)-1)$, $(d-deg(F),y,deg(F))$, $(d-deg(F),y+1,deg(F)-a-3)$, $(d-deg(F),y+1,deg(F)-a-2)$ ensuring $a\geq 3$. This yields $deg(F)-y-3$ many options. We count each option with multiplicity $a-2$.
This yields a total contribution of 
$$\sum_{f = 5}^{d-1}\sum_{y=1}^{f-4}\sum_{a=3}^{f-y-1}  (a-2) = \frac{d^4}{24}-\frac{7 d^3}{12} + \frac{71 d^2}{24} - \frac{77 d}{12} +5.$$
So for all lattice paths connected to polytopes of family 10 we obtain a contribution of $$\frac{d^4}{8}-\frac{19 d^3}{12} + \frac{59 d^2}{8} - \frac{179 d}{12} +11.$$

For polytopes of family 13 there were two different ways of including the polytope. However, the two cases are symmetric and they both give the same multiplicity, so we only need to count one.
We have to count with multiplicity all the ways how the triangle $(x+1,y,a)$, $(x+1,y,a+1)$, $(x+1,y,a+2)$, $(x+1,y+1,0)$ can be contained in $d\Delta_3.$ We know it can only occur in floors $F$ of degree $6\leq deg(F)<d$. For such a floor $F$, there are $deg(F)-5$ many columns that can accommodate the triangle. 
If we fix the left most column of floor $F$, we see that there are $deg(F)-5$ options to fit in the triangle via  $(d-deg(F),0,a)$, $(d-deg(F),0,a+1)$, $(d-deg(F),0,a+2)$, $(d-deg(F),1,0)$ ensuring $a\geq 4$. For an arbitrary column that is high enough, we have to fit in $(d-deg(F),y,a)$, $(d-deg(F),y,a+1)$, $(d-deg(F),y,a+2)$, $(d-deg(F),y+1,0)$ ensuring $a\geq4$, this yields $deg(F)-y-6$ many options. We count each option with its multiplicity: $a-2$  if $a$ is even, and $a-3$ if $a$ is odd. Writing $a=2k$ resp. $a=2k+1$ we see that in both cases the multiplicity is $2k$.
We obtain the following contribution to our count:
\begin{align*}
		\frac{1}{12}d^4 - \frac{7}{6}d^3 + \frac{17}{3}d^2 - \frac{34}{3}d + 8\hspace{1cm}
		& \text{ if } d \text{ even, }\\
		\frac{1}{12}d^4 - d^3 + \frac{25}{6}d^2 - 7d - \frac{17}{4}
	\hspace{1cm}	& \text{ if } d \text{ odd. }
	\end{align*}

For polytopes of family 20 we have to count with multiplicity all the ways how the triangle $(1,y,d-y-1)$, $(1,y+1,z)$, $(1,y+1, z+1)$ can be contained in $d\Delta_3.$ We know it can only occur in floors $F$ of degree $4\leq deg(F)<d$. For such a floor $F$ there are $deg(F)-3$ many columns that can accommodate the triangle. 
If we fix the left most column of floor $F$, we see that there are $deg(F)-3$ options to fit in the triangle via $(d-deg(F),0,deg(F))$, $(d-deg(F),1,z)$, $(d-deg(F),1, z+1)$ ensuring $a=deg(F)-z\geq 4$. For an arbitrary column that is high enough, we have to fit in $(d-deg(F),y,deg(F)-y)$, $(d-deg(F),y+1,z)$, $(d-deg(F),y+1, z+1)$ ensuring $a=deg(F)-y-z\geq 4$. This yields $f-y-3$ many options.   We count each option with multiplicity $a-3 = deg(F)-y-z-3$.
This yields a total contribution of 
$$\sum_{f = 4}^{d-1}\sum_{y=0}^{f-4}\sum_{z=0}^{f-y-4}  (f-y-z-3) = \frac{d^4}{24}-\frac{5 d^3}{12}+\frac{35 d^2}{24}-\frac{25 d}{12}+1.$$

It follows that for $d>6$ binodal polytopes contribute 
\begin{align*}
			\frac{1}{4}d^4 - 	\frac{19}{6}d^3 + 	\frac{29}{2}d^2 - 	\frac{85}{3}d + 20\hspace{1cm}&\text{ if } d \text{ even, }\\
			 & \\
			\frac{1}{4}d^4 - 	3d^3 + 13d^2 - 	24 d - 	\frac{31}{4}\hspace{1cm}
			&\text{ if } d \text{ odd. }
		\end{align*}
		to the count of binodal degree $d$ surfaces.
\end{proof}

\begin{remark} There are still many open questions connected to the tropical count of multinodal surfaces. We list a few open questions and directions of research.
\begin{enumerate}
\item So far we cannot compute the second order term of $N_{\delta,\mathbb{C}}^{\mathbb{P}^3}(d)$ by tropical methods. The surfaces contributing to this term probably correspond to unseparated nodes encoded in subdivisions of binodal polytopes with more lattice points, or arise from binodal polytopes of larger width, i.e. are non-floor decomposed surfaces.
    \item We are still missing $66$ surfaces from the count of binodal cubic surfaces.
These surfaces must contain binodal polytopes with 7 or more vertices in their dual subdivisions. These polytopes could have width 1 or width greater than 1, in which case we expect them to  be subdivided into width 1 polytopes. 
This question extends to understanding unseparated nodes in more generality.
\item At this time, we do not know how to prove Conjecture \ref{conj:mult}. This would require a technique for symbolic computation with Groebner bases for polynomial ideals with parameters in the exponents, or other methods.
\item In this paper, we do not investigate the contribution of the binodal polytopes determined in this paper with respect to a count over $\mathbb{R}$, but this could be an interesting direction of future research.
\item 
Initial computations indicate that the binodal polytopes studied in this paper are also Newton polytopes of cuspidal surfaces. For these polytopes, the cusp would tropicalize to the vertex of the tropical surface. So, these polytopes could contribute to a count of cuspidal tropical surfaces using methods similar to the ones used in this paper.
\end{enumerate}
We leave these questions for further research.
\end{remark}

\clearpage
\appendix
\section{Details of proofs}\label{appendix}
We include here computational details for some of the proofs of section \ref{sec:binodalpolytopes}.

\subsection{Proof of Proposition \ref{prop:eliminatepolytopes}}\label{app:prop:eliminatepolytopes}

\subsubsection*{Polytope Family 9}
Since polytope family 9 consists only of one polytope, we refer to it as Polytope~9.
The binodal locus of Polytope 9 consists of surfaces with non-isolated singularities, which we can see as follows.
Using \texttt{Singular} \cite{singular} or the \texttt{OSCAR} \cite{oscar} code available at \cite{Code21}, we can compute the polynomials generating the binodal variety in the polynomial ring of coefficients $\mathbb{Q}[a_1..a_6]$.
We obtain the three polynomials
$$
	g_1=a_4a_5-a_2a_6, \ \ \ \ 
	g_2=a_3a_5-a_1a_6, \ \ \ \ 
	g_3=a_2a_3-a_1a_4.
$$
Let $f=a_1+a_2x+a_3y+a_4xy+a_5z+a_6yz$ be a polynomial defining a binodal surface. Then we know that the coefficients are zeroes of the three polynomials $g_1, g_2, g_3$ and the coefficients are all nonzero, so using the above equations, we can write
$$
a_4 f = (a_2+a_4y)(a_3+a_3x+a_6z).
$$
So for any $f$ with Polytope 9 as  Newton polytope, the variety of $f$ contains non-isolated singularities.

\subsubsection*{Polytope Family 16}
The polynomial $p \in \mathbb{C}[a_1,\ldots, a_6][x,y,z]$ describing a generic surface with Newton polytope given by a polytope in family 16 is
$$
   p=a_1+a_2y+a_3y^2+x(a_4+a_5y^bz^a+a_6y^{2b}z^{2a}) .
$$
Its derivatives are given by
\begin{align*}
   \partial p/\partial x &= a_4+a_5y^bz^a+a_6y^{2b}z^{2a},\\
   \partial p/\partial y &= a_2+2a_3y+bxy^{b-1}z^a(a_5+2a_6y^{b}z^{a}),\\
   \partial p/\partial z &= axy^bz^{a-1}(a_5+2a_6y^{b}z^{a}).
\end{align*}
Consider the variety $X$ whose points are singular surfaces with Newton polytope in family 16 (in coordinates $a_1, \ldots, a_6$) together with a singular point on the surface (in coordinates $x,y,z$). The equations above vanish on $X$.

Writing $f=a_1+a_2y+a_3y^2$ and $g=a_4+a_5y^bz^a+a_6y^{2b}z^{2a}$ we can rephrase $p=f+xg$. Since  $\partial p/\partial x =g$, it follows that $f=0$ on the variety $X$.

If $V(p)$ is a binodal surface, we know that $x,y$ and $z$ cannot be zero as follows.
If $y$ or $z$ is zero, this would force some $a_i$ to be zero, so this cannot occur. If $x=0$, then $y = -a_2/2a_3$. Then, the equation for $\partial p / \partial x$ yields $2a$ possibilities for $z$, so there are more than two singular points, a contradtiction.

Setting $\Tilde{z}=y^az^b$ allows us to rewrite $g = a_4+a_5\Tilde{z}+ a_6\Tilde{z}^2.$ Moreover, $\partial g / \partial \Tilde{z} = a_5+2a_6\Tilde{z} = a_5+2a_6y^az^b.$ 
By the polytope family we know that $a>0$. It follows from $\partial p/\partial z=0$ that $a_5+2a_6y^{b}z^{a}=0$, since $x,y,z$ are not zero.
So it follows that $g$ as a univariate polynomial in $\Tilde{z}$ has a double zero, so $a_5^2-4a_4a_6$ vanishes on $X$. Hence, we can write $g = \gamma(y^az^b-\epsilon)^2$ for $\gamma,\epsilon\in\mathbb{C}[a_1,\ldots, a_6]$.

Furthermore, we can conclude from $\partial p / \partial y$ and the above that $\partial f /\partial y= a_2+2a_3y=0$, which means that $f$ as a univariate polynomial in $y$ as a double zero. So $a_2-4a_1a_3$ vanishes on $X$, and we can write $f=\beta (y-\alpha)^2$ with $\alpha,\beta \in \mathbb{C}[a_1,\ldots, a_6].$

Thus, we can rewrite $p$ and its derivatives as
\begin{align*}
	p&= \beta (y-\alpha)^2 + \gamma x(y^az^b-\epsilon)^2,\\
	\partial p /\partial x &= \gamma (y^az^b-\epsilon)^2,\\
	\partial p /\partial y &= 2\beta (y-\alpha)+ 2\gamma a xy^{a-1}z^b (y^az^b-\epsilon),\\
	\partial p /\partial z &= 2\gamma b xy^{a}z^{b-1} (y^az^b-\epsilon).
\end{align*}
It follows that all $(x,y,z)$ satisfying
\begin{align*}
    y^az^b&=\epsilon,\\
    y&=\alpha
\end{align*}
are singularities of $p$. This is an infinite family. Therefore, every polytope in family 16 has non-isolated singularities.

\subsubsection*{Polytope Family 17}
The polynomial $p \in \mathbb{C}[a_1, \ldots, a_6][x,y,z]$ describing a generic surface with Newton polytope given by a polytope in family 17 is
$$
p=a_1+a_2y+a_3z+x(a_4+a_5y^{a}z^{b}+a_6y^{2a}z^{2b}).
$$
Its derivatives are given by
\begin{align*}
   \partial p/\partial x &=a_4+a_5y^{a}z^{b}+a_6y^{2a}z^{2b},\\
   \partial p/\partial y &= a_2+axy^{a-1}z^{b}(a_5+2a_6y^{a}z^{b}),  \\
   \partial p/\partial z &= a_3+bxy^{a}z^{b-1}(a_5+2a_6y^{a}z^{b}). 
\end{align*}

Consider the variety $X$ whose points correspond to singular surfaces with Newton polytope in family 17, together with a singular point on the surface. The above equations vanish on $X$.

Setting $f = a_1+a_2y+a_3z$ and $g= a_4+a_5y^{a}z^{b}+a_6y^{2a}z^{2b}$ we can write $p = f+xg$.
Since $g = \partial p /\partial x,$ we can conclude that $f$ vanishes on $X$. 

By considering $\partial p/\partial y -az \partial p/\partial z$ we can conclude that $y =
\frac{a_3}{a_2}\frac{a}{b}z$ on $X$. We can divide by $a_2$ since we assume all $a_i$ to be non-zero. 
Substituting this into $f$, we can solve for $z$ and obtain $z = -\frac{a_1}{a_3(1+\frac{a}{b})}.$
Thus, we have unique solutions for $y$ and $z$.

Now, $\partial p / \partial y$ and $\partial p /\partial z$ are linear in $x$. Under the assumption that $a_5+2a_6y^{a}z^{b}\neq 0$, we obtain a unique solution for $x$ from both of these equations. Since $V(p)$ should have two singularities, this is a contradiction independent of whether $\partial p / \partial y$ and $\partial p /\partial z$ yield the same solution for $x$.

Thus, $a_5+2a_6y^{a}z^{b} = 0,$ which means that $g$ has a double zero when considered as a univariate polynomial in $w = y^{a}z^{b}$. So $a_5^2=4a_4a_6.$
Moreover, this implies $a_2=a_3=0,$ which cannot occur. Hence, any polytope of family 17 is not binodal.

\subsection{Proof of Lemma \ref{lem:nodiscon20}}\label{app:lem:nodiscon}
Recall, that we use the notation $\cdot^\vee$ to indicate dual objects under the duality between the subdivision of the Newton polytope (which in our case is trivial) and the tropical surface $S$ and the 3-dimensional areas of $\mathbb{R}^3\setminus S$. In particular, for a vertex $V$ of the polytope $V^\vee$ denotes the 3-dimensional area of $\mathbb{R}^3\setminus S$ dual to $V$. 
\subsubsection{Proof of Lemma \ref{lem:nodiscon20} for polytope family 20}\label{app:poly20}

 \begin{proof} We first consider the possible disconnected paths. There can only exist disconnected lattice paths if the line $L$, on which the points from the point conditions are distributed, intersects with the tropical surface more than 3 times, i.e., if the line passes through more than four of the 3-dimensional areas into which the tropical surface subdivides $\mathbb{R}^3.$
 
  For polytopes of family 20, we have the vertices $A=(0,1,0)$, $B=(0,1,1),$ $C=(0,2,0),$ $D=(1,0,a),$ $E=(1,1,0),$ $F =(1,1,1)$. The 3-dimensional areas defined by the tropical surface dual to this kind of polytope with vertex of the surface at $(0,0,0)$ are given by the dual areas to the vertices:
  \begin{align*}
      A^\vee \text{ defined by } &  (-1,0,0), (0,0,-1), (-1,-1,0), (0,-a,-1). \\
      B^\vee \text{ defined by } & 
      (-1,0,0),(-1,-1,0),(-a,1,1).\\
      C^\vee \text{ defined by }  & (-1,0,0),(0,0,-1),(a-2,a-1,1),(-a,1,1),(1,1,0).\\
      D^\vee \text{ defined by }  & (1,0,0),(-1,-1,0),(0,-a,-1),(-a,1,1),(a-2,a-1,1).\\
      E^\vee \text{ defined by }  & (1,0,0),(0,0,-1),(0,-a,-1),(1,1,0). \\
      F^\vee \text{ defined by }  & (1,1,0),(1,0,0),(a-2,a-1,1).
  \end{align*}
  
  Recall, that the line $L$ has direction vector $(1,\eta,\eta^2)$ with $0<\eta\ll 1$.

  We observe that $B^\vee\subset \{z>0\}$ and $E^\vee\subset \{z<0\}$. Thus, it immediately follows that if our line $L$ passes through $B^\vee$ it cannot pass through $E^\vee$ afterwards. 
  If $L$ passes through $A^\vee, C^\vee$ and $D^\vee$, then we know that $L$ passes in $D^\vee$ through $D^\vee\cap\{y>0\}\subset\{z>0\}$,  since  $C^\vee\subset \{y>0\}$. Then the line $L$ cannot pass through $E^\vee\subset\{z<0\}$. So, there are only disconnected lattice paths if $L$ intersects $A^\vee,B^\vee,C^\vee,D^\vee$ and $F^\vee$.

This yields four possibilities. Note, that leaving out the first or last point leads to a connected path. Each of the four points of $L$ intersected with the surface can be chosen to be not marked as a point from the point conditions.

We first consider the cases where one of the middle points is left out.
Let the general points be $(0,0,0)$, $(1,\eta,\eta^2)$, $(\lambda,\lambda\eta,\lambda\eta^2)$ and let $(x,y,z)$ be the vertex of the tropical surface dual to a polytope from family 20 that passes through these 3 points. Suppose the surface passes through the points according to the disconnected lattice path $A-B | C-D-F$. Asserting that the points are contained in the correct cells of the surface, we obtain the following system of 9 equations in 9 unknowns:
\begin{align*}
    \begin{pmatrix}
    0\\0\\0
    \end{pmatrix} &= \begin{pmatrix}
    x\\y\\z
    \end{pmatrix} + \alpha_1\begin{pmatrix}
    -1\\0\\0
    \end{pmatrix}+ \alpha_2\begin{pmatrix}
    -1\\-1\\0
    \end{pmatrix} , \,\, \alpha_1,\alpha_2 >0,\\
    \begin{pmatrix}
    1\\\eta\\\eta^2
    \end{pmatrix} &= \begin{pmatrix}
    x\\y\\z
    \end{pmatrix} + \beta_1\begin{pmatrix}
    2-a\\1\\1
    \end{pmatrix}+ \beta_2\begin{pmatrix}
    a-2\\a-1\\1
    \end{pmatrix}  , \,\,\beta_1,\beta_2 >0,\\
    \begin{pmatrix}
     \lambda\\\lambda\eta\\\lambda\eta^2
    \end{pmatrix} &= \begin{pmatrix}
    x\\y\\z
    \end{pmatrix} + \gamma_1\begin{pmatrix}
    a-2\\a-1\\1
    \end{pmatrix}+ \gamma_2\begin{pmatrix}
    1\\0\\0
    \end{pmatrix} , \,\, \gamma_1,\gamma_2 >0.
\end{align*}
When solving for the positive parameter $\beta_2$, we obtain an upper bound on the value of $\lambda$. This contradicts the genericity of the point conditions. 
For the disconnected lattice path $A-B-C|D-F$ we obtain a similar set of equations and an analogous contradiction.

Now consider the two connected paths where the first or last point are left out. The path $A|B-C-D-F$ is indeed possible as the system of equations for this path is solvable. For the path $A-B-C-D|F$ we obtain the following system of equations:

\begin{align*}
\begin{pmatrix}
0\\0\\0
\end{pmatrix} &= \begin{pmatrix}
x\\y\\z
\end{pmatrix} +\alpha_1\begin{pmatrix}
-1\\0\\0
\end{pmatrix}+\alpha_2\begin{pmatrix}
-1\\-1\\0
\end{pmatrix},\, \text{ with }\alpha_1,\alpha_2 >0,\\
\begin{pmatrix}
1\\\eta\\\eta^2
\end{pmatrix} &= \begin{pmatrix}
x\\y\\z
\end{pmatrix} +\beta_1\begin{pmatrix}
-1\\0\\0
\end{pmatrix}+\beta_2\begin{pmatrix}
2-a\\1\\1
\end{pmatrix},\, \text{ with }\beta_1,\beta_2 >0,\\
\begin{pmatrix}
\lambda\\\lambda\eta\\\lambda\eta^2
\end{pmatrix} &= \begin{pmatrix}
x\\y\\z
\end{pmatrix} +\gamma_1\begin{pmatrix}
2-a\\1\\1
\end{pmatrix}+\gamma_2\begin{pmatrix}
a-2\\a-1\\1
\end{pmatrix},\, \text{ with }\gamma_1,\gamma_2 >0.
\end{align*}
This gives the following solution for $\gamma_1$:
\begin{align*}
\gamma_1&= \lambda\eta^2-\frac{(\eta-\eta^2)(\lambda -1)}{a-2}.
\end{align*}
Then $\gamma_1>0$ if and only if $\lambda+1<a$.
This contradicts the genericity conditions of our points. It follows that $A|B-C-D-F$ is the only possible connected lattice path where the line $L$ has $4$ intersection points with the surface.

It remains to check the connected paths arising from a line $L$ that intersects the surface 3 times. Taking into account that we already know that $L$ cannot pass through both $B^\vee$ and $E^\vee$, we obtain the following possible paths: 
$
A-B-C-F,\ 
A-B-D-F,\ 
A-C-D-F,\ 
A-C-E-F,\ 
A-D-E-F.
$ 
For the paths $A-B-C-F$, $A-B-D-F$, $A-C-E-F$, and $A-D-E-F$ we can set up the equation system the same way as for the disconnected paths. Doing so, we find out that each of these paths is possible.

For the path $A - C - D - F$ we obtain the following equations: 

\begin{align*}
    \begin{pmatrix}
    0\\
    0\\
    0
    \end{pmatrix} &= \begin{pmatrix}
    x \\y\\   z
    \end{pmatrix} + \alpha_1\begin{pmatrix}
    -1\\0\\0
    \end{pmatrix}+ \alpha_2\begin{pmatrix} 0\\0\\-1\end{pmatrix}, \,\, \alpha_1,\alpha_2>0,\\
    \begin{pmatrix}
    1\\
    \eta \\
    \eta^2
    \end{pmatrix} &= \begin{pmatrix}
    x \\
    y\\
    z
    \end{pmatrix} +
    \beta_1\begin{pmatrix}
    2-a\\1\\1
    \end{pmatrix}+ \beta_2\begin{pmatrix} a-2\\a-1\\1\end{pmatrix}, \,\, \beta_1,\beta_2>0,\\
    \\
    \begin{pmatrix}
    \lambda\\
   \lambda \eta \\
    \lambda\eta^2
    \end{pmatrix} &= \begin{pmatrix}
    x \\
    y\\
    z
    \end{pmatrix}+
    \gamma_1\begin{pmatrix} a-2\\a-1\\1\end{pmatrix}+ \gamma_2\begin{pmatrix}
    1\\0\\0
    \end{pmatrix}, \,\, \gamma_1,\gamma_2>0.
\end{align*}
The first equation implies that $y=0$ and $x,z>0.$
The last equation however implies that $z = \lambda\eta(\eta-\frac{1}{a-1}).$ Since $\eta < \frac{1}{a}<\frac{1}{a-1}$, it follows that $z<0,$ a contradiction. So, this path is not possible.

 \end{proof}

\subsubsection{Proof of Lemma \ref{lem:nodiscon20} for polytope family 21}

 \begin{proof} We first consider disconnected paths. 
 These can only exist if the line $L$, on which the points from the point conditions are distributed, intersects with the tropical surface more than 3 times, i.e., if the line passes through more than four of the 3-dimensional areas into which the tropical surface subdivides $\mathbb{R}^3.$
 
  For polytopes of family 21 we have the vertices $A=(0,0,0)$, $B=(0,0,1),$ $C=(0,1,0),$ $D=(1,0,0),$ $E=(1,a,b),$ $F =(1,c,d)$. The 3-dimensional areas defined by the tropical surface dual to such a polytope with vertex of the surface at $(0,0,0)$ are given by the dual areas to the vertices:
  \begin{align*}
      A^\vee \text{ defined by } &  (-1,0,0),(0,-1,0),(0,0,-1).\\
      B^\vee \text{ defined by } & 
      (-1,0,0),(0,-1,0),(1-(d+c),1,1),(c,-d,c). \\
      C^\vee \text{ defined by }  & (-1,0,0),(0,0,-1),(d-b-1,d-b,a-c),(1-(d+c),1,1),(b,b,-a).\\
      D^\vee \text{ defined by }  & (0,-1,0),(0,0,-1),(1,0,0),(b,b,-a),(c,-d,c).\\
      E^\vee \text{ defined by }  & (1,0,0),(b,b,-a),(d-b-1,d-b,a-c). \\
      F^\vee \text{ defined by }  & (c,-d,c),(1-(d+c),1,1),(1,0,0),(d-b-1,d-b,a-c).
  \end{align*}
    Recall, that the line $L$ has direction vector $(1,\eta,\eta^2)$ with $0<\eta\ll 1$.

  We observe that $B^\vee\subset \{z>0\}$ and $E^\vee\subset \{z<0\}$. Thus, it immediately follows that if our line $L$ passes through $B^\vee$ it cannot pass through $E^\vee$ afterwards. Since $B^\vee$ is contained in $\{z>0\}$ and $C^\vee$ is contained in $\{y>0\}$ and $D^\vee\cap\{z>0,y>0\}=\emptyset$, it is not possible for $L$ to pass through $B^\vee$, $C^\vee$ and $D^\vee$.
So, the only possibility for a disconnected path is when the intersections of $L$ are distributed on $A-C-D-E-F$.

This gives rise to four possible lattice paths:
$
     A | C-D-E-F,\
     A-C | D-E-F,\
     A-C-D | E-F\
     A-C-D-E | F,
$
 where the first and last path are connected paths.
As in proof for polytope family 20 in \ref{app:poly20}, we consider an analogous set of equations for each path and investigate these systems for contradictions.

The system of equations we obtain for $A\,|\, C-D-E-F$ does not lead to contradictions, so this path is possible for all polytopes of family $21$.

Consider the lattice path $A-C\,|\,D-E-F$. We obtain the following system of equations:
\begin{align*}
    \begin{pmatrix}
     0\\0\\0
    \end{pmatrix}&=\begin{pmatrix}
    x\\y\\z
    \end{pmatrix}+\alpha_1 \begin{pmatrix}-1\\0\\0\end{pmatrix} +\alpha_2\begin{pmatrix}0\\0\\-1\end{pmatrix} \text{ with } \alpha_1,\alpha_2>0,\\
    \begin{pmatrix}
     1\\ \eta\\ \eta^2
    \end{pmatrix}&=\begin{pmatrix}
    x\\y\\z
    \end{pmatrix}+\beta_1 \begin{pmatrix}1\\0\\0\end{pmatrix} +\beta_2\begin{pmatrix}b\\b\\-a\end{pmatrix} \text{ with } \beta_1,\beta_2>0,\\
  \begin{pmatrix}
     \lambda\\\lambda \eta\\\lambda \eta^2
    \end{pmatrix}&=\begin{pmatrix}
    x\\y\\z
    \end{pmatrix}+\gamma_1 \begin{pmatrix}1\\0\\0\end{pmatrix} +\gamma_2\begin{pmatrix}d-b-1\\d-b\\a-c\end{pmatrix} \text{ with } \gamma_1,\gamma_2>0.
\end{align*}
 The first two conditions lead to $y=0$ and $z=\eta^2+\frac{a\cdot \eta}{b}.$ Inserting this into the last condition we can solve for $\gamma_2= \frac{\lambda\eta}{d-b}$ and obtain another equation for $z$:
\begin{align*}
    z = \lambda(\eta^2+\frac{c-a}{d-b}\eta)
  \,\,  &\Rightarrow\,\, \eta^2+\frac{a\cdot \eta}{b} = \lambda(\eta^2+\frac{c-a}{d-b}\eta)\\
    &\Rightarrow 0=(\lambda -1)\eta^2 + (\lambda\frac{c-a}{d-b}-\frac{a}{b})\eta .
\end{align*}
Since for the parameter condition $ad-bc=1$ and $d+c>a+b$, it follows that $d-b>0$ so we do not divide by zero. For the genericity of our point conditions, it is important that $\eta$ and $\lambda$ can be chosen independently of each other. This is not satisfied by the above equation for all $a,b,c,d$.  
Thus, the path $A-C\,|\,D-E-F$ is not possible for all polytopes in family 21.

The systems of equations that we obtain for $A-C-D\,|\, E-F$ and $A-C-D-E\,|\, F$ are solvable.  
So we conclude, that they are possible for all polytopes of family 21.

Next, we investigate the connected paths that arise from the $L$ intersecting the tropical surface 3 times. 
First we observe that the connected lattice path $A-C-D-F$ is not possible: Since $A^\vee\subset \{x<0,y<0,z<0\}$, $C^\vee\subset\{y>0\}$ and  $D^\vee\subset\{x>0\}$, it follows that $L$ passes through $D^\vee\cap\{x>0,y>0\}.$ To pass from $D^\vee$ directly to $F^\vee$, $L$ has to pass through $\mathbb{R}_{>0}\cdot(1,0,0)+\mathbb{R}_{>0}\cdot(c,-d,c) \subset \{y<0\}$. This is impossible. 
Taking into account that the path has to run along existing edges, we have the following list of remaining possible connected paths: 
$$
    A-B-C-F,\
    A-B-D-F,\
    A-C-E-F,\
    A-D-E-F.
$$

Solving for the equations produced by the paths $A-B-C-F$ and $A-B-D-F$ does not lead to a contradiction, so these paths are possible.

Consider the lattice path $A-C-E-F$. We obtain the following system of equations:
\begin{align*}
    \begin{pmatrix}
  0\\0\\0 
    \end{pmatrix}&= \begin{pmatrix}
    x \\ y\\z
    \end{pmatrix}+\alpha_1\begin{pmatrix}
      -1\\0\\0
    \end{pmatrix}+\alpha_2\begin{pmatrix}
      0\\0\\-1
    \end{pmatrix}\,\text{ with }\alpha_1,\alpha_2>0,\\
    \begin{pmatrix}
  1\\\eta\\\eta^2 
    \end{pmatrix}&= \begin{pmatrix}
    x \\ y\\z
    \end{pmatrix}+\beta_1\begin{pmatrix}
      b\\b\\-a
    \end{pmatrix}+\beta_2\begin{pmatrix}
      d-b-1\\d-b\\a-c
    \end{pmatrix}\,\text{ with }\beta_1,\beta_2>0,\\
    \begin{pmatrix}
  \lambda\\\lambda\eta\\\lambda\eta^2
    \end{pmatrix}&= \begin{pmatrix}
    x \\ y\\z
    \end{pmatrix}+\gamma_1\begin{pmatrix}
      1\\0\\0
    \end{pmatrix}+\gamma_2\begin{pmatrix}
      d-b-1\\d-b\\a-c
    \end{pmatrix}\,\text{ with }\gamma_1,\gamma_2>0.
\end{align*}
Solving for $\beta_2$, shows that $\beta_2>0$ if and only if  $\frac{a}{b}+\eta >\lambda(\frac{c-a}{d-b}+\eta)$. For any choice of $a,b,c,d$ parameters of the polytope family 21,
this implies an upper bound for the value of $\lambda$ which contradicts the generality of our chosen points. So, the path $A-C-E-F$ is not possible for any polytope in family 21.

Consider the path $A-D-E-F$.
We obtain the following system of equations:
\begin{align*}
    \begin{pmatrix}
  0\\0\\0 
    \end{pmatrix}&= \begin{pmatrix}
    x \\ y\\z
    \end{pmatrix}+\alpha_1\begin{pmatrix}
      0\\-1\\0
    \end{pmatrix}+\alpha_2\begin{pmatrix}
      0\\0\\-1
    \end{pmatrix}\,\text{ with }\alpha_1,\alpha_2>0,\\
    \begin{pmatrix}
  1\\\eta\\\eta^2 
    \end{pmatrix}&= \begin{pmatrix}
    x \\ y\\z
    \end{pmatrix}+\beta_1\begin{pmatrix}
      1 \\0\\0
    \end{pmatrix}+\beta_2\begin{pmatrix}
      b\\b\\-a
    \end{pmatrix}\,\text{ with }\beta_1,\beta_2>0,\\
    \begin{pmatrix}
  \lambda\\\lambda\eta\\\lambda\eta^2
    \end{pmatrix}&= \begin{pmatrix}
    x \\ y\\z
    \end{pmatrix}+\gamma_1\begin{pmatrix}
      1\\0\\0
    \end{pmatrix}+\gamma_2\begin{pmatrix}
      d-b-1\\d-b\\a-c
    \end{pmatrix}\,\text{ with }\gamma_1,\gamma_2>0.
\end{align*}
By using only the $y,z-$equations of the second and third condition we can solve for $\beta_2,\gamma_2,y,z$:
\begin{align*}
    y&= \lambda \eta - (\lambda-1)(d-b)(a\eta+b\eta^2), \\  z&=\lambda \eta^2 + (\lambda-1)(c-a)(a\eta+b\eta^2),\\
    \beta_2&= \frac{1}{a}(\lambda-1)(\eta^2 + (c-a)(a\eta+b\eta^2))=\frac{1}{b}(\lambda-1)(-\eta + (d-b)(a\eta+b\eta^2)),\\
    \gamma_2&= (\lambda -1)(a\eta+b\eta^2).
\end{align*}
Since $\alpha_1 =y$ we have to check whether $y>0$. However, since $d>b$, we have $$y>0 \,\,\Leftrightarrow\,\, \lambda < \frac{(b-d)(a+b\eta)}{1+(b-d)(a+b\eta)}.$$ This contradicts the generality of our point conditions. So the path $A-D-E-F$ is not possible for any polytope in family 21.
 \end{proof}

\section*{Conflicts of Interest}
On behalf of all authors, the corresponding author states that there is no conflict of interest.

\bibliographystyle{abbrv}
\bibliography{sample}

\end{document}